\documentclass[a4,12pt]{article} \usepackage{amssymb}
 \usepackage{amsthm}\usepackage{amsmath}\usepackage{latexsym}

\newtheorem{teo}{Theorem}[section]
\newtheorem{oss}[teo]{Remark}

\newtheorem{Prop}[teo]{Proposition}
\newtheorem{lemma}[teo]{Lemma}

\newtheorem{Defi}[teo]{Definition}
\newtheorem{es}[teo]{Example}
\newtheorem{corollario}[teo]{Corollary}
\newtheorem{no}[teo]{Notation}

\newcommand{\cc}{_{^{_\HH}}}
\newcommand{\vv}{_{^{_\VV}}}

\newcommand{\ccx}{_{^{_{\HH_x}}}}

\newcommand{\res}{\mathop{\hbox{\vrule height 7pt width .5pt depth 0pt
\vrule height .5pt width 6pt depth 0pt\,}}\nolimits}

\def \cont{{\mathbf{C}}}

\def \vd{v}
\def \DH{h}
\def \ore{\Tor{_{^{\!E}}}}

\def \blae{\nabla^{_{^{\!E}}}}

\newcommand{\LL}{\mathop{\hbox{\vrule height .5pt width 6pt depth
0pt \vrule height 7pt width .5pt depth 0pt\,}}\nolimits}
\newcommand{\rr}{_{^{_\mathcal{R}}}}
\newcommand{\sr}{_{^{_\mathcal{S\!R}}}}

\newcommand{\ck}{_{^{_{\HH_k}}}}
\newcommand{\ci}{_{^{_{\HH_i}}}}
\newcommand{\ctr}{_{^{_{\HH_3}}}}
\newcommand{\cd}{_{^{_{\HH_2}}}}
\newcommand{\cu}{_{^{_{\HH_1}}}}

\def \cin{{\mathbf{C}^{\infty}}}

\def\dim {\mathrm{dim}}
\def\dc {d\cc}

\def\dg{\textit{grad}\cc}

\def \per {\sigma^{n-1}\cc}

\newcommand{\ce}{_{^{_E}}}

\def\HC{\mathcal{H}^Q_{\mathbf{cc}}}

\def\SC{C}

\def \nn{\nu\cc}

\def \XH{\mathfrak{X}(\HH)}
\def \XX{\mathfrak{X}}

\def \P{{\mathcal{P}}}

\def \PH{\P\cc}

\def \Om{\Omega}

\def \ee{\mathrm{e}}

\def \R{\mathbb{R}}
\def \Rn{\mathbb{R}^{\DN}}

\def \GG{\mathbb{G}}
\def \gg{\mathfrak{g}}

\def\grr{{{\mathtt{gr}}}}

\def\gc{\nabla^{^{_{\HH}}}}

\def \Tor{{\textsc{T}}}

\def \RC{\textsc{R}}

\def \nn{\nu_{_{\!\HH}}}
\def \XG{\mathfrak{X}(\GG)}

\def \ee{\mathrm{e}}

\def \Om{\Omega}
\def \Rn{\mathbb{R}^{n}}
\def \R{\mathbb{R}}
\def \N{\mathbb{N}}

\def \cji {c_{j\,i}(x)}
\def \C { C(x):=[\cji]_{j,i},\,\, {j=1,\ldots,m \,,\, i=1,\ldots,n}}

\def \Qdim {Q:=\sum_{i=1}^{k}i\,\DH_i}

\def \X {X=(X_{1}, \ldots, X_{m_1})}

\def \X0 {X_{1}(0)\!=\!\partial_{x_{1}}, \ldots, X_{m_1}(0)\!=\!\partial_{x_{m_1}}}

\def \HG {\HH\GG}
\def \HS {\HH\!{S}}

\def \TG {\mathit{T}\GG}
\def \HH {\mathit{H}}
\def \UH {\mathit{UH}}

\def \VV {\mathit{V}}
\def \TT {\mathit{T}}
\def \TS {\mathit{T}S}

\def \grad{\textit{grad}}
\def \C0H{\mathbf{C}_{0}^{\infty}(U,\HG)}
\def \C00{\mathbf{C}_{0}^{\infty}(U)}
\def \C01{\mathbf{C}_{0}^{1}(U)}

\def \L1{d\,\mathcal{L}^n}

\def \H1{\mathcal{H}_{{\bf cc}}^{1}}

\def \exp{\textsl{exp\,}}

\def \llog{{\textsl{log}}_\gg}
\def \esp{{\textsl{exp}}_\gg}

\def \Om{\Omega}
\def \Rn{\mathbb{R}^{n}}
\def \R{\mathbb{R}}
\def \N{\mathbb{N}}

\def \cji {c_{j\,i}(x)}
\def \C { C(x):=[\cji]_{j,i},\,\, {j=1,\ldots,m \,,\, i=1,\ldots,n}}

\def \GG{\mathbb{G}}
\def \gg{\mathfrak{g}}

\def \X {X=(X_{1}, \ldots, X_{m_1})}

\def \X0 {X_{1}(0)\!=\!\partial_{x_{1}}, \ldots, X_{m_1}(0)\!=\!\partial_{x_{m_1}}}

\def \HG {\mathit{H}}

\def \C0H{\mathbf{C}_{0}^{\infty}(\Om,\HG)}
\def \C00{\mathbf{C}_{0}^{\infty}(\Om)}
\def \C01{\mathbf{C}_{0}^{1}(\Om)}

\def\GG{\mathbb{G}}

\title{CC-distance and
metric normal of smooth hypersurfaces in sub-Riemannian Carnot
groups}
\author{Nicola Arcozzi \thanks{F. F. is partially supported by GALA project Geometric Analysis in Lie groups and Applications,
supported by the European Commission within the 6th Framework Programme and by the PRIN project
Viscosity, metric and control theoretic methods in nonlinear partial differential equations, MIUR (Italy). } \and
Fausto Ferrari \thanks{F. F. is partially supported by GALA project Geometric Analysis in Lie groups and Applications,
supported by the European Commission within the 6th Framework Programme and by the PRIN project
Viscosity, metric and control theoretic methods in nonlinear partial differential equations, MIUR (Italy). } \and Francescopaolo Montefalcone
\thanks{F. M. is partially supported by University of Bologna,
Italy, founds for selected research topics and by GNAMPA of INdAM,
Italy.} }

\begin{document}
\maketitle

\begin{abstract}
In this paper we study the main geometric properties of the
Carnot-Carath\'eodory (abbreviated CC) distance $\dc$  in the
setting of $k$-step sub-Riemannian Carnot groups from many
different points of view; see Section \ref{onnorm} and Section
\ref{ipo}. An extensive study of the so-called {\it normal
CC-geodesics} is given. We state and prove some related
variational formulae and we find suitable Jacobi-type equations
for normal CC-geodesics; see Section \ref{ipo}. One of our main
results is a sub-Riemannian version of the Gauss Lemma; see
Section \ref{GLem}. We show the existence of the {\it metric
normal} for smooth non-characteristic hypersurfaces; see Corollary
\ref{important}. In Section \ref{ipoo} we compute the
sub-Riemannian exponential map $\exp\sr$ for the case of $2$-step
Carnot groups. Other features of normal CC-geodesics are then
studied. In Section \ref{appe} we show how the system of normal
CC-geodesic equations can be integrated step by step. Finally, in
Section \ref{us} we show a regularity property of the CC-distance
function $\delta\cc$ from a $\cont^k$-smooth hypersurface $S$; see
Theorem \ref{ccdreg}.
\end{abstract}

\tableofcontents

\section{\large Introduction}

In the last thirty years, many progresses have been done in
developing Analysis through very general geometries and metric
spaces. This trend in mathematical research, already initiated by
Federer's treatise \cite{FE}, has been pursued by many authors,
with different point of views: Ambrosio \cite{A2}, Ambrosio and
Kirchheim \cite{AK1, AK2}, Capogna, Danielli and Garofalo
\cite{CDG}, Cheeger \cite{Che}, Cheeger and Kleiner
\cite{Cheeger1}, David and Semmes \cite{DaSe}, De Giorgi
\cite{DG}, Gromov \cite{Gr1, Gr2}, Franchi, Gallot and Wheeden
\cite{FGW}, Franchi and Lanconelli \cite{FLanc},  Franchi,
Serapioni and Serra Cassano \cite{FSSC3, FSSC6}, Garofalo and
Nhieu \cite{GN}, Heinonen and Koskela \cite{HaKo},  Jerison
\cite{Jer}, Korany and Riemann \cite{KR}, Pansu \cite{P2} but, of
course, this list is not complete.

 {\it Sub-Riemannian} or {\it
Carnot-Carath\'eodory} geometries have become a research field of
great interest, also because of their wide connections with
different fields in Mathematics and Physics, such as PDE's,
Control Theory, Mechanics, Theoretical Computer Science.

 For references, comments
and some different perspectives on sub-Riemannian
geometry we refer the reader to Agrachev and Ghautier \cite{Ag1}, Agrachev and Sarychev \cite{Ag2},
Gromov  \cite{Gr2}, Montgomery
\cite{Montgomery},   Pansu \cite{P2, P4},  and Strichartz
\cite{Stric}, Vershik and Gershkovich
\cite{Ver}.

 Very recently, the
 so-called Visual Geometry has also received new impulses from this
 field; see \cite{CMS},  \cite{CM}  and references therein.

In this article we begin a study of fine properties of the
function ``sub-Riemannian distance from a hypersurface'' in Carnot
groups. Before giving a formal outline of the context and content
of the paper, we would like to briefly and colloquially discuss
what we think are the main features of this research.

The elementary analysis of the function ``distance from a closed
set'' $E$ with smooth boundary $\partial E$ in the Euclidean space
$\mathbb{R}^n$ is based on the intuitive idea of ``moving'' a
sufficiently small closed Euclidean ball $B$ in
$\mathbb{R}^n\setminus E$ until it touches $\partial E$ at some
point $P$.  Let $B_0$ be the ball $B$ in its final position and
let $O$ be its center. The Euclidean segment $OP$ is (part of) the
segment perpendicular to $\partial E$ at $P$. For points $Q$ in
$OP$, the distance from $Q$ to $E$ is realized by the segment
$OQ$. The idea of a "sufficiently small moving ball", however
intuitive, is not easy to deal with in calculations and it is not
really necessary. In fact, it is easier to work with the final
configuration: given a point $P$ on $\partial E$ and the direction
$\nu$ normal to $\partial E$ at $P$, we can recover $B_0$ as any
of the sufficiently small balls centered at a point $O$ on the
segment leaving $P$ in the direction $\nu$, having radius $OP$.
The normal direction $\nu$ is a first order differential object;
while the fact that the ball $B_0$ touches $E$ at $P$ only,
generally depends on the ``extrinsic'' curvatures of $\partial E$
at P, which is a second order differential object.

Exactly the same line of reasoning works if $E$ lives in a
Riemannian manifold. The direction $\nu$, normal to $\partial E$
at $P$, determines the geodesic $\gamma$ having speed $\nu$ at
$P$, the \it geodesic normal to $\partial E$ at $P$. \rm The ball
$B_0$ can then be recovered as a ball having center on $\gamma$,
sufficiently close to $P$.

The sub-Riemannian case presents new difficulties. Consider, for
expository reasons, the case of $\mathbb{R}^n$ endowed with a
sub-Riemannian structure of codimension $\vd\geq1$. In particular,
to each point $x\in\mathbb{R}^n$ is associated a linear space
$\HH_x\subset\TT_x\Rn$ of admissible {\it horizontal directions},
$\mbox{dim}(\HH_x):=\DH=n-v$, and a metric $g\ccx$ to measure
lengths of vectors in $\HH_x$. We stress that the metric $g\cc$
can also be defined (non-canonically) as the restriction of a
 Riemannian  metric $g$ defined on the whole tangent bundle
$\TT\Rn$.

We can use this structure to define lengths of {\it admissible}
curves (those whose velocity belong to $\HH$) and, under suitable
hypotheses,  this leads to a new distance $\dc$ in $\mathbb{R}^n$,
called Carnot-Carath\'eodory (abbreviated CC) distance. This
distance $\dc$ is a length-metric, realized by geodesics In this
introduction, we assume they are smooth, although the regularity
of the so-called CC-geodesics is, in general, an open problem; see
Section \ref{onnorm} and Section \ref{ipoo}.

 We can mentally
repeat the previous procedure, moving a small metric ball $B$
until it touches $\partial E$ at $x$. We assume that $x$ is {\it
non-characteristic}, i.e. roughly speaking, the linear space
$\HH_x$ is not totally contained in $\TT_x\partial E$, the tangent
space to $\partial E$ at $x$.  In particular, we can determine a
direction $\nn(x)$ in $\HH_x$ (the {\it unit horizontal normal to
$\partial E$ at $x$}) which is orthogonal with respect to $g\ccx$ to all
directions  in $\HH_x\cap \TT_x\partial E$. The novelty is that
there are {\it infinitely many} CC-geodesics $\gamma$ passing
through $x$ and having tangent vector $\nn(x)$ there. In fact,
there is a $\vd$-parameter family of them, where $\vd={\rm
codim}\HH$. Hence, we can not identify the final ball $B_0$ based
on the information contained in $\nn$ alone.

However, in several cases there exists a unique {\it metric
normal}  $\gamma_{\mathcal{N}}$ to $\partial E$ at $x$. By this we
mean the (maximal) CC-geodesic curve $\gamma_{\mathcal{N}}$
starting from $x$, with the property that, for each $y$ in
$\gamma_{\mathcal{N}}\setminus E$, $\dc(x, y)$ realizes the
distance between $y$ and $E$. i.e., the metric ball centered at
$y$ and having radius $\dc(x, y)$ touches $\partial E$ at $x$. The
notion of metric normal was introduced by the first two authors in
the context of the (lowest dimensional) Heisenberg group
$\mathbb{H}^1$; see \cite{AF}.

 In order to specify the metric normal among the CC-geodesics
 tangent to $\nn$ at $x\in\partial E$ we need  $i$ real parameters. In the Heisenberg
case, $i=1$, the extra parameter could be read off from the
parametrization of $\partial E$. This fact is crucial, because
from the equation for $\partial E$ one can explicitly write down
the ``exponential map'' associated with $\partial E$, and, for
instance, deduce regularity properties of the function ``distance
from $E$''.

In \cite{AF2} it was observed that, in the Heisenberg case, the
extra parameter could be read as a ``imaginary curvature'' for
$\partial E.$ At the same time and independently, the third author
\cite{Monteb}, continuing his work \cite{Monte} on variational
formulas in sub-Riemannian Carnot groups, associated to each
non-characteristic point $x\in\partial E$ a ``vertical vector''
$\varpi$ of parameters, which reduces to the imaginary curvature
in the Heisenberg case.  These parameters, which can be
interpreted as ``curvatures'' (or Lagrangian multipliers; see
Section \ref{ipo}), have no immediate counterpart in Riemannian
geometry and play an important role in the geometric analysis of
$\partial E$ in this sub-Riemannian setting. An interesting
feature of these ``curvatures'', which appears in seminal form in
\cite{AF} and in much greater generality in \cite{Monte, Monteb},
is they can best be computed in terms of the {\it Riemannian
gradient} (with respect to the fixed Riemannian metric $g$ on the base
manifold) of the smooth function $f$ defining $\partial E$ locally
(i.e., $\partial E=\{f=0\}$, locally). Roughly speaking $\varpi$
can be regarded as the ``non-horizontal'' (more precisely,
vertical; see Section \ref{prelcar}) part of the Riemannian
gradient of $f$; see Definition \ref{carca}.

 Researchers working with ``first order''
geometric measure theory of hypersurfaces in Carnot group
(isoperimetric inequalities, Lipschitz submanifolds, functions of
$\HH$-bounded variation, etc.) are generally concerned with the
\it horizontal gradient \rm of the given defining function only.

It is in fact an accepted rule-of-thumb that horizontal vectors
contain all the first order geometric information of the group,
and that this extends to some second order objects (e.g.,
$\HH$-mean curvature and minimal surfaces). When dealing with such
``second order information'' as that encoded by the distance
function, the ``forgotten component'' of the Riemannian gradient
has to be taken into account. In this paper, we bring together
these two different perspectives, relating these ``curvatures''
$\varpi$ to the metric normal $\gamma_\mathcal{N}$, hence to the
distance function from a smooth hypersurface.

 This way, among
others things, we shall extend to Carnot groups many results of
\cite{AF} and \cite{AF2}, but with a somewhat different
approach.\\

We now give a less informal view of the paper and of its context.
\begin{Defi}[Metric Normal]
Let $(X,d)$ be a metric space and let $E$ be closed in $X$. The
{\it metric normal} to $E$ at $x\in E$ is the set
$\gamma_{\mathcal N}:=\{y\in X:\ d(x, y)=d(y,E)\}$, where
$d(y,E):=\inf_{z\in E}d(z,y)$.\end{Defi} In this paper, the metric
space $(X,d)$ is a {\it Carnot group}  $\GG$ endowed with its
``natural'' CC-distance $\dc$.

 A $k$-step Carnot group $\GG$ is a Lie
group, whose Lie algebra $\mathfrak{g}$ admits a stratification
$\mathfrak{g}=\HH\oplus \VV$, where $\VV:=\oplus_{j=2}^k \HH_k$
$\HH_1:=\HH,\ \HH_{j+1}=[\HH_j,\HH_1]$ and $\HH_{k+1}=\{0\}$. In
$\HH$ we find the so-called horizontal vectors, while in $V$ we
find the vertical vectors. \begin{no}[Projections]Throughout this
paper, the mappings $\P\ci:\TG\longrightarrow\HH_i$,
$\P\vv:\TG\longrightarrow\VV$, will denote the projection operator
onto the subbundles $\HH_i\,\,(i=1,...,k)$ and $\VV$,
respectively.
\end{no}

Remind that, as for any Lie group, an internal group of
translations, called ``left-translations'', is given on $\GG$. In
addiction, any Carnot group $\GG$ has intrinsic dilations  making
it a homogeneous group; see Section \ref{prelcar}. In general, we
shall say that a property in $\mathbb{G}$ is intrinsic whenever it
is invariant with respect to these transformations.

The {\it horizontal space} $\HH$ is endowed with a Riemannian
metric $g\cc$ which is used in the obvious way to define lengths
of {\it horizontal curves} $\gamma$ in $\mathbb{G}$ (i.e.
$\gamma:\R\longrightarrow\GG$ is absolutely continuous and
$\dot{\gamma}(t)\in \HH$ for a.e. $t\in\R$). The distance between
points in $\mathbb{G}$ is the infimum of the lengths of curves
joining them, and it can be proved that such distance is realized
by the length of (not necessarily unique) CC-geodesic curves; see
Section \ref{onnorm} and Section \ref{ipo} for a precise
definition of CC-geodesic.

 We stress that is a difficult open problem showing, or
disproving, that CC-geodesics are smooth for all Carnot groups;
see \cite{Montgomery, Montgomery2}.

In the Euclidean case, by assuming reasonable hypotheses on
$\partial E$, the metric normal at a point $x\in\partial E$ can be
identified with the normal direction associated with the surface
at $x\in\partial E,$ because it is a segment; see \cite{FE, FE1}.
As a consequence, whenever we recall the normal direction, we
emphasize the linear aspect of the metric normal, forgetting that
it is, first of all, a path.  In the Riemannian  setting it is
known that, if we stay near to a smooth hypersurface $S$  there
exists one, and only one, unit vector belonging to the tangent
space to the base Riemannian  manifold that select the metric
normal to the hypersurface $S$ at $x$. As already said, in the
sub-Riemannian setting this problem it is not so straightforward.
Indeed, just to fix the ideas in the simplest Heisenberg group
$\mathbb{H}^1$ (see \cite{Montilincei} for further details), even
when we fix a non-trivial horizontal vector $\nn(x)\in\HH_x$,
where\footnote{According to the notation used in \cite{AF, AF2},
in this introduction we adopt the following convention: every
point $x\in \mathbb{H}^1$ is given, using exponential coordinates,
by the triple $x\equiv[x_1, y_1, t_1]$. Moreover, we use, as a
vector basis $\{X, Y\}$ for the horizontal space $\HH$, the
following vectors field: $X(x):=\partial_{x_1} + 2
y_1\partial_{t_1}$, $Y(x):=\partial_{y_1} - 2 x_1\partial_{t_1}$.
This convention for $\mathbb{H}^1$  will be slightly changed in
the sequel.} $\HH={\rm span}_\R\{X, Y\}$, infinite different
CC-geodesics starting from $x\in \mathbb{H}^1$ exist with the same
initial velocity $\nn(x).$ In \cite{AF} this problem was solved by
remarking that a path that locally parameterizes the metric normal
of a smooth surface S at $x\in S$ (whenever $X$ is
non-characteristic) can be selected just by considering two
intrinsic object associated with the surface $S$ at $x$, i.e. the
intrinsic unit normal vector $\nn(x)$ and the so-called
``imaginary curvature'' of $S$ at $x$; see \cite{AF}. Roughly
saying, if $S=\{y\in \mathbb{H}^1: f(y)=0\},$ with $f\in
C^2(\mathbb{H}^1)$ and $x\equiv[x_1, y_1, t_1]\in S$ is
non-characteristic, i.e. $X
f(x)\equiv\partial_{x_1}f(x)+2y_1\partial_{t_1}f(x)\neq 0$ or $Y
f(x)\equiv\partial_{y_1}f(x)-2x_1\partial_{t_1}f(x)\neq 0,$ where
$x\equiv[x_1,y_1,t_1],$ then
$$\nn(x)=\frac{Xf(x)X(x)+Yf(x)Y(x)}{\sqrt{(Xf(x))^2+(Yf(x))^2}}
,\:\:\varpi(x)=\frac{[X,Y]f(x)}{\sqrt{(Xf(x))^2+(Yf(x))^2}},$$
(which are, respectively, the intrinsic unit normal along $S$ at
$x$ and the imaginary curvature of $S$ at $x$) select an
horizontal path that is a subset of the metric normal
$\gamma_\mathcal{N}$ to $S$ at $x$.

The problem we trait in this paper concerns the generalization of
this program to any Carnot group. Later on we shall describe our
approach.

Assume that $S$ is the boundary of an open bounded subset $\Omega$
of a Carnot group $\mathbb{G}.$ Now let us assume that the set
$\Omega$ satisfies the so-called property of the internal ball.
Namely, any point $x\in S$ can be touched by a closed metric ball
$B(y,r)\subset \Omega$ and $\partial B(y,r)\cap S=\{x\}.$ In this
case $\partial B(y,r)$ and $S$ have the same tangent space at $x.$
The objects we are searching for, whenever they exist, are somehow
``hidden'' in the invariants associated with this tangent space.
On the other hand the CC-ball $B(y,r)$ is the union of all the
horizontal paths of length less than $r,$ starting from $y.$  This
flux of CC-geodesics is usually determined by the solutions of a
suitable Hamiltonian system; see Section \ref{onnorm}.

 At this point we have to
select the path $\gamma:[0,r]\longrightarrow\GG$, $\gamma(0)=y$,
$\gamma(r)=x$, connecting $y$ to $x$ and  satisfying the property
$\dc(x,\gamma(t))=\inf _{z\in \partial B(y,r)}\dc(\gamma(t), z).$

We will need two key facts: (i) the CC-distance from a point
satisfies the eikonal equation (see \cite{MoSC}), i.e.
$|\grad\cc\dc|=1$, at each regular point of $\dc$; (ii) even in
this sub-Riemannian structure, a Gauss-type lemma holds (see
Section \ref{GLem}), namely if $\gamma$ is a normal CC-geodesic
leaving the center of the ball, then
$\dot{\gamma}(t)=\grad\cc\dc(\gamma(t)).$  We stress that the last
identity is a straightforward consequence of
 the eikonal equation.

The remaining part of the proof can be described as follows.

Let us introduce a left-invariant frame
$\underline{X}:=\{X_1,\dots,X_\DH, X_{\DH+1},\dots,X_n\}$ for
$\TG$ where $\HH={\rm span}_\R\{X_1,\dots,X_\DH\}$ and $\VV={\rm
span}_\R\{X_{\DH+1},\dots,X_n\}.$ Moreover we fix a Riemannian
metric $g(\cdot, \cdot):=\langle\cdot, \cdot\rangle$ on $\TG$ that
makes orthonormal the frame $\underline{X}$. We shall assume that
the restriction of $g$ to $\HH$ equals the sub-Riemannnian metric
$g\cc$, i.e. $g|_\HH=g\cc$.

Set $P_l(t):=(X_l\dc)(\gamma(t))\,\,(l=1,...,n),$ where
$\dc(\cdot):=\dc(\cdot,y)$.

 In particular, one has
$$
\dot{P}_l(t)=\langle\grad\,
(X_l\dc)(\gamma(t)),\dot{\gamma}(t)\rangle\qquad(l=1,...,n).
$$
By the Gauss' lemma (i.e.
$\dot{\gamma}(t)=\grad\cc\dc(\gamma(t))$), we get
$$
\dot{P}_l(t)=\sum_{j=1}^{\DH} (X_jX_l)(\dc(\gamma(t)))P_j(t).
$$
On the other hand $X_jX_l=[X_j,X_l]+X_lX_j,$ thus
$$
\dot{P}_l(t)=\sum_{j=1}^{\DH}
(X_lX_j)(\dc(\gamma(t)))P_j(t)+\sum_{j=1}^{\DH}
[X_j,X_l](\dc(\gamma(t)))P_j(t).
$$
 Here plays a role the eikonal equation, namely $\sum_{j=1}^{\DH} (X_j\dc)^2(\gamma(t))=1,$ because
\begin{eqnarray*}
\dot{P}_l(t)&=&\frac{1}{2}X_l\left(\sum_{j=1}^{\DH}
(X_j\dc)^2(\gamma(t))P_j(t)\right)+\sum_{j=1}^{\DH}
[X_j,X_l](\dc(\gamma(t)))P_j(t)\\&=&\sum_{j=1}^{\DH}
[X_j,X_l](\dc(\gamma(t)))P_j(t). \end{eqnarray*}It turns out that
$[X_j,X_l](\dc(\gamma(t)))=\sum_{\alpha=\DH+1}^{n}\SC_{j l}^\alpha
X_\alpha(\dc(\gamma(t))),$ where the coefficients $\SC_{j
l}^\alpha$ are the {\it structural constants} of the Lie algebra
$\gg$; see Section \ref{prelcar}.

This way,  we obtain the following system of O.D.E.'s:
\begin{equation}\label{systemprime}
\dot{P}_l(t)=\sum_{j=1}^{\DH}\sum_{\alpha=\DH+1}^n
C^\alpha_{jl}\,P_j(t)P_\alpha(t)\qquad (l=1,...,n).
\end{equation}
By using the properties of the Carnot structural constants (in
particular, see \eqref{chypc}), this system can be solved step by
step. A similar algorithm will be discussed in the Appendix of
Section \ref{vipo}.

The previous discussion, for which we refer the reader to Section
\ref{GLem}, yields the next:

\begin{lemma}Let $\gamma:[0,r]\longrightarrow\GG\,(r>0)$ be any
normal CC-geodesic \footnote{See Section \ref{onnorm} for a
precise definition of {\it normal CC-geodesic}.} of unit-speed and
parameterized by arc-length. Let $\gamma(0)=y$, $P(0)={P}_0$ be
its initial data and set $\dc(x)=\dc(x, y)\,(x\in\GG)$. Then we
have\begin{itemize}\item[(i)]$\dg\dc(\gamma(t))=P\cc(t)$ for every
$t\in[0,r]$;\item[(ii)]$\grad\vv\dc(\gamma(t))=P\vv(t)$ for every
$t\in[0,r]$;\item[(iii)]$\dot{P}_l(t)=\sum_{j=1}^{\DH}\sum_{\alpha=\DH+1}^n
C^\alpha_{jl}\,P_j(t)P_\alpha(t)$ for every $t\in[0,r]\qquad
(l=1,...,n)$.\end{itemize}
\end{lemma}

This lemma solves the problem of selecting, for any regular point
$x$ belonging to the CC-sphere $\mathbb{S}^n\sr(y,r):=\partial
B(y, r)$, the unique normal CC-geodesic having velocity vector
equal to the horizontal normal direction at that point (i.e.
$P\cc(0)=\nn(x)$) and connecting this point to the center $y$ of
the CC-sphere. In fact, by uniqueness of solutions of O.D.E.'s
systems, one gets that the desired curve must be the normal
CC-geodesic defined by
$$\gamma(t):=\exp\sr(y,-\mathcal{N}(x))(t)\qquad
t\in[0,r].$$Here we have set
$$\mathcal{N}:=\frac{\nu}{|\PH\nu|}=(\nn,\varpi),$$ where $\nu$ denotes the Riemannian unit normal along the CC-sphere
$\mathbb{S}^n\sr(y,r)$, thus $\nn$ is just the horizontal unit
normal along $\mathbb{S}^n\sr(y,r)$ and
$\varpi=\frac{\P\vv\nu}{|\PH\nu|}$. Moreover, $\exp\sr$ denote the
sub-Riemannian exponential map; see Section \ref{ipoo}.

An immediate ``geometric'' consequence can be given:
\begin{corollario}
Let $S=\{x\in\mathbb{G}:\:\:f(x)=0 \},$ where $f$ is a $C^2$
function. Assume that there exists a CC-ball $B(y,r)$ with center
$y$ and radius $r,$ such that $B(y,r)\subset \{f(x)<0 \}$, or
$B(y,r)\subset \{f(x)>0 \}$, and ${B(y,r)}\cap S=\{x\}$, where
$x\in S$ is non-characteristic. Then there exists the metric
normal $\gamma_{\mathcal{N}}$ to $S$ at $x$ and  for every
$t\in[0,r],$ $\gamma(t)\in \gamma_{\mathcal{N}},$ where
$$\gamma(t):=\exp\sr(y,-\mathcal{N}(x))(t)\qquad
t\in[0,r].$$\end{corollario}

Till now we have described some of the main results of this paper,
but there are many others aspects  and related questions. Here we
briefly give an account of the rest of the paper.

In Section \ref{prelcar} we shall introduce the main ingredients
to deal with Carnot groups and their sub-Riemannian structures:
Lie algebraic preliminaries, Carnot dilations, CC-distance,
properties of the Carnot structural constants and the notion of
curvature of a distribution. Moreover we shall discuss some other
tools  as,  for instance, the Levi-Civita connection, the
so-called $\HH$-connection (or horizontal connection and, the
related notion of $E$-connection) and covariant differentiation
along curves. We also give some basic examples.

In Section \ref{dere} we will just recall notation and some basic
definitions to work with hypersurfaces of Carnot groups, such as
the notion of  characteristic point and that of horizontal
perimeter.

In Section \ref{vipo} we begin our study of Carnot-Carath\'eodory
metrics and their geodesics.

In Section \ref{onnorm} we will introduce CC-geodesics from the
so-called {\it Hamiltonian} point of view, well-known in
literature; see \cite{Bellaiche}, \cite{Montgomery}, \cite{GoKA}.
We will give the equations for normal and abnormal curves (and
minimizers) and discuss some examples. Here we just remind the
so-called {\it normal CC-geodesic equations}:

\begin{eqnarray*}\mbox{(Normal Equations)}\left\{\begin{array}{ll}\dot{x}
=P\cc\\\dot{P}=- C(P\vv) P\cc,\end{array}\right.\end{eqnarray*}
where $P=(P\cc, P\vv)$ is the $n$-vector of the {\it momentum
functions} associated with a fixed orthonormal (left-invariant)
moving frame $\underline{X}=\{X_1,...,X_n\}$ for $\TG$.
Furthermore, $C(P\vv)\in\mathcal{M}_n$ is a linear combination of
$n\times n$ constant matrices which only depends on the Carnot
structural constants. More precisely, one
has$$C(P\vv):=\sum_{\alpha=h+1}^n P_\alpha C^\alpha,$$where
$C^\alpha:=[C^\alpha_{ij}]_{i, j=1,...,n}\,\,(\alpha=h+1,...,n)$
and, by definition, $C^\alpha_{ij}:=\langle[X_i, X_j],
X_\alpha\rangle$.\\

In Section \ref{ipo} we shall discuss another ``natural'' point of
view on this subject: the {\it Lagrangian} one. This is because it
seems the more natural way to obtain some additional information
about normal CC-geodesics and their minimizing features. In
particular, we will derive both the first and the second variation
of the natural sub-Riemannian Lagrangian. For precise statements
see Proposition \ref{1varg} and Theorem \ref{sevate}. Then, after
introducing the notion of {\it CC-geodesic variation}, we shall
prove the validity of another interesting second variation
formula; see Corollary \ref{2sevate}. Starting from these
formulae, we will deduce the natural {\it Jacobi equations} for
normal CC-geodesics; see Definition \ref{JE}. Furthermore, we will
discuss another Jacobi-type system of O.D.E.'s, which is obtained
under suitable assumptions on the variations. These ``restricted''
Jacobi-type equations read as follows:
\[\nabla_t^{(2)}{Y}+\RC( P\cc,Y)P\cc + C(P\vv)\nabla_tY=0,\]where
$\nabla_t$ denotes ``covariant differentiation''  and  $\RC$
denotes the Riemannian curvature tensor. This material can be
useful in the study of the conjugate and cut loci of a point, in
this context.

In Section \ref{ipoo} we shall define the {\it sub-Riemannian
exponential map} $\exp\sr$ and we will show some of their basic
features.

In Section \ref{GLem} we will prove a sub-Riemannian version of
the Gauss Lemma; see Proposition \ref{PODER} and Corollary
\ref{important}.

In Section \ref{2-stepgeod} we will compute the sub-Riemannian
exponential map $\exp\sr$ for the special case of
 $2$-step Carnot groups. Remind that the
sub-Riemannian exponential map based at the point $x_0\in\GG$ is a
mapping
$$\exp\sr(x_0,\cdot)(\cdot):\UH\times\HH_2\times\R\longrightarrow\GG,$$where
$\UH=\{X\in\HH: |X|=1\}$ denotes the bundle of all horizontal unit
vectors. We will show that
\begin{eqnarray*}\label{explicit}\exp\sr(x_0,P_0)(t):=x_0 + \int_0^te^{-C\cc(P\cd)s}P\cc(0)\,d
s-\frac{1}{2} \sum_{\alpha=\DH+1}^n\bigg(\int_{0}^t \langle
C^\alpha\cc
x\cc,\dot{x}\cc\rangle\,ds\bigg)\,\ee_\alpha,\end{eqnarray*}where
$P_0=(P\cc(0), P\cd)\in\UH\times\HH_2$,
$C\cc^\alpha:=[C^\alpha_{ij}]_{i,
j=1,...,\DH}\in\mathcal{M}_{\DH}\,\,(\alpha=h+1,...,n)$ and
$C\cc(P\cd):=\sum_{\alpha=\DH+1}^nP_\alpha C\cc^\alpha$. Moreover
\[x\cc(t):=x\cc(0) +
\int_0^te^{-C\cc(P\cd)s}P\cc(0)\,d s.\]In the previous formula we
have used the notation for the exponential of a linear operator.
More precisely, $e^{-C\cc(P\cd)s}$ denotes the exponential of a
square-matrix, i.e. the $\DH\times\DH$-matrix defined by:
$$e^{-C\cc(P\cd)s}:={\rm Id}_\DH -C\cc(P\cd)s +
\frac{[C\cc(P\cd)s]^2}{2!}-\frac{[C\cc(P\cd)s]^3}{3!}+\ldots.$$We
shall analyze the existence of $T$-periodic solutions for the
``auxiliary'' horizontal path $x\cc(t)=(x_1(t),...,x_{\DH}(t))$
previously defined. Indeed, the $T$-periodicity  of $x\cc(t)$ is
somehow connected with the study of the conjugate and cut loci of
a point; see Remark \ref{k}.

In Section \ref{appe} we will show how, at least in principle,the
system of normal CC-geodesic equations can be integrated step by
step.

In Section \ref{us} we will apply some of the tools previously
developed toward the study of the CC-distance function $\delta\cc$
from a $\cont^k$-smooth $(k\geq 2)$ hypersurface $S$, i.e.
\[\delta\cc(x):=\inf_{y\in S}\dc(x,y).\] This will be done only for
$2$-step Carnot groups, by using  the explicit structure of
$\exp\sr$ in this case. More precisely, we shall define a mapping
$\Phi:S\times]-\epsilon,
\epsilon[\longrightarrow\GG\,(\epsilon>0)$ by
\[\Phi(y,t):=\exp\sr(y,\,\mathcal{N}(y))(t),\]and then we will compute its Jacobian; see Lemma
\ref{CINV}. To be more precise, let us fix Riemannian normal
coordinates $(u_1,...,u_{n-1})$ around $y_0\in S$. Then
$\mathcal{V}(y):=\frac{\partial y}{\partial
u_1}\wedge...\wedge\frac{\partial y}{\partial u_{n-1}}$ is a
normal (non-unit) vector along $S$, in a neighborhood of $y_0\in
S$ and it turns out that
\[\big|\det\big[\mathcal{J}_{(y,0)}\Phi\big]\big|=|\PH\mathcal{V}|,\]where
$\mathcal{J}_{(y,0)}\Phi$ denotes the Jacobian matrix operator at
$(y,0)\in S\times ]-\epsilon, \epsilon[$. Therefore, out of the
characteristic set $C_S$, the map $\Phi$ turns out to be
invertible.

Finally, we shall prove the following (see Theorem \ref{ccdreg}):
\begin{teo}Let $\GG$ be a $2$-step Carnot group. Let $S\subset\GG$ be
 a $\cont^k$-smooth hypersurface with $k\geq 2$ and let
$\delta\cc$ denote the CC-distance function for $S$. Set $S_0:=S\setminus C_S$,
where $C_S$ denote the characteristic set of $S$. Then, for every
open set $U_0$ compactly contained in $S_0$, there exists a
neighborhood $U\subset\GG$ of $U_0$ having the {\bf unique nearest
point property}, with respect to the CC-distance. Furthermore, the
CC-distance function from $U_0\cap S$ is $\delta\cc|_{U\setminus
U_0}$ is a $\cont^k$-smooth function.
\end{teo}

The proof of this result is based on explicit computations and on
the sub-Riemannian Gauss Lemma (see Proposition
\ref{PODER}).\\

\subsection{Sub-Riemannian Geometry of Carnot
groups}\label{prelcar} In this section, we will introduce the
definitions and the main features concerning the sub-Riemannian
geometry of Carnot groups. References for this subject are, for
instance, \cite{CDG}, \cite{GE1}, \cite{GN}, \cite{Gr1, Gr2},
\cite{Mag, Mag2}, \cite{Mi}, \cite{Montgomery}, \cite{P2, P4},
\cite{Stric}. First, let us consider a ${\cin}$-smooth connected
$n$-dimensional manifold $N$ and let $\HH\subset \TT N$ be a
$h_1$-dimensional smooth subbundle of $\TT N$. For any $p\in N$,
let $\TT^{k}_p$ denote the vector subspace of $\TT_p N$ spanned by
a local basis of smooth vector fields $X_{1}(p),...,X_{h_1}(p)$
for $\HH$ around $p$, together with all commutators of these
vector fields of order $\leq k$. The subbundle $\HH$ is called
{\it generic} if for all $p\in N$ $\dim \TT^{k}_p$ is independent
of the point $p$ and {\it horizontal} if $\TT^{k}_p = \TT N$ for
some $k\in \N$. The pair $(N,\HH)$ is a {\it $k$-step CC-space} if
is generic and horizontal and if $k:=\inf\{r: \TT^{r}_p = \TT N
\}$. In this case, we have that
\begin{equation}\label{filtration}0=\TT^{0}\subset
\HH=\TT^{\,1}\subset\TT^{2}\subset...\subset \TT^{k}=\TT
N\end{equation} is a strictly increasing filtration of smooth
subbundles of constant dimensions $n_i:=\dim \TT^{i}
 \\(i=1,...,{k}).$ Setting $(\HH_i)_p:=\TT^i_p\setminus \TT^{i-1}_p,$
then $\grr(\TT_p N):=\oplus_{i=1}^k (\HH_k)_p$ is the associated
{\it graded Lie algebra}, at the point $p\in N$, with Lie product
induced by $[\cdot,\cdot]$. Moreover, we shall set $\DH_i:=\dim
{\HH}_{i}=n_i-n_{i-1}\,(n_0=\DH_0=0)$. The $k$-vector
$(\DH_1,...,\DH_{k})$ is the {\it growth vector} of $\HH$. Notice
that every ${\HH}_{i}$ is a smooth subbundle of the tangent bundle
$\pi:\TT N \longrightarrow N$, i.e. $\pi_{\ci}:
{\HH}_{i}\longrightarrow N$, where
$\pi_{\ci}=\pi_{|{\HH}_{i}}\,(i=1,...,{k})$.

\begin{Defi}\label{eman}We will call {\bf graded frame} $\underline{X}=\{X_1,...,X_n\}$ for $N$,
 any frame for $N$ such that, for any $p\in N$ we have that $\{{X}_{i_j}(p): n_{j-1}<i_j\leq n_j\},$ is a
basis for ${\HH_j}|_p:=\HH_j(p)\,\,(j=1,...,k)$.\end{Defi}

\begin{Defi}\label{dccar} A {\bf sub-Riemannian metric} $g\cc=\langle\cdot,\cdot\rangle\cc$ on $N$ is a
symmetric positive bilinear form on $\HH$. If $(N,\HH)$ is a
{CC}-space, then the {\bf {CC}-distance} $\dc(p,q)$ between $p,\,
q\in N$ is
$$\dc(p,q):=\inf \int\sqrt{\langle
 \dot{\gamma},\dot{\gamma}\rangle\cc} dt,$$
where the infimum is taken over all piecewise-smooth horizontal
paths $\gamma$ joining $p$ to $q$.
\end{Defi}

In fact, Chow's Theorem (see \cite{Gr1}, \cite{Montgomery})
implies that $\dc$ is actually a metric on $N$, since any two
points can be joined with (at least one) horizontal path; moreover
the topology induced by the {CC}-metric turns out to be compatible
with the given topology of $N$.

The general setting introduced above is the starting point of
sub-Riemannian geometry. A nice and very large class of examples
of these geometries is represented by {\it Carnot groups} which
for many reasons play, in sub-Riemannian geometry, an analogous
role to that of Euclidean spaces in Riemannian geometry. Below we
will introduce their main features. For an introduction to the
following topics, we suggest Helgason's book, \cite{Helgason}, and
the survey paper by Milnor, \cite{3}, regarding the geometry of
Lie groups, and Gromov, \cite{Gr1}, Pansu, \cite{P2, P4}, and
Montgomery, \cite{Montgomery}, specifically for sub-Riemannian
geometry.

By definition a $k$-{\it{step Carnot group}}  $(\GG,\bullet)$ is a
$n$-dimensional, connected, simply connected, nilpotent and
stratified Lie group (with respect to\,the group law $\bullet$). This means
that its Lie algebra $\gg\cong\Rn$
satisfies:\begin{equation}\label{alg}
{\mathfrak{g}}={\HH}_1\oplus...\oplus {\HH}_k,\quad
 [{\HH}_1,{\HH}_{i-1}]={\HH}_{i}\quad(i=2,...,k),\,\,\,
 {\HH}_{k+1}=\{0\},\end{equation}
We shall denote by $0$ the identity on $\GG$. Any ${x}\in\GG$
defines smooth maps $L_{x}, R_{x}:\GG\longrightarrow\GG,$ called
left-translation and right-translation, respectively, by $
L_{x}({y}):={x}\bullet {y},$ $ R_{x}({y}):={y}\bullet {x},$ for
any ${y}\in\GG$. Remind that the Lie algebra $\gg$ is naturally
isomorphic to $\TT_0\GG$ by identifying any left-invariant vector
field $X$ with its value at $0$. The isomorphism is explicitly
given by
${L_{x}}_{\ast}:\mathit{T}_0\GG\longrightarrow\mathit{T}_{x}\GG$.
The smooth subbundle
  ${\HH}_1$ of the tangent bundle $\TG$ is said
{\it horizontal}  and henceforth denoted by $\HH$. We will set
${\VV}:={\HH}_2\oplus...\oplus {\HH}_k$ and call ${\VV}$ the {\it
vertical subbundle} of $\TG$. We shall set $\vd:=\dim \VV$. One
has \[\vd:={\rm{codim}} \HH,\qquad n=\DH+\vd.\]As before, we will
assume that $\dim{{\HH}_i}=\DH_i$ $(i=1,...,k)$ and that $\HH$ is
generated by some basis of left-invariant horizontal vector fields
 $\underline{X\cc}:=\{X_1,...,X_{\DH_1}\}$. This one can be completed to a  global
 basis (frame) of left-invariant sections of $\TG$,
 $\underline{X}:=\{X_1,...,X_n\}$,  which is {\it graded} or
 {\it adapted to the stratification}. This can be done by re-labelling the canonical basis $\{\ee_i:i=1,...,n\}$ of
$\gg\cong\Rn$ in a way that it turns out to be adapted to the
stratification and then by setting
$${X_i}(x):={L_x}_\ast\ee_i=\frac{\partial x\bullet y}{\partial y}\Big|_{y=0}\ee_i\qquad(i=1,...,n).$$We shall set
 $n_l:=\DH_1+...+\DH_l$ ($n_0=\DH_0:=0,\,n_k=n$) and
 ${\HH}_l={\mathrm{span}}_\R\big\{X_i:  n_{l-1}< i \leq
 n_{l}\big\}$ $(l=1,...,k)$.

\begin{no}We shall set $I\cc:=\{1,...,\DH_1\}$,
$I\cd:=\{\DH_1+1,...,n_2(=\DH_1+\DH_2)\}$,...,
$I\ck:=\{n_{k-1}+1,...,n_k(=n)\}$, and $I\vv:=\{\DH_1+1,...,n\}$.
Moreover, we will use Latin letters $i, j, k,...,$ for indices
belonging to $I\cc$ and Greek letters
 $\alpha, \beta, \gamma,...,$ for indices belonging to $I\vv$. Unless otherwise specified,
capital Latin letters $I, J,
 K,...,$ may denote any generic index. We also define the
  function $\mathrm{ord}:\{1,...,n\}\longrightarrow\{1,...,k\}$
 by $\mathrm{ord}(I):= i $ if, and only if, $n_{i-1}<I\leq n_{i}$
 $\,(i=1,...,k)$.
\end{no}

If $p\in\GG$ and $X\in\gg$ we set
$\vartheta_{(X,p)}(t):=\esp[tX](p)\, (t\in\R)$, i.e.
$\vartheta_{(X,p)}$ denotes the integral curve of $X$ starting
from $p$ and it is a {\it{1-parameter sub-group}} of $\GG$. The
{\it{Lie group exponential map}} is then defined by
$$\esp:\gg\longmapsto\GG,\quad\esp(X):=\esp[X](1).$$ It turns out
that $\esp$ is an analytic diffeomorphism between $\gg$ and $\GG$
whose inverse will be denoted by $\llog.$ Moreover we have
$$\vartheta_{(X,p)}(t)=p\bullet\esp(t
X)\quad\forall\,\,t\in\R.$$From now on we shall fix on $\GG$ the
so-called {\it exponential coordinates of 1st kind}, i.e. the
coordinates associated to the map $\llog$.

As for any nilpotent Lie group, the {\it Baker-Campbell-Hausdorff
formula} (see \cite{Corvin}) uniquely determines the group
multiplication $\bullet$ of $\GG$, from the ``structure'' of its
own Lie algebra $\gg$. In fact, one has
$$\esp(X)\bullet\esp(Y)=\esp(X\star Y)\,\,(X,\,Y \in\gg),$$ where
${\star}:\gg \times \gg\longrightarrow \gg$ is the
 {\it Baker-Campbell-Hausdorff product} defined by \begin{eqnarray}\label{CBHf}X\star Y= X +
Y+ \frac{1}{2}[X,Y] + \frac{1}{12} [X,[X,Y]] -
 \frac{1}{12} [Y,[X,Y]] + \mbox{ brackets of length} \geq 3.\end{eqnarray}

Using exponential coordinates, \eqref{CBHf} implies that the group
multiplication $\bullet$ of $\GG$ is polynomial and explicitly
computable (see \cite{Corvin}). Moreover, $0=\esp(0,...,0)$ and
the inverse of $x\in\GG$
 $(x=\esp(x_1,...,x_{n}))$ is ${x}^{-1}=\esp(-{x}_1,...,-{x}_{n})$.

\begin{no}\label{Janeiro}Using exponential coordinates for
$\GG$, every point $x=\esp(\sum_I x_I \ee_I)\in\GG$ can be
regarded as $n$-tuple $x=(x_1,...,x_\DH,
x_{\DH+1},...,x_n)\in\R^n$. We shall
set$$x\cc:=(x_1,...,x_\DH)\in\R^\DH,\qquad
x\vv:=(x_{\DH+1},...,x_n)\in\R^\vd \qquad(n=\DH+\vd).$$Hence
$x=\esp(x\cc, x\vv)\equiv(x\cc, x\vv)$.
\end{no}

When we endow the horizontal subbundle with a metric
$g\cc=\langle\cdot,\cdot\rangle\cc$, we say that $\GG$
 has a {\it sub-Riemannian structure}.  Is important to note that
 it is always possible to define a left-invariant Riemannian metric
  $g =\langle\cdot,\cdot\rangle$ in such a way that the
  frame $\underline{X}$ turns out to be {\it
orthonormal} and such that $g_{|\HH}=g\cc$. For this, it is enough
to choose a Euclidean metric on $\gg=\TT_0\GG$ which can be
left-translated to the whole tangent bundle and, by this way, the
direct sum (\ref{alg}) becomes an orthogonal direct sum.

Since for Carnot groups the hypotheses of Chow's Theorem trivially
apply, the {\it Carnot-Carath\'eodory distance} $\dc$ associated
with $g\cc$ can be defined as before, and $\dc$ makes $\GG$ a
complete metric space in which every couple of points can be
joined by (at least) one $\dc$-geodesic; see \cite{Bellaiche},
\cite{Montgomery}.

We remind that Carnot groups are {\it homogeneous groups} (see
\cite{Stein}), i.e. they have a 1-parameter group of automorphisms
$\delta_t:\GG\longrightarrow\GG$ $(t>0)$. Using exponential
coordinates, we have $\delta_t x
=\esp(\sum_{j,i_j}t^j\,x_{i_j}\ee_{i_j})$ for all
$x=\esp(\sum_{j,i_j}x_{i_j}\ee_{i_j})\in\GG.$\footnote{Here,
$j\in\{1,...,k\}$ and $i_j\in I{\!_{^{_{{\HH}_j}}}}=\{
n_{j-1}+1,..., n_{j}\}$.} The {\it homogeneous dimension} of $\GG$
is the integer $\Qdim$, coinciding with the {\it Hausdorff
dimension} of $(\GG,\dc)$ as a metric space (see \cite{Mi},
\cite{Montgomery}, \cite{Gr1}).

We remind that the {\it structural constants} of the Lie algebra
$\gg$ associated with the (left-invariant) frame $\underline{X}$
are defined by
$$\SC^R_{IJ}:=\langle [X_I,X_J],
 X_R\rangle\quad(I,J, R=1,...,n).$$
\noindent {They satisfy the customary properties:
\begin{itemize}\item  $\SC^R_{IJ} +\SC^R_{JI}=0$\,\,(skew-symmetry) \item
$\sum_{J=1}^{n} \SC^I_{JL}\SC^{J}_{RM} + \SC^I_{JM}\SC^{J}_{LR} +
\SC^I_{JR}\SC^{J}_{ML}=0$\,\,(Jacobi's identity).\end{itemize}
\noindent The stratification hypothesis on the Lie algebra implies
the following further property:
\begin{equation}\label{C}X_i\in {\HH}_{l},\, X_j \in
{\HH}_{\DH}\Longrightarrow [X_i,X_j]\in {\HH}_{l+m}.\end{equation}
Therefore, if $i\in I{\!_{^{_{{\HH}_s}}}}$ and $j\in
I{\!_{^{_{{\HH}_r}}}}$, one has
\begin{equation}\label{chypc} \SC^m_{ij}\neq 0 \Longrightarrow m \in I{\!_{^{_{{\HH}_{s+r}}}}}.\end{equation}

\begin{Defi}\label{nota}Throughout this paper we will
use the following notation:\begin{itemize}\item[(i)]
$C^\alpha\cc:=[\SC^\alpha_{ij}]_{i,j\in
I\cc}\in\mathcal{M}_{\DH_1\times \DH_1}(\R)\,\,\,\qquad(\alpha\in
I\cd)$; \item[(ii)] $ C^\alpha:=[\SC^\alpha_{IJ}]_{I,J=1,...,n}\in
\mathcal{M}_{n\times n}(\R)\qquad(\alpha\in I\vv).$\end{itemize}
Furthermore, for any $Z=\sum_{\alpha\in I\vv}z_\alpha
X_\alpha\in\VV$, we will set
\begin{itemize}\item[(iii)] $C\cc(Z):=\sum_{\alpha\in I\cd}z_\alpha
C^\alpha\cc$; \item[(iv)] $C(Z):=\sum_{\alpha\in I\vv}z_\alpha
C^\alpha$.\end{itemize}
\end{Defi}

\begin{Defi}\label{twist}
Let $i\in\{1,...,k-1\}$. Then the $i$-{\bf{th curvature}} of $\HH$
 is the skew-symmetric, bilinear map
$$\Om\ci:\HH\otimes{\HH}_{i}\longrightarrow {\HH}_{i+1},\qquad
\Om\ci(X\otimes Y):=[X,Y] \mod {\TT}^{i}$$whenever $\,X\in\HH$ and
$Y\in \HH_i$. By definition of $k$-step Carnot group, one has
$\Om\ck(\cdot,\cdot)=0$.
\end{Defi}Since the bracket map
$[\cdot,\cdot]:\HH\otimes{\HH}_{i}\longrightarrow {\HH}_{i+1}\,
(i=1,...,{k-1})$ is surjective, the definition is well posed. We
stress that the 1st curvature
$\Om\cc(\cdot,\cdot):=\Om\cu(\cdot,\cdot)$ of $\HH$ is the
``curvature of a distribution''; see \cite{GE1}, \cite{Gr2},
\cite{Montgomery}.

If  $Y\in\TG$, let $Y=(Y_1,...,Y_k)$ be the canonical
decomposition of $Y$ with respect to the Carnot grading, i.e.
$Y=\sum_{i=1}^k\P\ci(Y)$, where $\P\ci$ is the orthogonal
projection onto $\HH_i$. Set
$$\Om(X,Y):=\sum_{i=1}^{k-1}\Om\ci(X,Y_i),$$for $X\in \HH$
and $Y\in\TG$. Then we have

\begin{lemma}\label{twist1}Let $X\in \HH$ and $Y, Z\in \TG$. Then
we have\begin{itemize}\item[(i)]
$\langle\Om\cc(X,Y),Z\rangle=\langle C\cc(Z) Y, X\rangle;$
\item[(ii)] $\langle\Om(X,Y),Z\rangle=\langle C(Z) Y,
X\rangle.$\end{itemize}
\end{lemma}
\begin{proof}The proof is an immediate consequence of Definition \ref{twist} and Definition
\ref{nota}.

\end{proof}

 In the sequel, we will give a definition of connection
which recovers the usual definitions of Riemannian, partial and
non-holonomic connections. Classical notions of connection
(linear, affine or Riemannian) and related topics can be found in
\cite{Helgason}, \cite{Hicks}.

Partial connections was defined by Z. Ge in \cite{GE1}; see also
\cite{Gr2} and \cite{KRP}. Non-holonomic connections were used by
\'E. Cartan
 in his studies on non-holonomic mechanics and
by  the Russian school; see the survey by Vershik and Gershkovich,
\cite{Ver}.

\begin{Defi}\label{CONNgenerale} Let $N$ be a $\cin$ smooth
manifold
 and let $\pi_{_{E}}:E\longrightarrow {N},
\,\pi_{_{F}}:F\longrightarrow {N}$ be smooth subbundles of $\TT
{N}$. An $E$-{\bf connection} $\nabla^{(E,F)}$ {\bf on} $F$ is a
rule which assigns to each vector field $X\in \cin({N},E)$ an
$\R$-linear transformation
$\nabla^{(E,F)}_X:\cin({N},F)\longrightarrow\cin({N},F)$ such that
\begin{eqnarray*}{\mathit{(i)}}\,\,\nabla^{(E,F)}_{f X +
g Y} Z &= f \nabla^{(E,F)}_X Z + g \nabla^{(E,F)}_Y Z\quad
&\forall\,\,X,\,Y \in \cin(N,E)\,\,\forall\,\,Z \in
\cin(N,F)\\\,\,&&\forall\,\,f,\,g\in\cin({N});\\
{\mathit{(ii)}}\,\,\nabla^{(E,F)}_{X} f Y &= f \nabla^{(E,F)}_X\,Y
+  (X f)\,Y\quad &\forall\,\,X,\,Y\in
\cin(N,E)\qquad\forall\,\,f\in\cin({N}).\end{eqnarray*}
 If $E=F$ we shall set $\blae:=\nabla^{(E,E)}$ and
call $\blae$ an $E$-{\bf connection}. Any such connection will be
called a {\bf partial connection} of $\TT N$. If $E=\TT {N}$, then
$\nabla^{(\TT {N},F)}$ is called a {\bf non-holonomic
$F$-connection}\footnote{This definition recovers the usual one of
``vector bundle connection'' (see \cite{Milnor2}) where instead of
a generic vector bundle $\pi: F\longrightarrow N$ we make use of a
subbundle of the tangent bundle.}. If $E$ has a positive definite
inner product $g\ce$, then an $E$-connection $\blae$ is said {\bf
metric preserving} if
\begin{eqnarray*}{\mathit{(iii)}}\,\,Z g\ce(X,Y)=g\ce(\blae_{Z}X,Y)+
g\ce(X,\blae_{Z}Y)\qquad \forall\, X, Y,
Z\in\cin(N,E).\end{eqnarray*} The {\bf torsion} $\ore$ associated
to the $E$-connection $\blae$ is defined by
$$\ore(X,Y):=\blae_{X}Y -\blae_{Y}X -
\P\ce[X,Y]\quad \forall\, X, Y\in \cin(N,E),$$where $\P\ce:\TT
N\longrightarrow E$ denotes the orthogonal projection onto $E$. An
$E$-connection is {\bf torsion free} if $\ore(X,Y)=0$ for every
$X, Y\in\cin(N,E)$. We shall say that $\blae$ is the {\bf
Levi-Civita $E$-connection} on $E$ if it is metric preserving and
torsion-free. Note that if $E=\TT N$, terminology and definitions
adopted here are the customary ones. In this case, we will denote
by $\nabla$ the {\bf Levi-Civita connection} on $\TT N$.\end{Defi}

We stress that the difference between the definitions of partial
and non-holonomic connection is that the latter allows us to
covariantly differentiate along any curve of ${N}$ whereas using
the first one only curves that are tangent to the subbundle $E$
can be considered.

\begin{Defi}\label{parzconn}
In the sequel, $\nabla$ will denote the unique {\bf left-invariant
Levi-Civita connection} on $\GG$ associated with the fixed left
invariant metric $g$. Moreover, for any $X,
Y\in\XH:=\cin(\GG,\HH)$, we shall set $\gc_X Y:=\PH(\nabla_X Y)$.
We note that $\gc$ is a partial connection, also called {\bf
horizontal connection} or {\bf $\HH$-connection}. For notational
convenience, in the sequel we will denote by the same symbol the
non-holonomic connection on $\GG$, i.e. $\gc=\nabla^{(\TG,\HH)}$.
\end{Defi}

\begin{Defi}\label{horcur} We define the {\bf horizontal curvature} $\RC\cc$ of the $\HH$-connection $\gc$
to be the trilinear map $
\RC\cc:{\HH}\times{\HH}\times{\HH}\longrightarrow{\HH}$ given by
\begin{eqnarray}\RC\cc(X,Y)Z:=\gc_Y\gc_X Z -
\gc_X \gc_Y Z - \gc_{\PH[Y, X]} Z,
\end{eqnarray}where $X, Y, Z\in {\XH}$. In the sequel, $\RC$ will denote the {\bf Riemannian
curvature tensor}, defined by
\begin{eqnarray*}\RC(X,Y)Z:=\nabla_Y\nabla_XZ-\nabla_X\nabla_YZ -\nabla_{[Y,X]}Z\quad
(X, Y, Z\in\XG).\end{eqnarray*}
\end{Defi}

\begin{oss}\label{flatness}From
Definition \ref{parzconn}, using the properties of the structural
constants of any Levi-Civita connection, we get that the
horizontal connection $\gc$ is {\bf flat}, i.e.
$$\gc_{X_i}X_j=0\qquad (i,j\in I\cc).$$
Note that $\gc$ turns out to be compatible  with  the
sub-Riemannian metric $g\cc$, i.e.
$$X\langle Y, Z \rangle\cc=\langle \gc_X Y, Z \rangle\cc
+ \langle  Y, \gc_X Z \rangle\cc\qquad \forall\,\,X, Y, Z\in
\XH.$$This follows immediately from the very definition of $\gc$
using the compatibility property of the left-invariant Levi-Civita
connection $\nabla$ with respect to the Riemannian metric $g$.
Furthermore, $\gc$ is torsion-free, i.e.
$$\gc_X Y - \gc_Y X-\PH[X,Y]=0\qquad \forall\,\,X, Y\in \XH.$$In
particular, it turns out that  the horizontal curvature $\RC\cc$
is identically zero.
\end{oss}

\begin{Defi}
If $\psi\in\cin({\GG})$ we define the {\bf horizontal gradient} of
$\psi$,  $\dg \psi$, as the unique horizontal vector field such
that
$$\langle\dg \psi,X \rangle\cc= d \psi (X) = X
\psi\quad \forall \,X\in \XH.$$Later on, we will denote by
$\mathcal{J}\cc$ the Jacobian matrix of a vector-valued function,
computed with respect to any given orthonormal frame
$\underline{\tau}\cc=\{\tau_1,...,\tau_{h_1}\}$ for $\HH$.
\end{Defi}

We shall now define the left-invariant co-frame
$\underline{\omega}:=\{\omega_I:I=1,...,n\}$ dual  to
$\underline{X}$ (with respect to to the left invariant  metric
$g$). The {\it left-invariant 1-forms} \footnote{That is, $L_x
^{\ast}\omega_I=\omega_I$ for every $x\in\GG.$} $\omega_I$ are
uniquely determined by the condition:
$$\omega_I(X_J)=\langle X_I,X_J\rangle=\delta_I^J\quad {\mbox{(Kroneker)}}\quad\,(I, J=1,...,n).$$
{The {\it Cartan's structure equations} for the left-invariant
co-frame $\underline{\omega}$ are given by: \[{\rm(I)}\quad d
\omega_I = \sum_{J=1}^{n} \omega_{IJ}\wedge \omega_J,\qquad\,\,\,
{\rm(II)}\quad d \omega_{JK} = \sum_{L=1}^{n} \omega_{JL}\wedge
\omega_{LK} - \Om_{JK}\qquad(I, J, K=1,...,n),\] where
$\omega_{IJ}(X)=\langle\nabla_{X} X_I, X_J\rangle$ are the {\it
connection 1-forms} for $ \underline{\omega}$  while $\Om_{JK}$
are the {\it curvature 2-forms}, defined by
$$\Om_{JK}(X,Y)=\omega_K(\RC(X,Y)X_J)\qquad(X, Y\in\XG).$$

For what concerns the theory of connections on Lie group and
left-invariant differential forms, see \cite{Helgason}. Moreover,
for many topics about the geometry of nilpotent Lie groups
equipped with a left-invariant connection, see \cite{3}; for the
Carnot case see \cite{Monte, Monteb}.

\begin{oss}We have\begin{equation}\label{c2}\nabla_{X_I} X_J =
\frac{1}{2}\sum_{R=1}^n( \SC_{IJ}^R  - \SC_{JR}^I + \SC_{RI}^J)
X_R\qquad (I,\,J=1,...,n).\end{equation}This formula and condition
\eqref{C} are needed to make explicit computations in terms of the
structural constants. For instance, from (\ref{c2}) it follows
that the 1st structure equation for the coframe
$\underline{\omega}$, becomes
\begin{equation}\label{dext}d
\omega_R = - \frac{1}{2}\sum_{1\leq I,J \leq n_{i-1}}
\SC^R_{IJ}\,\, \omega_I\wedge\omega_J,\end{equation}where $R\in
I\ci=\{j: n_{i-1}< j \leq n_{i}\}$ and $i=1,...,k$.\end{oss}

In the sequel we will need the notion of {\it covariant derivative
along a path}; we refer the reader to \cite{Ch1} for a detailed
introduction.

\begin{Defi}Let $\gamma:[a,b]\subset\R\longrightarrow\GG$
be a $\mathbf{C}^1$ path and let
$X:[a,b]\subset\R\longrightarrow\TG$,
$X=\sum_I\xi_I(t)X_I(\gamma)$, be a vector field along $\gamma$.
Then the {\bf covariant derivative of $X$ along $\gamma$}, denoted
by $\nabla_t$, is defined as$$\nabla_tX:=
\nabla_{\dot{\gamma}}X=\sum_{I=1}^n\Big\{\dot{\xi_I}+ \sum_{J,
K=1}^n\Gamma_{JK}^I(\gamma) \xi_J\,
\dot{\gamma_K}\Big\}X_L(\gamma),$$where $\nabla$ denotes the
Levi-Civita connection on $\GG$ and
$$\Gamma_{JK}^I:=\langle\nabla_{X_K}X_J,X_I\rangle=\omega_{JI}(X_K)\qquad(I,
J, K=1,...,n)$$ are the  Christoffel symbols of $\nabla$, with respect to
the left invariant frame $\underline{X}=\{X_1,...,X_n\}$ on
$\GG$.\end{Defi}

We end this section with some important
examples.\begin{es}[Heisenberg group $\mathbb{H}^1$]
\label{epocase}Let $\mathfrak{h}_1:=\TT_0\mathbb{H}^1=\R^{3}$
denote the Lie algebra of the Heisenberg group $\mathbb{H}^1$,
that is the most simple example of $2$-step Carnot group. Its Lie
algebra $\mathfrak{h}_1$
satisfies:$$[\ee_{1},\ee_{1}]=\ee_{3}$$and all other commutators
vanish. We have $\mathfrak{h}_1=\HH\oplus \R\ee_{3}$ where
$\HH={\rm span}_{\R}\{\ee_1,\ee_2\}.$ In particular, the 2nd layer
of the grading $\R\ee_{3}$ is the center of the Lie algebra
$\mathfrak{h}_n$. These conditions determine the group law
$\bullet$ via the Baker-Campbell-Hausdorff formula. Indeed, if
$x=\esp(\sum_{i=1}^{3}x_i\ee_i),\, y=\esp(\sum_{i=1}^{3}y_i
\ee_i)\in \mathbb{H}^1$, one has\begin{center}$x\bullet y =\esp
\big(x_1 + y_1,x_2+y_2, x_{3} + y_{3} + \frac{1}{2}(x_{1} y_{2}-
x_{2} y_{1})\big).$\end{center}The standard frame of orthonormal
left invariant vector fields for $\mathbb{H}^1$ is given, using
exponential coordinates, by\begin{eqnarray*}X_1(x)&:=&L_{x\ast
}\ee_1=\ee_1-\frac{x_2}{2}\,\ee_3;\\
X_2(x)&:=&L_{x\ast}\ee_2=\ee_2+\frac{x_1}{2}\,\ee_3;\\X_3(x)&:=&L_{x\ast
}\ee_3=\ee_3.\end{eqnarray*}

\end{es}
\begin{es}[Heisenberg group $\mathbb{H}^n$] \label{epocase}Let $\mathfrak{h}_n:=\TT_0\mathbb{H}^n=\R^{2n + 1}$
denote the Lie algebra of the Heisenberg group $\mathbb{H}^n$,
that is the most important example of $2$-step Carnot group. Its
Lie algebra $\mathfrak{h}_n$ is defined by the following
rules:$$[\ee_{i},\ee_{i+1}]=\ee_{2n + 1}\qquad\mbox{for
every}\,\,\,
 i=2k-1,\,\,k=1,...,n=\frac{\DH}{2}$$and all other commutators
vanish. One has $\mathfrak{h}_n=\HH\oplus \R\ee_{2n+1}$, where
\[\HH={\rm span}_{\R}\{\ee_i:i=1,...,2n\}.\]The 2nd
layer of the grading $\R\ee_{2n+1}$ is the center of
$\mathfrak{h}_n$. The above conditions uniquely determine the
group law $\bullet$ via the Baker-Campbell-Hausdorff formula. More
precisely, if $x=\esp(\sum_{i=1}^{2n+1}x_i\ee_i),\,
y=\esp(\sum_{i=1}^{2n+1}y_i \ee_i)\in \mathbb{H}^n$,
then\begin{eqnarray*}x\bullet y =\esp \Big(x_1 + y_1,...,x_{2n} +
y_{2n}, x_{2n+1} + y_{2n+1} + \frac{1}{2}\sum_{k=1}^{n} (x_{2k-1}
y_{2k}- x_{2k} y_{2k-1})\Big).\end{eqnarray*}If $i\in\{1,...,n\}$,
the standard frame of orthonormal left invariant vector fields for
$\mathbb{H}^n$ is given, using exponential coordinates, by
\begin{eqnarray*}
X_{2i-1}(x)&:=&L_{x\ast}\ee_{2i-1}=\ee_{2i-1}-\frac{x_{2i}}{2}\,\ee_{2n+1};\\
X_{2i}(x)&:=&L_{x\ast}\ee_{2i}=\ee_{2i}+\frac{x_{2i-1}}{2}\,\ee_{2n+1};\\
X_{2n+1}(x)&:=&L_{x\ast}\ee_{2n+1}=\ee_{2n+1}.\end{eqnarray*}
\end{es}

\begin{es}[2-step Carnot groups]We have $\gg=\HH\oplus\HH_2$, where $\HH_2$ is the center of
$\gg$. Furthermore, we have the following rules:
$$[\ee_i,\ee_j]=\sum_{\alpha\in
I\cd}\SC^{\alpha}_{ij}\,\ee_\alpha\qquad\mbox{for every}\,\,\,
i,\,j\in I\cc=\{1,...,\DH_1\}$$and all other commutators vanish.
Let $x=\esp(\sum_{i=1}^{n}x_i\ee_i), \,y=\esp(\sum_{i=1}^{n}y_i
\ee_i)\in \GG$. Then
\begin{eqnarray*}x\bullet y =\esp \Big(x+y-\frac{1}{2}
\sum_{\alpha\in I\cd}\langle C^\alpha\cc x,y\rangle\ee_\alpha
\Big).\end{eqnarray*}The standard frame of orthonormal left
invariant vector fields for $\GG$ is given
by\begin{eqnarray*}X_{i}(x)&=&L_{x\ast
}\ee_{i}=\ee_{i}-\frac{1}{2}\sum_{\alpha\in I\cd}\langle
C^\alpha\cc x,\ee_i\rangle \ee_{\alpha}\qquad (i\in
I\cc)\\X_{\alpha}(x)&=&L_{x\ast
}\ee_{\alpha}=\ee_{\alpha}\qquad\qquad\qquad\,\,\qquad\qquad(\alpha\in
I\cd).\end{eqnarray*}Remind that
$C^\alpha\cc=[\SC^\alpha_{ij}]_{i,j\in I\cc}$; see Definition
\ref{nota}.
\end{es}

\begin{es}[Engel group $\mathbb{E}^1$]\label{e} The Engel group is
the simpler example of a $3$-step Carnot group. Its Lie algebra
$\mathfrak{e}$ is $4$-dimensional and is defined by the following
rules:
$$[\ee_{1},\ee_{2}]=\ee_{3},\,\,[\ee_{1},\ee_{3}]=\ee_{4}$$and
all other commutators vanish. We have $\mathfrak{e}=\HH\oplus\R
\ee_3\oplus\R\ee_4,$
 where $\HH=\rm{span}_{\R}\{\ee_1,\ee_2\}$ and the center of the Lie algebra $\mathfrak{e}$ is ${\R}\ee_4.$
Therefore $C^3\cc=\left|\!\!
\begin{array}{cc}
  0 & 1 \\
  -1 & 0 \\
\end{array}\!\!
\right|$ and $$C^4=\left|\!\!
\begin{array}{cccc}
  0 & 0 & 1 & 0 \\
  0 & 0 & 0& 0 \\
  -1 & 0 & 0& 0\\0 & 0 & 0& 0\\
\end{array}\!\!
\right|.$$ The group law $\bullet$ is given, for
$x=\esp(\sum_{i=1}^{4}x_i\ee_i),$ $y=\esp(\sum_{i=1}^{4}y_i
\ee_i)\in \mathbb{E}^1$, by
\begin{eqnarray*}x\bullet y =\esp \bigg(x+y-\frac{1}{2}
\langle C^3\cc x,y\rangle\ee_3 - \Big(\frac{1}{2} \langle C^4
x,y\rangle + \frac{1}{12}\langle C^3\cc x\cc,y\cc\rangle
\big\langle C^4\ee_3,(x-y)\big\rangle
\Big)\ee_4\bigg).\end{eqnarray*}The standard frame of orthonormal
left invariant vector fields for $\mathbb{E}^1$ is given
by\begin{eqnarray*}X_{i}(x)&:=&L_{x\ast
}\ee_{i}=\ee_{i}-\frac{1}{2}\langle C^3\cc x\cc,\ee_i\rangle \ee_3
-\Big(\frac{1}{2}\langle C^4 x,\ee_i\rangle  + \frac{1}{12}\langle
C^3\cc x\cc,\ee_i\rangle
\langle C^4\ee_3,x\rangle\Big)\ee_4\quad (i=1,2)\\
X_{3}(x)&:=&L_{x\ast }\ee_{3}=\ee_{3}-\frac{1}{2}\langle C^4
x,\ee_3\rangle \ee_4
\\X_{4}(x)&:=&L_{x\ast
}\ee_{4}=\ee_{4}.\end{eqnarray*}\end{es}

\begin{es}[3-step Carnot groups]\label{3stex}We have that  $\gg=\HH\oplus\HH_2\oplus\HH_3$ where $\HH_3$
is the center of $\gg$. In order to describe $\gg$, we make use of
the matrices of the structure constants of $\gg$,
i.e.\begin{eqnarray*}C^\alpha\cc&=&[\SC^\alpha_{ij}]_{i,j\in
I\cc}\quad\quad\,\,\quad(\alpha\in
I\cd)\\C^\alpha&=&[\SC^\alpha_{IJ}]_{I,J=1,...,n}\qquad(\beta\in
I\ctr).\end{eqnarray*}The group law $\bullet$ is given, for
$x=\esp(\sum_{i=1}^{n}x_i\ee_i),$ $y=\esp(\sum_{i=1}^{n}y_i
\ee_i)\in \GG$, by
\begin{eqnarray*}x\bullet y &=&\esp \bigg(x+y-\frac{1}{2}
\sum_{\alpha\in I\cd}\langle C^\alpha\cc
x\cc,y\cc\rangle\ee_\alpha
 \\&&-\sum_{\beta\in I\ctr}
\Big[\frac{1}{2} \langle C^\beta x,y\rangle -
\frac{1}{12}\sum_{\alpha\in I\cd}\big\langle
C^\beta(x-y),\ee_\alpha\big\rangle\langle C^\alpha \cc
x\cc,y\cc\rangle \Big]\ee_\beta\bigg).\end{eqnarray*}The standard
frame of orthonormal left invariant vector fields for $\GG$ is
given by\begin{eqnarray*}X_{i}(x)&=&L_{x\ast
}\ee_{i}\\&=&\ee_{i}-\frac{1}{2}\sum_{\alpha\in I\cd}\langle
C^\alpha\cc x\cc,\ee_i\rangle \ee_{\alpha}-\sum_{\beta\in
I\ctr}\Big[\frac{1}{2}\langle C^\beta x,\ee_i\rangle  -
\frac{1}{12}\sum_{\alpha\in I\cd}\langle C^\beta x,
\ee_\alpha\rangle\langle C^\alpha\cc x\cc,\ee_i\rangle
\Big]\ee_\beta\\&&(i\in I\cc);\\ X_{\gamma}(x)&=&L_{x\ast
}\ee_{\gamma}=\ee_{\gamma}-\frac{1}{2}\sum_{\beta\in I\ctr}\langle
C^\beta x,\ee_\gamma\rangle\ee_\beta\quad(\gamma\in I\cd);
\\X_{\beta}(x)&=&L_{x\ast
}\ee_{\alpha}=\ee_{\beta}\quad (\beta\in I\ctr).\end{eqnarray*}
\end{es}
We remind that, for any $z=\esp(\sum_Iz_I\ee_I)\in\GG$ we usually
set $z\cc:=(z_1,...,z_\DH)\in\R^{\DH}$; see Notation
\ref{Janeiro}.

\subsection{Hypersurfaces: some basic facts}\label{dere} We now remind some basic facts about
hypersurfaces which will be needed in the sequel. This material
can be found in \cite{Monte, Monteb}.

Later on, $\mathcal{H}^m_{\bf cc}$ and $\mathcal{S}^m_{\bf cc}$
will denote, respectively, the usual and the spherical Hausdorff
measures associated with the CC-distance. The (left-invariant)
{\it Riemannian volume form} on $\GG$ is defined as
$\sigma^n\rr:=\Lambda_{i=1}^n\omega_i\in \Lambda^n(\TG).$

\begin{oss} By integrating $\sigma^n\rr$ we obtain a measure ${vol}^n\rr$,
which is the {\it Haar measure} of $\GG$. Since the determinant of
${L_{x}}_{\ast}$ is equal to 1, this measure equals the measure
induced on $\GG$ by the push-forward of the $n$-dimensional
Lebesgue measure $\mathcal{L}^n$ on $\Rn\cong\gg$. Moreover, up to
a constant multiple, ${vol}^n\rr$ equals the {\it $Q$-dimensional
Hausdorff measure} $\HC$ on $\GG$. This follows because they are
both Haar measures for the group and therefore equal up to a
constant; see \cite{Montgomery}. Here we assume this constant
equal to $1$.
\end{oss}

In the study of hypersurfaces of Carnot groups we have to
introduce the notion of {\it characteristic point}.
\begin{Defi}\label{caratt}
If $S\subset\GG$ is a $\mathbf{C}^r$-smooth $(r=1,...,\infty)$
hypersurface, we say that $S$ is {\bf characteristic} at $x\in S$
if $\dim\,\HH_x = \dim (\HH_x \cap \TT_x S)$ or, equivalently, if
$\HH_x\subset\TT_x S$. The {\bf characteristic set} of $S$ is
denoted by $C_S$, i.e.$$ C_S:=\{x\in S : \dim\,\HH_x = \dim (\HH_x
\cap \TT_x S)\}.$$
\end{Defi}
A hypersurface $S\subset\GG$, oriented by its unit normal vector
$\nu$, is {\it non-characteristic} if, and only if, the horizontal
subbundle $\HH$ is {\it transversal} to $S$ ($\HH\pitchfork \TS$).
We have then
$$\HH_x\pitchfork\TT_x S\Longleftrightarrow \PH \nu(x)\neq 0\Longleftrightarrow\exists
  X\in\XH:
 \langle X(x), \nu(x)\rangle \neq 0\qquad(x\in S),$$ where
 $\PH:\TG\longrightarrow\HH$ denotes the orthogonal projection onto
 $\HH$.

\begin{oss}[Hausdorff measure of $C_S$; see \cite{Mag}]
If $S\subset\GG$ is a $\mathbf{C}^1$-smooth hypersurface, $r\geq
1$ then the $Q-1$-dimensional Hausdorff measure associated with
$\dc$ of $C_S$ is zero, i.e.
$$\mathcal{H}_{\bf cc}^{Q-1}(C_S)=0.$$\end{oss}

\begin{oss}[Riemannian measure on hypersurfaces]
Let $S\subset\GG$ be a $\mathbf{C}^r$-smooth hypersurface and let
$\nu$ denote the unit normal vector along $S$. By definition, the
$\,n-1$-dimensional Riemannian measure along $S$ is given
by\begin{equation}\label{misup}\sigma^{n-1}\rr\res
S:=(\nu\LL\sigma^{n}\rr)|_{S},\end{equation} where  $\LL$ denotes
the ``contraction'', or interior product, of a differential
 form\footnote{\label{fo4}The linear map $\LL:
\Lambda^k(\TG)\rightarrow\Lambda^{k-1}(\TG)$ is defined, for
$X\in\TG$ and $\omega^k\in\Lambda^k(\TG)$, by $(X \LL \omega^k)
(Y_1,...,Y_{k-1}):=\omega^k (X,Y_1,...,Y_{k-1})$; see
\cite{Helgason}, \cite{FE}.}.\end{oss}

Since we shall study smooth hypersurfaces, instead of the usual
weak definition of $\HH$-perimeter measure (see \cite{A2},
\cite{CDG}, \cite{CDG}, \cite{gar}, \cite{FSSC3, FSSC4, FSSC5},
\cite{GN}) we now introduce a $(n-1)$-differential form which, by
integration, coincides with the $\HH$-perimeter measure.

\begin{Defi}[$\per$-measure on hypersurfaces]
Let $S\subset\GG$ be a $\mathbf{C}^r$-smooth non-characteristic
hypersurface and let us denote by $\nu$ its unit normal vector. We
will call $\HH$-{\bf normal} along $S$, the normalized projection
onto $\HH$ of $\nu$, i.e.$$\nn: =\frac{\PH\nu}{|\PH\nu|}.$$ We
then define the $(n-1)$-dimensional measure $\per$ along ${S}$ to
be the measure associated with the $(n-1)$-differential form
$\per\in\Lambda^{n-1}(\TS)$ given by the contraction of the volume
form $\sigma^n\rr$ of $\GG$ with the horizontal unit normal $\nn$,
i.e.\begin{equation}\per \res S:=(\nn \LL
\sigma^n\rr)|_S.\end{equation} If we allow $S$ to have
characteristic points, we may trivially extend the definition of
$\per$ by setting $\per\res C_{S}= 0$. We stress that $\per \res S
= |\PH \nu |\cdot\sigma^{n-1}\rr\, \res S$.\end{Defi}

\begin{Defi}\label{carca}
If $\nn$ is the horizontal unit normal along $S$, at each regular
point $x\in S\setminus {\it C}_S$ one has that $\HH_x= (\nn)_x
\oplus \mathit{H}_x S$, where we have set$$\mathit{H}_x
S:=\HH_x\cap\TT_x S.$$We call $\mathit{H}_x S$ the {\bf horizontal
tangent space} at $x$ along $S$. We define in the obvious way the
associated subbundles $\HS (\subset \TS)$ and $\nn S$, called,
respectively, {\bf horizontal tangent bundle} and {\bf horizontal
normal  bundle} of $S$. Moreover we shall set:
\begin{itemize}\item[{(i)}]$\mathcal{N}:=\frac{\nu}{|\PH\nu|}=(\nn, \varpi),$
 where \[\varpi:=\frac{\P\vv\nu}{|\PH\nu|};\]\item[{(ii)}]$\varpi_\alpha:=\frac{\nu_\alpha}{|\PH\nu|}\qquad
(\alpha\in I\vv)$;\item[{(iii)}] $C\cc(\varpi):=\sum_{\alpha\in
{I\cd}} \varpi_\alpha\,C^\alpha\cc;$ \item[{(iv)}]
$C(\varpi):=\sum_{\alpha\in {I\vv}}
\varpi_\alpha\,C^\alpha.$\end{itemize}
\end{Defi}

\section{\large The CC-distance in Carnot groups}\label{vipo}

\subsection{On normal and abnormal CC-geodesics in Carnot
groups}\label{onnorm} Below we shall introduce the main notions
about normal and abnormal geodesics in the setting of Carnot
groups and we shall explicitly write down the associated
equations. Our approach is the Hamiltonian one and we follow that
given in \cite{GoKA} (see also \cite{Montgomery2}). However, at
the end of this section, we shall also briefly recall and discuss
the Lagrangian point of view, which will turn out to be useful in
the sequel.

From now on we shall set \[\DH:=\DH_1=\dim\HH.\]

As already recalled in the previous section, the CC-metric $\dc$
measures the distance between two given points $p$ and $q$ by
minimizing the length of all (absolutely continuous) horizontal
curves (i.e. tangent to the horizontal subbundle
$\HH\subset\TT\GG$) joining $p$ and $q$. Thus we have to study
minimizers for this metric. A {\it minimizer} is any absolutely
continuous horizontal curve $\gamma:I\subset\R\longrightarrow\GG$
which is such that for every $t\in I$ there exists $\epsilon>0$
such that $\gamma$ minimizes the length between $\gamma(t_0)$ and
$\gamma(t_1)$ whenever $t_0, t_1$ belong to $(t-\epsilon,
t+\epsilon)\subset I$. A specific feature of sub-Riemannian
geometry is that minimizers can be of two different types. However
they are not necessarily mutually exclusive. Roughly speaking, the
minimizers of the first type, called {\it normal}, are projections
of solutions of a Hamiltonian system and, in a sense, they
generalize the Riemannian situation since, in particular, they are
differentiable. Minimizers of the second type are called {\it
abnormal} or {\it singular}. Although their existence was
originally deduced from the Pontrjagin Maximum Principle (see
\cite{Montgomery} and discussion therein), they can also be
defined as projection onto $\GG$ of characteristic curves (in the
symplectic sense) of the annihilator of $\HH$ in the cotangent
bundle $\TT^\ast\GG$, as we shall see below.

For sake of completeness we recall the main definitions of the
hamiltonian formalism in our particular setting. A {\it
Hamiltonian} is a function
$\mathcal{H}:\TT^\ast\GG\longrightarrow\R$ where the cotangent
bundle $\TT^\ast\GG$ is the {\it phase space}. The exponential
coordinates $x=\esp(x_1,...,x_n)$, which have been fixed on the
whole $\GG$, induce fiber coordinates on $\TT^\ast\GG$ by
expanding any arbitrary  covector $p\in\TT_x^\ast\GG$ in terms of
the coordinate covector fields $\{dx_1,...,dx_n\}$, i.e.
$p=\sum_{I=1}^n p_I dx_I$. The $2n$ functions
$(x,p)=(x_1,...,x_n,p_1,...,p_n)$ are said {\it canonical
coordinates} on the phase space and the 1-form
$\Theta=\sum_{I=1}^np_I dx_I$ is the {\it tautological one form}
on $\TT_x^\ast\GG$ (which is actually independent of the choice of
coordinates on $\GG$). The {\it canonical symplectic form} on
$\TT^\ast\GG$ is, by definition, the non-degenerate 2-form
$\omega=-d\Theta$. Now, given a function $H$, $\omega$ uniquely
determines a vector field $X_\mathcal{H}$ satisfying
$d\mathcal{H}=\omega(X_\mathcal{H},\cdot)$, which is called the
{\it Hamiltonian vector field} for $\mathcal{H}$ (or symplectic
gradient of $\mathcal{H}$). The {\it Hamilton's equations} for a
smooth Hamiltonian $\mathcal{H}$ are the O.D.E.'s for the integral
curves of $X_\mathcal{H}$. In canonical coordinates, they are
given by
\begin{equation}\label{hameq}\dot{x}_I=\frac{\partial \mathcal{H}}{\partial
p_I},\qquad\dot{p}_I=-\frac{\partial \mathcal{H}}{\partial
x_I}\qquad(I=1,...,n).\end{equation} We remind that the {\it
momentum function} $P_Y:\TT^\ast\GG\longrightarrow\R$ is defined
by $$P_Y(x,p):=p(Y(x)).$$In the sequel $p_I$ will denote the {\it
momentum function} associated to the $I$-th coordinate vector
field\footnote{We are using the notation:$${\partial}/{\partial
x_I}\equiv\ee_I=(0,...,\underbrace{1}_{I-th\, place},...,0)\qquad
(I=1,...,n).$$}\,${\partial}/{\partial x_I}$.     Hence, in
canonical coordinates one has $P_Y(x,p)=\sum_{I=1}^n p_I Y_I(x)$,
where $Y=\sum_{I=1}^n Y_I{\partial}/{\partial x_I}$.

We now start with the derivation of the CC-geodesic equations by
defining the {\it sub-Riemannian Hamiltonian} (or {\it kinetic
energy}) $\mathcal{H}\sr$. To this aim let us set
$$P_i:=P_{X_i}\qquad\mbox{for
every}\,\,  i\in I\cc=\{1,...,\DH\},$$ to denote the momentum
functions associated with a  orthonormal (left-invariant) moving
frame $\underline{X}\cc=\{X_1,...,X_\DH\}$ for $\HH$ and note
that, if $X_i(x)=\sum_{I=1}^n(X_i(x))_I\ee_I,$ one has
$$P_i(x,p)=P_{X_i}(x,p)=\sum_{I=1}^n(X_i(x))_Ip_I\qquad\mbox{for
every}\,\,  i\in I\cc.$$

\begin{no}We shall set \begin{itemize}\item $P\cc:=\sum_{i\in I\cc}P_iX_i$,
\item $P\vv:=\sum_{\alpha\in I\vv}P_\alpha X_\alpha$,\item
$P:=\sum_ IP_IX_I$.\end{itemize}
\end{no}

\begin{Defi}The {\bf sub-Riemannian Hamiltonian} is defined
by$$\mathcal{H}\sr(x,p):=\frac{1}{2}\sum_{i\in I\cc}P_i^2(x,p).$$
The Hamiltonian equations \eqref{hameq} associated with the
sub-Riemannian Hamiltonian $\mathcal{H}\sr$  are called the {\bf
normal geodesics equations}. A {\bf normal curve} is the
projection onto $\GG$ of a solution of the normal geodesics
equations.
\end{Defi}
The following result is well-known and a proof can be found in
\cite{Montgomery}.
\begin{teo}\label{mongeo} Every sufficiently short arc of a normal curve $x\subset\GG$ is a
minimizing CC-geodesic. Moreover $x$ is the unique minimizing
CC-geodesic connecting its endpoints.
\end{teo}\noindent The normal CC-geodesic equations are
given by the following system:

\begin{eqnarray}\label{ngs}\left\{\begin{array}{ll}\dot{x}=P\cc\\\dot{P}=-\sum_{\alpha\in
I\vv} P_\alpha C^\alpha P\cc;\end{array}\right.\end{eqnarray}see
Definition \ref{nota} for the notation $C^\alpha$. Note that the
first equation express the fact that the normal curve is
horizontal. The second equation can be deduced by noting that, for
any function $f:\TT^\ast\GG\longrightarrow\R$ and for any solution
$x:I\subset\R\longrightarrow\GG$ of the Hamiltonian equations, one
has
$$\frac{d}{d t} f(\gamma(t))=\{f,\mathcal{H}\sr\},$$where $\{\cdot,\cdot\}$ denotes Poisson bracket. In particular
$\dot{P_I}=\{P_I,\mathcal{H}\sr\}$  for every $I=1,...,n$.
Therefore, it follows that
$$\dot{P_I}=\{P_I,\frac{1}{2}\sum_{i\in I\cc}P_i^2(x,p)\}=
\frac{1}{2}\sum_{i\in I\cc}\{P_I,P_i\}P_i=-\sum_{i\in
I\cc}\sum_{\alpha\in I\vv} \SC_{Ii}^\alpha P_i P_\alpha.$$The last
computation yields \eqref{ngs}.

Below we shall discuss some other features of the normal geodesic
equations but let us first introduce the following (see
\cite{GoKA}, \cite{Montgomery}):
\begin{Defi}\label{abn}An {\bf abnormal curve} is a horizontal curve which is
the projection onto $\GG$ of an absolutely continuous curve in the
annihilator $\HH^\perp\subset\TT^\ast\GG$ of $\HH$, with square
integrable derivative, which does not intersect the zero section
of $\HH^\perp$ and whose derivative, whenever it exists, is in the
kernel of the canonical symplectic form restricted to $\HH^\perp$.
An {\bf abnormal minimizer} is an abnormal curve which is a
minimizer. A {\bf strictly abnormal curve} -resp. {\bf minimizer}-
is an abnormal curve  -resp. minimizer- which is not normal.
\end{Defi}

For a deep study on abnormal extremals and related problems in
sub-Riemannian geometry, see \cite{Ag1}, \cite{Ag2} and
references therein.

The equations for abnormal curves in our Carnot setting can
explicitly be derived along the lines of \cite{GoKA} (see also
\cite{Montgomery2, Montgomery}). More precisely, they are given by
\begin{eqnarray}\label{angs}\left\{\begin{array}{ll} \sum_{\alpha\in
I\vv}P_\alpha C\cc^\alpha x\cc=0\\ \dot{P\vv}=-\sum_{\alpha\in
I\vv}P_\alpha C^\alpha x\\\dot{x}\vv=0\\
P\cc=0.\end{array}\right.\end{eqnarray}Unlike the system
\eqref{ngs}, the equations of the system \eqref{angs} are mixed
algebraic-differential equations and they cannot be expressed as
O.D.E.'s. Note also that abnormal curves only depends on $\HH$ and
not on the metric.

We remind that, for any $x=\esp(\sum_Ix_I\ee_I)\in\GG$ we set
$x\cc:=(x_1,...,x_\DH)\in\R^{\DH}\cong\HH$ and
$x\vv:=(x_{\DH+1},...,x_n)\in\R^{\vd}\cong\VV$; see Notation
\ref{Janeiro}.

\begin{Defi}\label{NONONO}According to Definition \ref{nota}, we shall
set:
\begin{itemize}\item[(i)]$C\cc(P\vv):=\sum_{\alpha\in I\vv} P_\alpha
C\cc^\alpha$;\item[(ii)]$C(P\vv):=\sum_{\alpha\in I\vv} P_\alpha
C^\alpha$.\end{itemize}
\end{Defi}Using this notation, \eqref{ngs} and \eqref{angs} can be
rewritten more compactly as follows:
\begin{eqnarray}\label{ngs1}\mbox{(Normal Equations)}\left\{\begin{array}{ll}\dot{x}
=P\cc\\\dot{P}=- C(P\vv) P\cc;\end{array}\right.\end{eqnarray}

\begin{eqnarray}\label{angs}\mbox{(Abnormal Equations)}
\left\{\begin{array}{ll} C\cc(P\vv) x\cc=0\\ \dot{P\vv}=-C(P\vv) x\\\dot{x}\vv=0\\
P\cc=0.\end{array}\right.\end{eqnarray}

 In the next examples we shall write down the normal
CC-geodesic equations. Some of them will be studied in greater
detail in the sequel.

\begin{es}[Heisenberg group
$\mathbb{H}^1$]We get
\begin{eqnarray}\label{ngsh1}\left\{\begin{array}{ll}\dot{x}=P\cc\\\dot{P}\cc=-
C\cc(P_3) P\cc\\\dot{P}_3=0.\end{array}\right.\end{eqnarray}More
explicitly, one has $\dot{P}\cc=-P_3 \left|\!\!
\begin{array}{cc}
  0 & 1 \\
  -1 & 0 \\
\end{array}\!\!
\right|P\cc$, where $P\cc=\left|\!\!
\begin{array}{c}
  P_1\\
  P_2\\\end{array}\!\!
\right|$.
 \end{es}

\begin{es}[Heisenberg group
$\mathbb{H}^n$]We get
\begin{eqnarray}\label{ngsh1}\left\{\begin{array}{ll}\dot{x}=P\cc\\\dot{P}\cc=-
C\cc(P_{2n+1})
P\cc\\\dot{P}_{2n+1}=0.\end{array}\right.\end{eqnarray}We have
$$\dot{P}\cc=-P_{2n+1} \left|\!\!\!\!
\begin{array}{cccccccc}
  0 & 1 & 0 & 0 & \cdot &  \cdot &  \cdot &  0 \\
  -1 & 0 & 0 & 0 &  \cdot &  \cdot &  \cdot &  0 \\
  0 & 0 & 0 & 1 &  \cdot & \cdot & \cdot & 0 \\
  0 & 0 & -1 & 0 & \cdot & \cdot & \cdot & 0 \\
  \cdot & \cdot & \cdot & \cdot & \cdot & \cdot & \cdot & \cdot \\
  \cdot & \cdot & \cdot & \cdot & \cdot & \cdot & \cdot & \cdot \\
  0 & 0 & 0 & 0 & \cdot & \cdot & 0 & 1 \\
  0 & 0 & 0 & 0 & \cdot & \cdot & -1 & 0 \!\!\!\!
\end{array}%
\right|P\cc, \qquad P\cc=\left|\!\!
\begin{array}{c}
  P_{1}\\
  P_{2}\\P_{3}\\
  P_{4}\\\vdots\\P_{2n-1}\\
  P_{2n}\end{array}\!\!
\right|.$$
 \end{es}

\begin{es}[2-step Carnot groups]We get
\begin{eqnarray} \label{ngstep2}
\left\{\begin{array}{ll}\dot{x}=P\cc\\\dot{P}\cc=- C\cc(P\cd)
P\cc\\\dot{P}\cd=0,\end{array}\right.\end{eqnarray}where
$C\cc(P\cd)=\sum_{\alpha\in I\cd}P_\alpha C^\alpha\cc.$
\end{es}

\begin{es}[Engel group $\mathbb{E}^1$]We have \begin{eqnarray}\label{ngengel1}
\left\{\begin{array}{ll}\dot{x}=P\cc\\\dot{P}\cc=- \big(P_3 C\cc^3
+ P_4
C^4)P\cc\\\dot{P}_3=-P_4C^4P\cc\\
\dot{P}_4=0.\end{array}\right.\end{eqnarray}
 \end{es}

\begin{es}[3-step Carnot groups]For the general 3-step case
from \eqref{ngs1} we get
\begin{eqnarray}\label{ngstep3}\left\{\begin{array}{ll}\dot{x}=P\cc\\
\dot{P}\cc=- C\cc(P\cd) P\cc\\\dot{P}\cd=- C(P\ctr)
P\cc\\\dot{P}\ctr=0,\end{array}\right.\end{eqnarray}where
$P\cd:=\sum_{\alpha\in I\cd}P_\alpha X_\alpha$,
$P\ctr:=\sum_{\alpha\in I\ctr}P_\alpha X_\alpha$. We also remind
that
\[C\cc(P\cd)=\sum_{\alpha\in I\cd}P_\alpha C^\alpha\cc,\qquad C(P\ctr)=\sum_{\alpha\in I\ctr}P_\alpha
C^\alpha.\]
\end{es}

\subsection{CC-normal geodesics, variational formulae and Jacobi
fields}\label{ipo}

In this section we shall again discuss normal CC-geodesics and
their minimizing properties. We first recall some results which
can be found in \cite{Ver}.

If we look for minimizing geodesics joining two (fixed) points
$x,\,y\in\GG$, we are solving the minimization problem (with fixed
endpoints):
$$\min \int_a^b|\dot{x}\cc|dt,\quad \dot{x}(t)\in\HH_{x(t)}\,\,\forall\, t\in[a,b],\quad x=x(a),\,\,y=x(b).$$
\begin{no}If $Z\in\XG$, then $Z\cc,\,Z\vv$ denote the
orthogonal projection of $Z$ onto $\HH$ and $\VV$,
respectively.\end{no}

The {\it Euler-Lagrange equations} for this problem, written with
respect to\,the global left invariant frame
$\underline{X}=\{X_1,...,X_n\}$, are\footnote{Hereafter the
symbols $\flat$ and $\sharp$ will be used to denote the so-called
{\it musical isomorphisms} between vectors and co-vectors, defined
in an obvious way with the help of the metric; see \cite{Lee}. For
the definition of the ``contraction operator'' $\LL$, see footnote
\ref{fo4}.}
$$\frac{d}{dt}L_{\dot{x}}-L_{x}=\sum_{\alpha\in
I\vv}\big( \dot{P}_{\alpha}\omega_{\alpha} -
P_{\alpha}(\dot{x}\cc\LL d\omega_{\alpha})\big)^{\sharp}$$ where
$L(t,x,\dot{x})=|\dot{x}\cc|$ and $P_\alpha\,(\alpha\in I\vv)$
denote the $\alpha$-th {\it Lagrange multiplier}; see \cite{Ver}.
The horizontality for the curve $x$ is expressed by the equations
$\omega_\alpha(\dot{x})=0\,\,(\alpha\in I\vv)$. These equations
can be rewritten in invariant form by replacing the Lagrangian $L$
of the unconditional problem by $L\sr:=L + \sum_{\alpha\in
I\vv}P_\alpha\omega_{\alpha}(\dot{x}).$ Using the Levi-Civita
connection related to the fixed (left invariant) Riemannian metric
on $\GG$, the normal CC-geodesic equations are given (see
\cite{Ver}) by
\begin{eqnarray}\label{lnge}
\left\{\begin{array}{ll}\nabla_t\dot{x}\cc+\sum_{\alpha\in
I\vv}\big( \dot{P}_{\alpha}\omega_{\alpha} -
P_{\alpha}(\dot{x}\cc\LL
d\omega_{\alpha})\big)^{\sharp}=0\\\omega_\alpha(\dot{x})=0\quad
(\alpha\in I\vv).\end{array}\right.\end{eqnarray}Note that the
second equation is again the horizontality condition which can
equivalently be expressed by $\dot{x}\vv=0$. Setting
$P\vv:=\sum_{\alpha\in I\vv}P_\alpha X_\alpha$ and
$C(P\vv):=\sum_{\alpha\in I\vv}P_\alpha C^\alpha$, the system
\eqref{lnge} can be rewritten as\begin{eqnarray}\label{lnge2}
\left\{\begin{array}{ll}\nabla_t\dot{x}\cc+
\dot{P}\vv+C(P\vv)\dot{x}\cc=0\\\dot{x}\vv=0.\end{array}\right.\end{eqnarray}To
see this it is enough to express the right-hand side of the first
equation in \eqref{lnge} using formula \eqref{dext} for the
exterior derivative $d\omega_\alpha\,(\alpha\in I\vv)$. Note that
the Lagrangian multiplier $P\vv$ can be regarded as a curve in the
vertical subbundle $\VV$. {\it Obviously, \eqref{ngs1} and
\eqref{lnge2} coincide}. This immediately follows by explicitly
calculating the first equation in \eqref{lnge2}. More precisely,
as in the Riemannian case\footnote {See \cite{Ch1},
\cite{Helgason}, \cite{Hicks} for a classical setting, or
\cite{KRP} for a discussion of geodesics equations for
Nonholonomic geometries.}, it turns out that
$$\nabla_t\dot{x}\cc=\sum_{L=1}^n\Big(\frac{d^2x_L}{dt^2}
+ \sum_{i,j\in I\cc}\Gamma_{ij}^L \dot{x}_i\dot{x}_j\Big) X_L,$$
where $\Gamma_{ij}^L:=\langle\nabla_{X_i}X_j,X_L\rangle$ denote
Christoffel Symbols. Using formula \eqref{c2} together with
condition \eqref{chypc}  on the Carnot structure constants, we
 get that $\Gamma_{ij}^L=\frac{1}{2}\SC_{ij}^L=0$
whenever $L\in I\cc\,(i,j\in I\cc)$. Moreover, by skew-symmetry of
the structure constants $\SC_{ij}^L$, for every\footnote{Actually,
by using again \eqref{chypc}, $\SC_{ij}^L$ can be different from 0
only if $L$ belong to $I\cd$.} $L\in I\vv$ we get that
$\sum_{i,j\in I\cc}\Gamma_{ij}^L \dot{x}_i\dot{x}_j=0$. Therefore
$\ddot{x}\cc=\nabla_t\dot{x}\cc.$ Setting $P\cc:=\dot{x}\cc$ and
projecting \eqref{lnge2} onto $\HH$ and $\VV$, respectively,  one
gets\begin{eqnarray}\label{lnge23}
\left\{\begin{array}{ll}\dot{x}=\dot{x}\cc=P\cc\\\dot{x}\vv=0\\\dot{P}\cc=\ddot{x}\cc=\nabla_t\dot{x}\cc=-
C\cc(P\vv)P\cc\\\dot{P}\vv=-C(P\vv)P\cc\end{array}\right.\end{eqnarray}
which is the projected form of \eqref{ngs1}, as we wished to
prove.

As already  said in the previous discussion, the normal
CC-geodesic equations can be directly deduced by minimizing the
constrained Lagrangian
$$L\sr(t,x,\dot{x})=|\dot{x}\cc| + \langle P\vv,\dot{x}
\rangle=|\dot{x}\cc| +\sum_{\alpha\in I\vv} P_\alpha
\omega_\alpha(\dot{x}).$$This is equivalent to compute the
so-called {\it 1st variation} of the functional
$$I\sr(x):=\int_{a}^bL\sr(t,x,\dot{x})dt.$$ For this reason and in order to
 explicitly computing the {\it 2nd variation} of
$I\sr(x)$, we need preliminarily the next:
\begin{Defi}If $x:[a, b]\longrightarrow\GG$ is a smooth path, we
call {\bf variation of $x$} any smooth mapping
$\vartheta:[a,b]\times]-\epsilon_0,\epsilon_0[\longrightarrow\GG$
, $\epsilon_0>0$, for which $x(t)=\vartheta(t,0)$ for all
$t\in[a,b]$. We say that $\vartheta$ {\bf fixes endpoints} if
$x(a)=\vartheta(a,s),\,\,x(b)=\vartheta(b,s)$ for all
$s\in]-\epsilon_0, \epsilon_0[$. In this case we say that
$\vartheta$ is a {\bf homotopy of $x$}. Moreover, if for every
$s\in]-\epsilon_0,\epsilon_0[$ the path $\vartheta_s:[a,
b]\longrightarrow\GG,\, \vartheta_s(t)=\vartheta(t,s),$ is a
normal CC-geodesic, we say that $\vartheta$ is a {\bf
sub-Riemannian} or {\bf CC-geodesic variation of $x$}.
\end{Defi}
In the sequel we shall write $\partial_t, \partial_s$ for
$\vartheta_\ast(\partial_t), \vartheta_\ast(\partial_s)$
respectively, and denote covariant differentiation of vector
fields along $\vartheta$ with respect to $\partial_t, \partial_s$ by
$\nabla_t, \nabla_s$ respectively. Furthermore, let us recall the
following general identities:
\begin{eqnarray}\label{fi1}\nabla_s\partial_t\vartheta &=& \nabla_t\partial_s\vartheta\\\label{fi2}
\nabla_s\nabla_t-\nabla_t\nabla_s&=&R(\partial_t\vartheta,\partial_s\vartheta),\end{eqnarray}where
$\RC$ is the Riemannian curvature tensor of $\GG$.

\begin{Prop}[1{st} variation of $I\sr(x)$]\label{1varg}Let us assume that $x:[a, b]\longrightarrow\GG$ is a differentiable path
and let $\vartheta:[a,b]\times
]-\epsilon_0,\epsilon_0[\longrightarrow\GG$ be a differentiable
variation of $x$. Then \begin{eqnarray}\nonumber \delta^1
I\sr:=\frac{d}{ds}I\sr(\vartheta_s)\Big|_{s=0}=\Big\langle\partial_s\vartheta|_{s=0},
\Big(\frac{\partial_tx}{|\partial_tx|},P\vv\Big)\Big\rangle\Big|_a^b\\\label{ficfic}-\int_a^b\Big\{
\Big\langle\partial_s\vartheta|_{s=0},
\Big[\nabla_t\frac{\partial_tx}{|\partial_tx|} +
\partial_tP\vv+ C(P\vv)\partial_tx\Big]\Big\rangle-
\langle\partial_sP\vv|_{s=0},\partial_tx\rangle\Big\}dt.\end{eqnarray}
Let us assume that $|\partial_tx|=1$ and let $\vartheta$ be a
homotopy of $x$. Setting
\[Y(t):=\partial_s\vartheta(t,0),\quad Q\vv(t):=\partial_s P\vv(t,0)\]
we get that
\begin{eqnarray}\label{Kafka}\delta^1
I\sr=\int_a^b\Big\{ \langle Q\vv,\partial_tx\rangle- \Big\langle
Y, \Big[\nabla_t{\partial_tx}+
\partial_tP\vv+ C(P\vv)\partial_tx\Big]\Big\rangle
\Big\}dt.\end{eqnarray}
\end{Prop}
 Note that \eqref{lnge2} immediately
follows from the previous result once we require that the path $x$ is an
extremal of the functional $I\sr(x)$. Indeed, in such a case, the
fist term under integral sign in \eqref{Kafka} must vanish  for every
$Y$, which gives the first equation  in \eqref{lnge2}. The same
procedure implies that the second term under integral sign vanishes  for every $Q\vv$, which in turn implies the horizontality
condition for $x$.

\begin{proof}[Proof of Proposition \ref{1varg}]First we have to compute
$$\partial_s\int_a^b\Big\{|\partial_t{\vartheta}\cc| + \langle P\vv,\partial_t{\vartheta}
\rangle\Big\}dt= \underbrace{
\partial_s\int_a^b|\partial_t{\vartheta}\cc|dt}_{=:A_1} +
\underbrace{\partial_s\int_a^b\langle P\vv,\partial_t{\vartheta}
\rangle dt}_{=:A_2}.$$Actually, the first term $A_1$ can be
directly deduced from the Riemannian computation (see, for
instance, \cite{Ch1}, Theorem 2.3) and, with our notation, we get
that
\begin{eqnarray*}A_1=\Big\langle\partial_s\vartheta,
\frac{\partial_tx}{|\partial_tx|}\Big\rangle\Big|_a^b-\int_a^b
\Big\langle\partial_s\vartheta,
\nabla_t\frac{\partial_tx}{|\partial_tx|} \Big\rangle
dt.\end{eqnarray*}So we have to compute the other term. We have
\begin{eqnarray*}A_2&=&\partial_s\int_a^b\langle P\vv,\partial_t{\vartheta}
\rangle dt=\int_a^b\partial_s\langle P\vv,\partial_t{\vartheta}
\rangle dt\\&=&\int_a^b\Big\{\langle\nabla_s
P\vv,\partial_t{\vartheta}\rangle + \langle
P\vv,\nabla_s\partial_t{\vartheta}\rangle\Big\}
dt\\&=&\int_a^b\Big\{\langle\partial_s P\vv +
\nabla_{\partial_s\vartheta}P\vv,\partial_t{\vartheta}\rangle +
\langle P\vv,\nabla_t\partial_s{\vartheta}\rangle\Big\}
dt\\&=&\int_a^b\Big\{\langle\partial_s P\vv
,\partial_t{\vartheta}\rangle +
\langle\nabla_{\partial_s\vartheta}P\vv
,\partial_t{\vartheta}\rangle+ \partial_t\langle
P\vv,\partial_s{\vartheta}\rangle-\langle \nabla_t
P\vv,\partial_s{\vartheta}\rangle\Big\}
dt\\&=&\int_a^b\Big\{\langle\partial_s P\vv
,\partial_t{\vartheta}\rangle +
\langle\nabla_{\partial_s\vartheta}P\vv
,\partial_t{\vartheta}\rangle+ \partial_t\langle
P\vv,\partial_s{\vartheta}\rangle-\langle
\partial_tP\vv+ \nabla_{\partial_t\vartheta}
P\vv,\partial_s{\vartheta}\rangle\Big\} dt\\&=&\langle
\partial_s{\vartheta},P\vv\rangle\big|_a^b+
\int_a^b\Big\{\langle\partial_s P\vv ,\partial_t{\vartheta}\rangle
-\langle\partial_tP\vv,\partial_s{\vartheta}\rangle+
\underbrace{\langle\nabla_{\partial_s\vartheta}P\vv
,\partial_t{\vartheta}\rangle-\langle \nabla_t
P\vv,\partial_s{\vartheta}\rangle}_{=:B}\Big\}dt.\end{eqnarray*}
We claim that $B=\langle[\partial_s\vartheta,
\partial_t{\vartheta}], P\vv\rangle$. To prove this claim we may proceed
by using a well-known formula for the Riemannian connection which
can be found in \cite{Ch1} (see  formula 1.29, p.15). This way we
get
\begin{eqnarray*}
B&=&\langle\nabla_{\partial_s\vartheta}P\vv
,\partial_t{\vartheta}\rangle-\langle \nabla_t
P\vv,\partial_s{\vartheta}\rangle\\&=&\frac{1}{2}\Big\{\langle[\partial_s\vartheta,
P\vv],\partial_t{\vartheta}\rangle-\langle[P\vv,\partial_t{\vartheta}],\partial_s{\vartheta}\rangle
+\langle[\partial_t{\vartheta},\partial_s{\vartheta}]\Big\}\\&-&\frac{1}{2}\Big\{\langle[\partial_t{\vartheta},
P\vv],\partial_s\vartheta\rangle-\langle[P\vv,\partial_s{\vartheta}],\partial_t{\vartheta}\rangle
+\langle[\partial_s{\vartheta}, \partial_t{\vartheta}],
P\vv\rangle\Big\}\\&=&\langle[\partial_s\vartheta,
\partial_t{\vartheta}], P\vv\rangle.
\end{eqnarray*}Moreover, it is easy to show that $\langle[\partial_s\vartheta,
\partial_t{\vartheta}], P\vv\rangle=-\langle
C(P\vv)\partial_tx,\partial_s\vartheta\rangle$. Indeed one has
$$\langle[\partial_s\vartheta,
\partial_t{\vartheta}], P\vv\rangle=\sum_{I, J=1}^n\sum_{\alpha\in I\vv}(\partial_s\vartheta)_I(\partial_t{\vartheta})_JP_\alpha\langle[X_I,
X_J], X_\alpha\rangle=\sum_{I, J=1}^n\sum_{\alpha\in
I\vv}\SC_{IJ}^\alpha(\partial_s\vartheta)_I(\partial_t{\vartheta})_JP_\alpha$$and
the claim immediately follows by using (ii) of Definition
\ref{NONONO}. So we have shown that
\begin{equation}\label{terminsec}A_2=\langle
\partial_s{\vartheta},P\vv\rangle\big|_a^b+
\int_a^b\Big\{\langle\partial_s P\vv ,\partial_t{\vartheta}\rangle
-\langle\big( \partial_t P\vv+
C(P\vv)\partial_t\vartheta\big),\partial_s\vartheta\rangle\Big\}dt.\end{equation}By
adding $A_1$ and $A_2$  \eqref{ficfic} claim easily follows.
Finally, \eqref{Kafka} follows from \eqref{ficfic}.
\end{proof}

\begin{teo}[2{nd}  Variation of $I\sr(x)$]\label{sevate}Under the notation of Proposition \ref{1varg} let  $x$ be a normal CC-geodesic satisfying $|\partial_tx|=1$,
and assume that $\vartheta$ is a homotopy of $x$. Moreover, set
$Y(t):=\partial_s\vartheta(t,0)$ and $Q\vv(t):=\partial_s
P\vv(t,0)$. Then we have

\begin{eqnarray}\nonumber\delta^2I\sr
=\int_a^b\Big\{2\Big\langle Q\vv, \Big[\nabla_tY
-\frac{3}{4}[Y,\partial_t x] -\frac{1}{4}
C(Y\vv)\partial_tx\Big]\Big\rangle\\\label{secvarfor}-\Big\langle
Y, \Big[ \nabla_t^{(2)}{Y\cc}+ C(P\vv)\big(\nabla_tY\cc + [Y,
\partial_tx]\big)+[Y,
\partial_tx]+\RC(
\partial_tx,Y)\partial_tx\Big]\Big\rangle\Big\}dt\end{eqnarray}

\end{teo}

\begin{proof}Since, by hypothesis, $x$ is a normal CC-geodesic and
$\vartheta$ is a homotopy of $x$, we may start by considering the
following identity:
\begin{eqnarray}\label{KaKa}\frac{d}{ds}I\sr(\vartheta_s)=-\int_a^b\Big\{ \Big\langle \partial_s\vartheta, \Big[\nabla_t{\partial_t\vartheta\cc}+
\partial_tP\vv+ C(P\vv)\partial_t\vartheta\cc\Big]\Big\rangle-
\langle
\partial_sP\vv,\partial_t\vartheta\rangle\Big\}dt.\end{eqnarray}
This identity can easily be deduced by the proof of Proposition
\ref{1varg}. So we have
\begin{eqnarray*}\frac{d^2}{ds^2}I\sr(\vartheta_s)
&=&-\partial_s\int_a^b\Big\{ \Big\langle \partial_s\vartheta,
\Big[\nabla_t{\partial_t\vartheta\cc}+
\partial_tP\vv+ C(P\vv)\partial_t\vartheta\cc\Big]\Big\rangle-
\langle
\partial_sP\vv,\partial_t\vartheta\rangle\Big\}dt\\&=&-\int_a^b\Big\{ \partial_s\Big\langle \partial_s\vartheta,
\Big[\nabla_t{\partial_t\vartheta\cc}+
\partial_tP\vv+ C(P\vv)\partial_t\vartheta\cc\Big]\Big\rangle-
\partial_s\langle
\partial_sP\vv,\partial_t\vartheta\rangle\Big\}dt\\&=&-\int_a^b\Big\{
 \underbrace{\Big\langle \nabla_s\partial_s\vartheta,
\Big[\nabla_t{\partial_t\vartheta\cc}+
\partial_tP\vv+ C(P\vv)\partial_t\vartheta\cc\Big]\Big\rangle}_{=:A_1}
\\&&+\underbrace{\Big\langle \partial_s\vartheta,
\nabla_s\Big[\nabla_t{\partial_t\vartheta\cc}+
\partial_tP\vv+ C(P\vv)\partial_t\vartheta\cc\Big]\Big\rangle}_{=:A_2}-\underbrace{
\partial_s\langle
\partial_sP\vv,\partial_t\vartheta\rangle}_{=:A_3}
\Big\}dt.\end{eqnarray*}Now it is obvious that, at $s=0$, one has
$A_1=0$ because $x$ is assumed to be a normal
CC-geodesic.\\\\\noindent{\bf Step 1.}\,{{(Computation of
$A_2$)}}\, {\it We have
}\,\,\begin{eqnarray*}\nabla_s\big(\nabla_t{\partial_t\vartheta\cc}+
\partial_tP\vv+ C(P\vv)\partial_t\vartheta\cc\big)\big|_{s=0}&=&
\nabla_t^{(2)}{Y\cc}+[Y, \partial_tx]+\RC(
\partial_tx,Y)\partial_tx+
\nabla_tQ\vv \\&& + C(Q\vv)\partial_tx+
C(P\vv)\big(\nabla_tY\cc + [Y,
\partial_tx]\big).\end{eqnarray*}
\begin{proof}This computation
 generalizes the classical deduction of Jacobi's
equation. We have
\begin{eqnarray*}B:=\nabla_s\big(\nabla_t{\partial_t\vartheta\cc}+
\partial_tP\vv+ C(P\vv)\partial_t\vartheta\cc\big)=
\underbrace{\nabla_s\nabla_t{\partial_t\vartheta\cc}}_{=:B_1}+
\underbrace{\nabla_s\partial_tP\vv}_{=:B_2}+
\underbrace{\nabla_s\big(C(P\vv)\partial_t\vartheta\cc\big)}_{=:B_3}.\end{eqnarray*}
The first term $B_1$ can be computed in analogy with the
Riemannian case; see \cite{DC}, p.111. More precisely, by
\eqref{fi2} we
get$$B_1=\nabla_s\nabla_t{(\partial_t\vartheta)\cc}=
 \nabla_t\nabla_s{(\partial_t\vartheta)\cc}-\RC(\partial_s\vartheta,
  \partial_t\vartheta)(\partial_t\vartheta)\cc.$$
At this point we have to compute the term
$\nabla_s(\partial_t\vartheta)\cc$. Note that this must be done in
a different way with respect to the Riemannian case because the term
$\partial_t\vartheta$ is not horizontal, a priori, and so it may
be $\partial_t\vartheta\neq
\partial_t\vartheta\cc$. We claim that
\begin{eqnarray}\label{vigao}\nabla_s(\partial_t\vartheta)\cc=\nabla_t(\partial_s\vartheta)\cc
+ [\partial_s\vartheta, (\partial_t\vartheta)\cc].\end{eqnarray}
Indeed we have
\begin{eqnarray*}\nabla_s(\partial_t\vartheta)\cc&=&
\nabla_s(\PH(\partial_t\vartheta))=\partial_s(\PH(\partial_t\vartheta))
+\nabla_{\partial_s\vartheta}\PH(\partial_t\vartheta)\\&=&\partial_t(\PH(\partial_s\vartheta))
+ \sum_{j\in
I\cc}\sum_{I=1}^n(\partial_t\vartheta)_j(\partial_s\vartheta)_I(\nabla_{X_I}X_j)|_\vartheta\\
&=&\partial_t(\PH(\partial_s\vartheta)) + \sum_{j\in
I\cc}\sum_{I=1}^n(\partial_t\vartheta)_j(\partial_s\vartheta)_I\Big(\nabla_{X_j}X_I
+ \sum_{\alpha\in I\vv}\SC_{Ij}^\alpha
X_\alpha\Big)\Big|_\vartheta\\&=&\nabla_t\PH(\partial_s\vartheta)
+ \sum_{\alpha\in I\vv}\Big(\langle
C^\alpha\PH(\partial_t\vartheta),\partial_s\vartheta\rangle
X_\alpha\Big)\Big|_\vartheta\\&=&\nabla_t(\partial_s\vartheta)\cc
+ \sum_{\alpha\in I\vv}\langle
C^\alpha(\partial_t\vartheta)\cc,\partial_s\vartheta\rangle
X_\alpha(\vartheta)\\&=&\nabla_t(\partial_s\vartheta)\cc +
[\partial_s\vartheta, (\partial_t\vartheta)\cc],\end{eqnarray*}and
\eqref{vigao} follows. Therefore we obtain
\begin{eqnarray}\label{B_1}B_1=\nabla_s\nabla_t{(\partial_t\vartheta)\cc}=
\nabla_t\nabla_t{(\partial_s\vartheta)\cc}+[\partial_s\vartheta,
(\partial_t\vartheta)\cc]+\RC(
\partial_t\vartheta,\partial_s\vartheta)(\partial_t\vartheta)\cc.
\end{eqnarray}

\begin{oss}
In what above, we have used the identity
\[\partial_s(\PH(\partial_t\vartheta))=
\partial_t(\PH(\partial_s\vartheta)).\]In order to prove this identity, we may
argue componentwise. More precisely, for $i\in I\cc$ we compute

\begin{eqnarray*}\partial_s\langle\partial_t\vartheta,X_i\rangle&=&
\langle\partial_s\partial_t\vartheta,X_i\rangle+
\langle\partial_t\vartheta,\partial_sX_i\rangle=
\langle\partial_s\partial_t\vartheta,X_i\rangle+
\langle\partial_t\vartheta,\mathcal{J}
X_i\partial_s\vartheta\rangle\\
\partial_t\langle\partial_s\vartheta,X_i\rangle&=&
\langle\partial_t\partial_s\vartheta,X_i\rangle+
\langle\partial_s\vartheta,\partial_tX_i\rangle=
\langle\partial_t\partial_s\vartheta,X_i\rangle+
\langle\partial_s\vartheta,\mathcal{J}
X_i\partial_t\vartheta\rangle.
\end{eqnarray*}By subtracting the second identity from the first
one, we easily get that
\begin{eqnarray*}A_i:=\partial_s\langle\partial_t\vartheta,X_i\rangle-
\partial_t\langle\partial_s\vartheta,X_i\rangle=\langle\partial_t\vartheta,\big[\mathcal{J}
X_i-(\mathcal{J}X_i)^{\rm
Tr}\big]\partial_s\vartheta\rangle\rangle.
\end{eqnarray*}From the last relation we infer that $$A_i=\sum_{I,J=1}^n\langle\big[\mathcal{J}
X_i-(\mathcal{J}X_i)^{\rm
Tr}\big]X_I,X_J\rangle(\partial_s\vartheta)_I(\partial_t\vartheta)_J.$$Now
since
\begin{eqnarray*}\langle\mathcal{J}X_iX_I,X_J\rangle&=&-\langle\nabla_{X_I}X_J,X_i\rangle,\\\langle(\mathcal{J}X_i)^{\rm
Tr}X_I,X_J\rangle&=&\langle(\mathcal{J}X_i)^{\rm
Tr}X_J,X_I\rangle=-\langle\nabla_{X_J}X_I,X_i\rangle,\end{eqnarray*}
we get that
$$A_i=\sum_{I,J=1}^n\langle\nabla_{X_J}X_I-
\nabla_{X_I}X_J\rangle(\partial_s\vartheta)_I(\partial_t\vartheta)_J=\sum_{I,J=1}^n\underbrace
{\SC_{JI}^i}_{=0\,\mbox{\tiny{by
\eqref{chypc}}}}(\partial_s\vartheta)_I(\partial_t\vartheta)_J=0$$for
every $i\in I\cc$, which implies the claim. \end{oss}

\begin{oss}\label{ossavate}If the variation
$\vartheta$ is a CC-geodesic variation of $x$, the term $B_1$ can
be computed as in the Riemannian case; see \cite{DC}, p.111.
Indeed, in such a case we have
$\partial_t\vartheta=\PH(\partial_t\vartheta)$  for every $s\in
]-\epsilon_0, \epsilon_0[$. Using \eqref{fi1} and \eqref{fi2}
yields
\begin{eqnarray}\nonumber B_1&=&\nabla_s\nabla_t{
(\partial_t\vartheta)\cc}=\nabla_s\nabla_t{(\partial_t\vartheta)}
=\nabla_t\nabla_s{(\partial_t\vartheta)}-\RC(\partial_s\vartheta,
  \partial_t\vartheta)\partial_t\vartheta)\\&=&
\nabla_t\nabla_t{(\partial_s\vartheta)}+\RC(\partial_t\vartheta,
  \partial_s\vartheta)\partial_t\vartheta.\end{eqnarray}This can also be
   deduced from \eqref{B_1}, by noting that
  \[[\partial_s\vartheta,
(\partial_t\vartheta)\cc]=[\partial_s\vartheta,
\partial_t\vartheta]=\vartheta_\ast[\partial_s,\partial_t]=0,\]
because $[\partial_s,\partial_t]=0$.\end{oss} \noindent The term
$B_2$ can be computed by means of \eqref{fi1} and we immediately
get
that\begin{eqnarray}\label{B_2}B_2=\nabla_s\partial_tP\vv=\nabla_t\partial_sP\vv.
\end{eqnarray}
\noindent Analogously, the term $B_3$ can be computed by means of
\eqref{fi1}, \eqref{fi2} and \eqref{vigao}, as
follows:\begin{eqnarray*}\label{B_3}B_3&=&\nabla_s(C(P\vv)\partial_t
\vartheta\cc)=\nabla_s\Big(\sum_{\alpha\in I\vv}P_\alpha
C^\alpha\partial_t \vartheta\cc\Big)=\sum_{\alpha\in
I\vv}\big(\partial_sP_\alpha C^\alpha\partial_t\vartheta\cc +
P_\alpha C^\alpha\nabla_s\partial_t\vartheta\cc
\big)\\&=&C(\partial_sP\vv)\partial_t\vartheta\cc+
C(P\vv)\nabla_s\partial_t\vartheta\cc=C(\partial_sP\vv)\partial_t\vartheta\cc
+ C(P\vv)\big(\nabla_t(\partial_s\vartheta)\cc +
[\partial_s\vartheta, (\partial_t\vartheta)\cc]\big).
\end{eqnarray*}
Finally, by adding the terms $B_1, B_2$ and $B_3$, we get
\begin{eqnarray*}B&=&\nabla_t\nabla_t{(\partial_s\vartheta)\cc}+[\partial_s\vartheta,
(\partial_t\vartheta)\cc]+\RC(
\partial_t\vartheta,\partial_s\vartheta)(\partial_t\vartheta)\cc +
\nabla_t\partial_sP\vv\\ &&+
C(\partial_sP\vv)\partial_t\vartheta\cc +
C(P\vv)\big(\nabla_t(\partial_s\vartheta)\cc +
[\partial_s\vartheta, (\partial_t\vartheta)\cc]\big).
\end{eqnarray*}
The thesis  follows by substituting
$Y(t)=\partial_s\vartheta(t,0)$ and $Q\vv(t)=\partial_s P\vv(t,0)$
into the last expression.\end{proof} \noindent{\bf Step 2.}\,{\it
We have}\,\,$\big\{\int_a^b A_3 dt\big\}\big|_{s=0}= -\int_a^b
\big\langle Y, \big[
\partial_t Q\vv + C(Q\vv)\partial_t x\big]\big\rangle dt.$ \begin{proof}By arguing as in Proposition \ref{1varg},
 we get that
\begin{eqnarray*}\int_a^b A_3 dt=\langle \partial_sP\vv, \partial_s\vartheta\rangle\big|_a^b + \int_a^b
\Big\{\Big\langle\frac{\partial^2 P\vv}{\partial
s^2},\partial_t\vartheta\Big\rangle -\Big\langle
\partial_s\vartheta, \big[ C(\partial_sP\vv)\partial_t x +
\partial_t \partial_sP\vv\big]\Big\rangle\Big\}dt.\end{eqnarray*}
Now we consider this expression at $s=0$ and we make the
substitutions: $\partial_t\vartheta|_{s=0}=\partial_tx$,
$\partial_s\vartheta|_{s=0}=Y$ and $\partial_sP\vv|_{s=0}=Q\vv$.
By hypothesis, $\vartheta$ is a homotopy. Hence
$$\langle Q\vv, Y\rangle\big|_a^b=\langle \partial_sP\vv|_{s=0},
\partial_s\vartheta|_{s=0}\rangle\big|_a^b=0.$$Set $s=0$. Then $\big\langle\frac{\partial^2 P\vv}{\partial
s^2},\partial_t\vartheta\big\rangle\big|_{s=0}=\big\langle\frac{\partial^2
P\vv}{\partial s^2}\big|_{s=0},\partial_tx\big\rangle=0$ and the
thesis follows.
\end{proof}\noindent Using  Step 1 and Step 2, we get that\begin{eqnarray*}\delta^2I\sr=\frac{d^2}{ds^2}I\sr(\vartheta_s)\Big|_{s=0}
=-\int_a^b\Big\{\Big\langle Y, \Big( \nabla_t^{(2)}{Y\cc}+[Y,
\partial_tx]+\RC(
\partial_tx,Y)\partial_tx \\+
\nabla_tQ\vv + C(Q\vv)\partial_tx+ C(P\vv)\big(\nabla_tY\cc +
[Y,
\partial_tx]\big)+\big(\partial_t Q\vv+
C(Q\vv)\partial_t x
\big)\Big)\Big\rangle\Big\}dt.\end{eqnarray*}Since
$\nabla_tQ\vv=\partial_tQ\vv + \nabla_{\partial_t x}Q\vv$, it
follows that
\begin{eqnarray*}\delta^2I\sr &=&-\int_a^b\Big\{\Big\langle Y,
\Big[ \nabla_t^{(2)}{Y\cc}+[Y,
\partial_tx]+\RC(
\partial_tx,Y)\partial_tx \\&&+
2\Big(\nabla_tQ\vv + C(Q\vv)\partial_tx
-\frac{1}{2}\nabla_{\partial_tx}Q\vv\Big)+
C(P\vv)\big(\nabla_tY\cc + [Y,
\partial_tx]\big)\Big]\Big\rangle\Big\}dt.\end{eqnarray*}By
reordering a bit the last expression we easily get that
\begin{eqnarray}\nonumber\delta^2I\sr
&=&-\int_a^b\Big\{\Big\langle Y, \Big[ \nabla_t^{(2)}{Y\cc}+
C(P\vv)\big(\nabla_tY\cc + [Y,
\partial_tx]\big)+[Y,
\partial_tx]+\RC(
\partial_tx,Y)\partial_tx\Big]\Big\rangle\\\label{presecvaruni}&&+2\Big\langle Y, \Big[\nabla_tQ\vv
+ C(Q\vv)\partial_tx-\frac{1}{2}\nabla_{\partial_tx}Q\vv
\Big]\Big\rangle\Big\}dt.\end{eqnarray}

\noindent{\bf Step 3.}\,{\it The following identities hold
true:}\begin{itemize}\item[{\it
(i)}]$\langle\nabla_{\partial_tx}Q\vv, Y\rangle=\frac{1}{2}\big(
\langle C(Y\vv)Q\vv,\partial_tx\rangle + \langle
C(Q\vv)\partial_tx, Y\rangle\big)$;\item[{\it (ii)}]$\langle
C(Q\vv)\partial_t x,Y\rangle=-\langle[\partial_t x,Y],
Q\vv\rangle$.\end{itemize}

\begin{proof}To prove the first identity we compute
\begin{eqnarray*}
\langle\nabla_{\partial_tx}Q\vv, Y\rangle&=&\sum_{i\in
I\cc}\sum_{\alpha\in I\vv}\sum_{J=1}^n(\partial_t x)_iQ_\alpha Y_J
\langle\nabla_{X_i}X_\alpha,
X_J\rangle\\&=&\frac{1}{2}\,\sum_{i\in I\cc}\sum_{\alpha\in
I\vv}\sum_{J=1}^n(\partial_t x)_iQ_\alpha
Y_J\big(\SC_{i\alpha}^J-\underbrace{\SC_{\alpha
J}^i}_{=0\,\mbox{\rm \tiny by \eqref{chypc}}} +
\SC_{Ji}^{\alpha}\big)\\&=&\frac{1}{2}\Big( \langle
C(Y\vv)Q\vv,\partial_tx\rangle + \langle C(Q\vv)\partial_tx,
Y\rangle\Big),
\end{eqnarray*}which proves (i), while (ii) follows from Definition \ref{NONONO} and Lemma \ref{twist1}.
\end{proof}\noindent Using Step 3 and integrating by parts, yields
\begin{eqnarray*}&&\int_a^b \Big\langle Y, \Big[\nabla_tQ\vv
+ C(Q\vv)\partial_tx-\frac{1}{2}\nabla_{\partial_tx}Q\vv
\Big]\Big\rangle dt\\&=&\underbrace{\langle Y,
Q\vv\rangle\big|_a^b}_{=0}+\int_a^b \Big\{-\langle \nabla_tY,Q\vv
\rangle -\langle[\partial_t x,Y], Q\vv\rangle
-\frac{1}{4}\big[\langle C(Y\vv)Q\vv,\partial_tx\rangle + \langle
C(Q\vv)\partial_tx,
Y\rangle\big]\Big\}dt\end{eqnarray*}\begin{eqnarray*}&=&\int_a^b
\Big\{-\langle \nabla_tY,Q\vv \rangle
-\frac{3}{4}\langle[\partial_t x,Y], Q\vv\rangle
-\frac{1}{4}\langle C(Y\vv)Q\vv,\partial_tx\rangle
\Big\}dt\\&=&-\int_a^b \Big\{\langle \nabla_tY,Q\vv \rangle
+\frac{3}{4}\langle[\partial_t x,Y], Q\vv\rangle
-\frac{1}{4}\langle C(Y\vv)\partial_tx,Q\vv\rangle
\Big\}dt\\&=&-\int_a^b \Big\langle Q\vv, \Big[\nabla_tY
+\frac{3}{4}[\partial_t x,Y] -\frac{1}{4}
C(Y\vv)\partial_tx\Big]\Big\rangle dt.\end{eqnarray*}Finally, by
using the last expression and \eqref{presecvaruni}, we get that

\begin{eqnarray*}\delta^2I\sr
&=&-\int_a^b\Big\{\Big\langle Y, \Big[ \nabla_t^{(2)}{Y\cc}+
C(P\vv)\big(\nabla_tY\cc + [Y,
\partial_tx]\big)+[Y,
\partial_tx]+\RC(
\partial_tx,Y)\partial_tx\Big]\Big\rangle\\&&-2\Big\langle Q\vv, \Big[\nabla_tY
-\frac{3}{4}[Y,\partial_t x] -\frac{1}{4}
C(Y\vv)\partial_tx\Big]\Big\rangle\Big\}dt,\end{eqnarray*}which is
equivalent to the thesis.
\end{proof}

Accordingly with the Hamiltonian theory already discussed in
Section \ref{onnorm}, we set $P\cc:=\partial_t x$. Starting from
the second variation formula \eqref{secvarfor}, it is natural to
consider the following system of O.D.E.'s:
\begin{eqnarray}\label{lnge}
\left\{\begin{array}{ll}\nabla_t^{(2)}{Y\cc}+
C(P\vv)\big(\nabla_tY\cc + [Y, P\cc]\big)+[Y,
 P\cc]+\RC(
 P\cc,Y) P\cc=0\\\\
\P\vv\big(\nabla_tY -\frac{3}{4}[Y,P\cc] -\frac{1}{4} C(Y\vv)
P\cc\big)=0.\end{array}\right.\end{eqnarray}

\begin{corollario}[2{nd} derivative of $I\sr(x)$ through CC-geodesic variations]\label{2sevate}Under the notation of Proposition \ref{1varg} let  $x$ be a normal CC-geodesic satisfying $|\partial_tx|=1$,
and let us assume that:\begin{itemize}\item[(i)]$\vartheta$ is a
homotopy of $\,x$;\item[(ii)] $\vartheta$ is a CC-geodesic
variation of $x$.\end{itemize} Moreover, set
$Y(t):=\partial_s\vartheta(t,0)$ and $Q\vv(t):=\partial_s
P\vv(t,0)$. Then \begin{eqnarray}\nonumber\delta^2I\sr
=\int_a^b\Big\{\Big\langle Q\vv, \Big( \nabla_tY +
[\partial_tx,Y]\Big) \Big\rangle-\Big\langle Y, \Big(
\nabla_t^{(2)}{Y}+\RC(
\partial_tx,Y)\partial_tx +
C(P\vv)\nabla_tY\Big)\Big\rangle\Big\}dt.\\\label{coseva}
\end{eqnarray}
\end{corollario}

\begin{proof}The proof follows the same lines of the general case.
Since, by hypothesis, $x$ is a normal CC-geodesic and the
CC-geodesic variation $\vartheta$ is also a homotopy of $x$, we
may start by considering the identity
\begin{eqnarray}\label{KaKa}\frac{d}{ds}I\sr(\vartheta_s)=-\int_a^b\Big\{ \Big\langle \partial_s\vartheta, \Big[\nabla_t{\partial_t\vartheta}+
\partial_tP\vv+ C(P\vv)\partial_t\vartheta\Big]\Big\rangle\Big\}dt.\end{eqnarray}
Indeed note that $\partial_t\vartheta=\partial_t\vartheta\cc$ for
every $(t,s)\in[a,b]\times ]-\epsilon_0,\epsilon_0[$. This
identity can easily be deduced, as in the proof of Proposition
\ref{1varg},  by using the hypotheses on $x$ and $\vartheta$.
Hence
\begin{eqnarray*}\frac{d^2}{ds^2}I\sr(\vartheta_s)
&=&-\partial_s\int_a^b\Big\{ \Big\langle \partial_s\vartheta,
\Big[\nabla_t{\partial_t\vartheta}+
\partial_tP\vv+ C(P\vv)\partial_t\vartheta\Big]\Big\rangle\Big\}dt
\\&=&-\int_a^b\Big\{ \partial_s\Big\langle \partial_s\vartheta,
\Big[\nabla_t{\partial_t\vartheta}+
\partial_tP\vv+ C(P\vv)\partial_t\vartheta\Big]\Big\rangle\Big\}dt\\&=&-\int_a^b\Big\{
 \underbrace{\Big\langle \nabla_s\partial_s\vartheta,
\Big[\nabla_t{\partial_t\vartheta}+
\partial_tP\vv+ C(P\vv)\partial_t\vartheta\Big]\Big\rangle}_{=:A_1}
\\&&+\underbrace{\Big\langle \partial_s\vartheta,
\nabla_s\Big[\nabla_t{\partial_t\vartheta}+
\partial_tP\vv+ C(P\vv)\partial_t\vartheta\Big]\Big\rangle}_{=:A_2}
\Big\}dt.\end{eqnarray*}Now it is obvious that, at $s=0$, one has
$A_1=0$, because $x$ is a normal CC-geodesic. So we need to prove
the following fact:\\\noindent{\bf Step 1.}\,{{(Computation of
$A_2$)}}\, {\it One has
}\,\,\begin{eqnarray*}\nabla_s\big(\nabla_t{\partial_t\vartheta}+
\partial_tP\vv+ C(P\vv)\partial_t\vartheta\big)\big|_{s=0}\\=
\nabla_t^{(2)}{Y}+\RC(
\partial_tx,Y)\partial_tx +
\nabla_tQ\vv + C(Q\vv)\partial_tx+
C(P\vv)\nabla_tY.\end{eqnarray*}
\begin{proof}The proof mimics that of
Theorem \ref{sevate}. We have

\begin{eqnarray*}B:=\nabla_s\big(\nabla_t{\partial_t\vartheta}+
\partial_tP\vv+ C(P\vv)\partial_t\vartheta\big)=
\underbrace{\nabla_s\nabla_t{\partial_t\vartheta}}_{=:B_1}+
\underbrace{\nabla_s\partial_tP\vv}_{=:B_2}+
\underbrace{\nabla_s\big(C(P\vv)\partial_t\vartheta\big)}_{=:B_3}.\end{eqnarray*}
The first term $B_1$ can be computed exactly as in the Riemannian
case (see \cite{DC}, p.111 or the previous Remark \ref{ossavate}).
More precisely, by using \eqref{fi1} and \eqref{fi2}, we get
$$B_1=\nabla_s\nabla_t{(\partial_t\vartheta)}=
 \nabla_t\nabla_s{(\partial_t\vartheta)}-\RC(\partial_s\vartheta,
  \partial_t\vartheta)(\partial_t\vartheta).$$
Also $B_2$  can be computed exactly as in the previous proof of
Theorem \ref{sevate} (see Step 1) and we get
that\begin{eqnarray*}B_2=\nabla_s\partial_tP\vv=\nabla_t\partial_sP\vv.
\end{eqnarray*}Analogously, for the term $B_3$ we get
\begin{eqnarray*}B_3&=&\nabla_s(C(P\vv)\partial_t
\vartheta)=\nabla_s\Big(\sum_{\alpha\in I\vv}P_\alpha
C^\alpha\partial_t \vartheta\Big)=\sum_{\alpha\in
I\vv}\big(\partial_sP_\alpha C^\alpha\partial_t\vartheta +
P_\alpha C^\alpha\nabla_s\partial_t\vartheta
\big)\\&=&C(\partial_sP\vv)\partial_t\vartheta+
C(P\vv)\nabla_s\partial_t\vartheta=C(\partial_sP\vv)\partial_t\vartheta
+ C(P\vv)\nabla_t(\partial_s\vartheta).
\end{eqnarray*}
Finally, by adding the terms $B_1, B_2$ and $B_3$ we have
\begin{eqnarray*}B=\nabla_t\nabla_t{(\partial_s\vartheta)}+\RC(
\partial_t\vartheta,\partial_s\vartheta)(\partial_t\vartheta) +
\nabla_t\partial_sP\vv + C(\partial_sP\vv)\partial_t\vartheta +
C(P\vv)\nabla_t(\partial_s\vartheta) .
\end{eqnarray*}
The thesis follows by substituting $Y(t)=\partial_s\vartheta(t,0)$
and $Q\vv(t)=\partial_s P\vv(t,0)$ into the last
expression.\end{proof}\noindent By what proved in Step 1 we get
that\begin{eqnarray*}\delta^2I\sr&=&\frac{d^2}{ds^2}I\sr(\vartheta_s)\Big|_{s=0}
\\&=&-\int_a^b\Big\{\Big\langle Y, \Big( \nabla_t^{(2)}{Y}+\RC(
\partial_tx,Y)\partial_tx +
\nabla_tQ\vv + C(Q\vv)\partial_tx+
C(P\vv)\nabla_tY\Big)\Big\rangle\Big\}dt.\end{eqnarray*}Now, by
arguing again as in proof of Theorem \ref{sevate}, we obtain
\begin{eqnarray*}
\int_a^b\big\langle Y, \big(\nabla_tQ\vv +
C(Q\vv)\partial_tx\rangle dt=\langle Y , Q\vv\rangle|_{a}^b-
\int_a^b\big\langle\big( \nabla_tY + [\partial_tx,Y]\big), Q\vv
\big\rangle dt\end{eqnarray*}and since $\vartheta$ is a homotopy
of $x$ we have $\langle Y , Q\vv\rangle|_{a}^b=0$. Putting all
together, \eqref{coseva} follows.\end{proof}\noindent As before
(see \eqref{lnge}), by setting $P\cc:=\partial_t x$ and using
\eqref{coseva}, we deduce the following O.D.E.'s system:
\begin{eqnarray}\label{lnge22}
\left\{\begin{array}{ll}\nabla_t^{(2)}{Y}+\RC( P\cc,Y)P\cc +
C(P\vv)\nabla_tY=0\\\\
\P\vv(\nabla_tY -[Y,P\cc]) =0.\end{array}\right.\end{eqnarray}

\begin{Defi}[Jacobi
equations for normal CC-geodesics]\label{JE}We say that
\eqref{lnge22} represents the system of Jacobi equations for
normal CC-geodesics. Let $x:[0, a]\longrightarrow\GG$ be a normal
CC-geodesic satisfying $|\partial_tx|=1$. Then we say that a
vector field $J\in\XX(\GG)$ is a Jacobi field along the normal
CC-geodesic $x$ if and only if $J$ satisfies \eqref{lnge22} for
all $t\in[0, a]$.
\end{Defi}

\begin{oss}
An interesting corollary of the second variation formula can be
formulated by considering only some particular CC-geodesic
variations. Indeed, we could consider variations of a given normal
CC-geodesic $x$, or, more precisely $(x,P\vv)$, through normal
CC-geodesics $\vartheta_s$ such that, for every $s\in
]-\epsilon_0, \epsilon_0[$ one has $\,Q\vv=0$.  Roughly speaking,
the variation $\vartheta$ is chosen in such a way that
$\vartheta_s$ satisfies \eqref{lnge2} for each $s\in ]-\epsilon_0,
\epsilon_0[$, with the same $P\vv$.  By making this substitution
into \eqref{coseva} we obtain the following formula:\begin{eqnarray*}\delta^2I\sr
=-\int_a^b\big\langle Y, \big( \nabla_t^{(2)}{Y}+\RC(
\partial_tx,Y)\partial_tx +
C(P\vv)\nabla_tY\big)\big\rangle dt.\end{eqnarray*}We therefore obtain the following 2nd order linear system of
O.D.E.'s:\begin{eqnarray}\label{lnge3} \nabla_t^{(2)}{Y}+\RC(
P\cc,Y)P\cc + C(P\vv)\nabla_tY=0.\end{eqnarray}
\end{oss}

\begin{Defi}[Jacobi
equations with constant Lagrangian multiplier $P\vv$]We say that
\eqref{lnge3} represent the system of Jacobi equations for normal
CC-geodesics having the same Lagrangian multiplier $P\vv$.
Furthermore, let $x:[0, a]\longrightarrow\GG$ be a normal
CC-geodesic satisfying $|\partial_tx|=1$ and having Lagrangian
multiplier $P\vv$. Then we say that a vector field $J\in\XX(\GG)$
is a Jacobi field along the normal CC-geodesic $(x, P\vv)$ if, and
only if, $J$ satisfies \eqref{lnge3} for all $t\in[0, a]$.
\end{Defi}

As in the classical setting, a Jacobi field is uniquely determined
by its initial conditions: $J(0),\, {\nabla_tJ}(0)$. To see this,
it is sufficient to develop either \eqref{lnge22} or \eqref{lnge3}
along the left invariant frame $\underline{X}=\{X_1,...,X_n\}$ by
using the very definition of $\nabla_t$ and the linearity of the
curvature tensor $R(\dot{x}, \cdot)\dot{x}$ (remind that
$\dot{x}=\dot{x}\cc=P\cc(x)$). Indeed, one easily deduce that
\eqref{lnge3} is a linear system of O.D.E.'s of the 2nd order.
Hence, for given initial conditions $J(0),\, {\nabla_tJ}(0)$,
there exists a $\cin$ solution of the system defined on $[0, a]$.
Thus there exist $2n$ linearly independent Jacobi fields along
$(x, P\vv)$.

\begin{oss}As in Riemannian Geometry,  Jacobi equations for normal CC-geodesics and Jacobi
equations with constant Lagrangian multiplier $P\vv$ -together
with the related notions of Jacobi fields- are important tools and
they can be used to analyze the sub-Riemannian exponential map and
to perform a precise study of the conjugate and cut loci of a
point. Nevertheless, we will not pursue this task here.
\end{oss}

\subsection{Sub-Riemannian exponential map  and CC-spheres}\label{ipoo}

The main references for this section are \cite{Rifford},
\cite{RiffordTrelat}, \cite{Ver}.

Starting from the system \eqref{ngs1} for normal CC-geodesics,
from the standard O.D.E.'s  theory we easily obtain existence,
uniqueness, regularity and smooth dependence on the initial data
for small times. So let $x_0\in \GG$ be a fixed point. Then we
shall denote by
$$\exp\sr(x_0,P_0):[0,r]\subset\R\longrightarrow\GG, \qquad
\exp\sr(x_0,P_0)(t)$$ the (unique) normal CC-geodesic starting
from $x_0$ with initial condition $P(0)=P_0$ for some fixed vector
$P_0=P\cc(0) + P\vv(0)\in \HH_{x_0}\oplus\VV_{x_0}$. Here
$r<r_{\rm max}(P_0)$ where $r_{\rm max}(P_0)$ is the maximal time
of existence and uniqueness of the solution
$x(t)=\exp\sr(x_0,P_0)(t)$ of \eqref{ngs1}. Actually, by standard
``extendibility'' results for O.D.E.'s, it can be shown that {\it
each solution of system \eqref{ngs1} is globally defined}
\footnote{More generally, in the setting of sub-Riemannian
manifolds, it can be shown that each solution of the system of
normal CC-geodesic can be continued as long as $x(t)$ remains in
the base manifold, so blow-up in the dual variable $P$ never
occurs. This claim can easily be proved using the
following:\begin{lemma}[\cite{Stric}] Let $x(t)$ be any normal
CC-geodesic defined for $t\in[0,r[$ and assume that $x(t)$ remains
inside a compact subset of the base manifold. Then $x(t)$ can be
extended beyond $t=r$.
\end{lemma}
\begin{proof}The proof can be found in \cite{Stric}, Lemma 4.1.
\end{proof}This result can be used to canonically
define a sub-Riemannian exponential map; see \cite{Stric}.} {\it
on} $\GG$, i.e. $r_{\rm max}(P_0)=+\infty$.

In order to define a sub-Riemannian equivalent to the ordinary
exponential map we preliminarily state the following:
\begin{lemma}[Homogeneity]If the normal
CC-geodesic $x(t)=\exp\sr(x_0,P_0)(t)$ is defined on the interval
$[0,r[$, then the normal CC-geodesic $\exp\sr(x_0, a  P_0)(t)$,
$a>0,$ is defined on the interval $[0,\frac{r}{a}[$ and it turns
out that
$$\exp\sr(x_0,a P_0)(t)=\exp\sr(x_0,P_0)(a t)\qquad \mbox{for every}\,\,t\in \Big[0,\frac{r}{a}\Big[.$$\end{lemma}
\begin{proof}Let $(y,Q):[0,\frac{r}{a}[\longrightarrow\TT\GG$ be the
curve given by $(y(t),Q(t))=(x(at),aP(at))$. Now we claim that
$(y,Q)$ satisfies \eqref{ngs1} with $y(0)=x_0$ and
$Q_0=Q(0)=aP_0$. Indeed, by the very definition of $x(t)$, we get
that $\dot{y}(t)=a\dot{x}(at)=aP\cc(at)$ and  that
$$\dot{Q}(t)=a^2\dot{P}(at)=-a^2C(P\vv(at))P\cc(at)=-C(Q\vv(t))Q\cc(t)$$
for  $t\in[0,\frac{r}{a}[$ and the claim follows. So by uniqueness
we get, in particular, that $$y(t)=\exp\sr(x_0, a
P_0)(t)=x(at)=\exp\sr(x_0, P_0)(at)\qquad \mbox{for every}\,\,
t\in\Big[0,\frac{r}{a}\Big[  ,$$and the thesis follows.
\end{proof}

\begin{Defi}[Sub-Riemannian exponential map]
From now on we shall set
$$\exp\sr(x_0,\cdot):\TT
_{x_0}\GG\longrightarrow\GG,\qquad\exp\sr(x_0,P_0)=\exp\sr(x_0,P_0)(1).$$The
map $\exp\sr(x_0,\cdot)$ is called the {\bf sub-Riemannian
exponential map} at $x_0\in\GG$.
 \end{Defi} The sub-Riemannian
exponential map parameterizes normal CC-geodesics. Note that every
minimizing curve connecting $x_0$ to a point of $\GG\setminus
\exp(x_0,\TT _{x_0}\GG)$ is necessarily strictly abnormal.

The map $\exp\sr(x_0,\cdot)$ plays in sub-Riemannian geometry a
similar role with respect to the ordinary exponential map in
Riemannian geometry. Nevertheless, there are many differences and
its structure is much more complicated. An important difference is
that $\exp\sr(x_0,\cdot)$ is not a diffeomorphism on any
neighborhood of the origin in $\HH_{x_0}\oplus\VV_{x_0}$. To see
this it is enough to choose $P_0\in\VV_{x_0}$ (i.e. $P\cc(0)=0$);
in this case  $\exp\sr(x_0,P_0)=x_0$. Furthermore, in any
arbitrarily small neighborhood of the origin in $\HH _{x_0}
\oplus\VV _{x_0}$ there are points $P_0=P\cc(0) + P\vv(0)$ with
$P\cc(0)\neq 0$ at which the rank of $d \exp\sr(x_0,P_0)$ is not
maximal. We shall discuss some of these facts later on.

We begin by stating the notion of sub-Riemannian wave front.

Let $\UH$ denote the set of all unit horizontal vector of $\HH$,
i.e. $\UH\subset\HH\cong\mathbb{S}^{\DH-1}$ and define the map
$$\widehat{\cdot}:\HH\setminus\{0\}\oplus\VV\longrightarrow\UH\oplus
\VV\qquad\widehat{P}=\frac{P}{|P\cc|}.$$The {\it sub-Riemannian
wave front} $W\sr(x_0,r)$ centered at $x_0$ and with radius $r>0$
is defined as the set of points
$x(r)=\exp\sr(x_0,r\widehat{P_0})=\exp\sr(x_0,\widehat{P_0})(r)$
for $P_0\in\HH_{x_0}\oplus\VV_{x_0}$. In other words
$$W\sr(x_0,r)=\exp\sr(x_0,r(\UH_{x_0}\oplus\VV_{x_0})).$$Note that in the
Riemannian setting the wave front simply coincides with the
$r$-sphere. In our case only the inclusion
$\mathbb{S}^n\sr(x_0,r)\varsubsetneq W\sr(x_0,r)$ holds true.
Actually the structure of $W\sr(x_0,r)$ is  very complicated and,
in general, the sub-Riemannian wave fronts are not manifolds.

\begin{Defi}[Conjugate locus] \label{conj}A point $x\in\exp\sr(x_0, \TT_{x_0}\GG)$ is
said {\bf conjugate to }$x_0$ if and only if it is a critical
value of $\exp\sr(x_0, \cdot)$, i.e. there exists
$P_0\in\HH_{x_0}$ such that $x=\exp\sr(x_0,P_0)$ and
$d\exp\sr(x_0,P_0)$ is not onto. The {\bf conjugate locus}
$\mathrm{Conj}(x_0)$ of $x_0$ is then the set of all points
conjugate to $x_0$.
\end{Defi}
Notice that, from what we have said above, it turns out that
$x_0\in\mathrm{Conj}(x_0)$. Moreover, by Sard Theorem applied to
$\exp\sr(x_0,\cdot)$ one gets that $\mathrm{Conj}(x_0)$ is a set
of (Lebesgue) measure zero in $\GG$.

\begin{Defi}[Cut locus] \label{cutl}Let us fix $x_0\in \GG$
and let us choose a normal CC-geodesic $x(t)=\exp\sr(x_0,P_0)(t)$.
If $t>0$ is sufficiently small, $\dc(x(0),x(t))=t$, i.e.
$x([0,t])$ is a minimizing normal CC-geodesic. Moreover, if for
some $t_1$ $x([0,t_1])$ is not minimizing the same is true for all
$t>t_1$. By continuity, the set of numbers $t>0$ such that
$\dc(x(0),x(t))=t$ is of the form $[0,t_0]$ or $[0,+\infty[$. In
the first case we say that $x(t_0)$ is the {\bf cut point of $x_0$
along $x(t)$} while, in the second case, we say that the cut point
does not exists. The set ${\rm Cut}(x_0)$ defined as the union of
the cut points of $x_0$ along all the normal CC-geodesics starting
from $x_0$ is called the {\bf cut locus} of $x_0$.

\end{Defi}

\subsection{A sub-Riemannian version of the Gauss Lemma}\label{GLem}

\begin{oss}[Working hypothesis]Let $\GG$ be a $k$-step Carnot group.
Throughout this section we shall assume that {\bf there are not
strictly abnormal minimizers} in $\GG$; see Definition \ref{abn}.
\end{oss}

Here below we shall perform an explicit and general computation
that will be an important tool for the rest of this paper. It is
somehow based on the validity of the {\it eikonal equation}, first
proven in \cite{MoSC}.

We may start by the O.D.E.'s system \eqref{ngs1}. We assume the
solution is parameterized by arc-length. Furthermore, let us fix
the initial conditions: $x(0)=x_0$, $P(0)={P}_0$\footnote{
$P(0)=P_0=P\cc (0)+ P\vv(0)$.} and  set $\dc(x):=\dc(x,x_0)$. We
also assume the solution to be with unit speed (i.e. $|P\cc|=1$).
All together, these assumptions uniquely determine a normal
unit-speed CC-geodesic. In particular, we get
$$\dc(x(t))=t$$for every $t\in[0,\epsilon]$ ($\epsilon>0$ small
enough). By differentiating this identity, we obtain
$$\frac{d}{dt}\dc(x(t))=\langle\grad\,\dc(x(t)),\dot{x}(t)\rangle=\langle \dg \dc(x(t)),\dot{x}(t)\rangle=1.$$
Since $\dot{x}=P\cc$ and $|P\cc|=1$, the eikonal equation implies
that $\measuredangle\big(\dg\dc(x(t)),\dot{x}(t)\big)=0$, or
equivalently, that
$$\dg\dc(x(t))=\dot{x}(t)$$for every
small enough $t\geq 0$. This can be written more explicitly as
follows: \begin{equation}\label{nordh}X_i\dc(x(t))=P_i\qquad i\in
I\cc=\{1,...,\DH\}.\end{equation}Let now $\alpha\in
I\vv=\{\DH+1,...,n\}$ and let us differentiate the quantity
$X_\alpha\dc$ along the normal CC-geodesic $x(t)$ defined by the
previous assumptions. We have\footnote{For sake of simplicity, in
these computations we shall drop the dependence on the variable
$t$.}
\begin{eqnarray*}\frac{d}{dt}X_\alpha\dc(x)&=&\langle\grad\,(X_\alpha\dc)(x),\dot{x}\rangle=
\langle\dg\,(X_\alpha\dc)(x),\dot{x}\rangle\\&=&
\langle\dg\,(X_\alpha\dc)(x),P\cc\rangle=\sum_{i\in I\cc}P_i
X_i(X_\alpha\dc)(x).\end{eqnarray*}Since $X_iX_\alpha=X_\alpha X_i
+ [X_i,X_\alpha]=X_\alpha X_i + \sum_{\beta \in I_{{\rm
ord}(\alpha) +1}}\SC_{i\alpha}^\beta X_\beta$, we get

\begin{eqnarray*}\frac{d}{dt}X_\alpha\dc(x)&=&\sum_{i\in I\cc}P_i
X_\alpha(X_i\dc)(x) +  \sum_{\beta \in I_{{\rm ord}(\alpha)
+1}}\SC_{i\alpha}^\beta X_\beta\dc(x)\\&=&{\sum_{i\in
I\cc}P_i\Big( X_\alpha(P_i)(x)} + \sum_{\beta \in I_{{\rm
ord}(\alpha) +1}}\SC_{i\alpha}^\beta
X_\beta\dc(x)\Big)\\&=&\frac{1}{2} {X_\alpha({|P\cc|^2})(x)} +
\sum_{i\in I\cc}\sum_{\beta \in I_{{\rm ord}(\alpha) +1}}P_i
\SC_{i\alpha}^\beta X_\beta\dc(x)\end{eqnarray*} the first term in
the sum is zero because $|P\cc|=1$ (unit-speed). Using Definition
\ref{nota} and the skew-symmetry of $C^\beta$, we finally obtain
\begin{eqnarray}\label{norver}\frac{d}{dt}X_\alpha\dc(x)&=&-\sum_{\beta \in I_{{\rm
ord}(\alpha) +1}}X_\beta\dc(x)\langle C^\beta P\cc,
X_\alpha\rangle.
\end{eqnarray}

We summarize the previous discussion in the following:
\begin{lemma}\label{LEM1}Let $x:[0,r]\longrightarrow\GG\,(r>0)$ be any
normal CC-geodesic of unit-speed and parameterized by arc-length.
Let $x(0)=x_0$, $P(0)={P}_0$ be its initial data and set
$\dc(x)=\dc(x_0,x)\,(x\in\GG)$. Then we
have\begin{itemize}\item[(i)]$\dg\dc(x(t))=P\cc(t)$ for every
$t\in[0,r]$;\item[(ii)]$\frac{d}{dt}\grad\vv\dc(x(t))=-\sum_{\beta
\in I_{{\rm ord}(\alpha) +1}}X_\beta\dc(x(t))C^\beta P\cc(t)$ for
every $t\in[0,r]$.\end{itemize}
\end{lemma}
\begin{proof}The first claim is \eqref{nordh}, while the second one is \eqref{norver},
 both rewritten using vector notation.\end{proof}

We reformulate  Lemma \ref{LEM1}, by using the notation given in
Definition \ref{NONONO}. One has
\begin{eqnarray*}\frac{d}{dt}\grad\vv\dc(x)=-C(\grad\vv\dc(x))P\cc.\end{eqnarray*}

At this point Lemma \ref{LEM1} can be restated, in geometric
terms, as follows:
\begin{Prop}[Sub-Riemannian Gauss' Lemma]\label{PODER}Let
$\mathbb{S}^n\sr(x_0, t)$ the CC-sphere centered at $x_0$ of
radius $t\in[0,r]$ and set
 $\nn=\nn|_{\mathbb{S}^n\sr(x_0,
t)}$ and $\varpi=\varpi|_{\mathbb{S}^n\sr(x_0, t)}$\,\, $(t\in
[0,r])$. Let $x:[0,r]\longrightarrow\GG\,(r>0)$ be any normal
CC-geodesic of unit-speed, parameterized by arc-length, with
initial data $x(0)=x_0$, $P(0)={P}_0$, i.e. $x(t)=\exp\sr(x_0,
P_0)(t)\,\,(t\in[0,r])$. Then, for every $t\in[0,r]$ the following
O.D.E.'s system holds:\begin{eqnarray}\label{ODER}
\left\{\begin{array}{ll}\frac{d x}{dt}=\nn\\\\
\frac{d\nn}{dt}=-C\cc(\varpi)\nn
\\\\\frac{d\varpi}{dt}=-C(\varpi)\nn.\end{array}\right.\end{eqnarray}
\end{Prop}

Notice that the first equation in \eqref{ODER} says that each
normal CC-geodesic starting from $x_0$ intersects
$\mathbb{S}^n\sr(x_0, t)$ {\it orthogonally} (in the horizontal
sense), i.e. at the intersection point $x(t)$ the velocity vector
of the normal CC-geodesic coincides with the horizontal unit
normal. Furthermore, the second and third equations in
\eqref{ODER} express how change $\nn=\nn|_{\mathbb{S}\sr(x_0, t)}$
and $\varpi=\varpi|_{\mathbb{S}\sr(x_0, t)}$ along $x(t)$.

\begin{proof}[Proof of Proposition \ref{PODER}]We stress that our hypothesis about the absence of
abnormal minimizers in $\GG$ implies the smoothness of $\dc$.
Since every CC-sphere $\mathbb{S}^n\sr(x_0,t)\,\,(t\in[0,r])$
turns out to be defined as
$\mathbb{S}^n\sr(x_0,t)=\{x:\dc(x)=t\}$, the Riemannian unit
normal vector $\nu$  along $\mathbb{S}^n\sr(x_0,t)$ (at each
regular non-characteristic point) may be written just by
normalizing the following (non unit) normal vector along
$\mathbb{S}^n\sr(x_0,t)$
$$\mathcal{N}:=\grad\,\dc=(\dg\dc,\grad\cd\dc).$$Clearly, $\nn$ is uniquely
determined \footnote{Remind that, by definition,
$\nn:=\frac{\PH\nu}{|\PH\nu|}$, and
$\varpi:=\frac{\P\vv\nu}{|\PH\nu|}$.} by $\mathcal{N}$. The
eikonal equation implies that
$$\nn=\dg\dc,\quad\varpi=\grad\vv\dc\qquad\forall\,\,t\in
[0,r].$$Therefore, by the first equation of system \eqref{ngs1}
and by (i) of Lemma \ref{LEM1}, we immediately get that
$\dot{x}(t)=P\cc(t)=\nn(x(t))$. This identity together with (ii)
of Lemma \ref{LEM1} implies the third equation in \eqref{ODER}.
Finally, the second equation in \eqref{ODER} immediately follows
by using (i) of \ref{LEM1} together with the third equation of
\eqref{lnge23}.
\end{proof}

\begin{oss}
We stress that Lemma \ref{PODER} solves the problem of selecting,
for any regular point $x_1$ belonging to the CC-sphere
$\mathbb{S}^n\sr(x_0,r)$, the unique normal CC-geodesic having
velocity vector equals to the horizontal normal direction at that
point (i.e. $P\cc(0)=\nn(x_1)$) and connecting this point to the
center $x_0$ of the CC-sphere. Actually, Lemma \ref{PODER} and the
uniqueness of solutions of O.D.E.'s imply that the desired curve
must be the normal CC-geodesic defined by
$$\widetilde{x}(t):=\exp(x_1,-\mathcal{N}(x_1))(t)\qquad
t\in[0,r]$$where $\mathcal{N}=(\nn,\varpi)$.

\end{oss}
\begin{corollario}\label{important}
Let $S=\{x\in\mathbb{G}:\:\:f(x)=0 \},$ where $f$ is a $C^2$
function. Assume that there exists a CC-ball $B(y,r)$ with center
$y$ and radius $r,$ such that $B(y,r)\subset \{f(x)<0 \}$, or
$B(y,r)\subset \{f(x)>0 \}$, and ${B(y,r)}\cap S=\{x\}$, where
$x\in S$ is non-characteristic. Then there exists the metric
normal $\gamma_{\mathcal{N}}$ to $S$ at $x$ and  for every
$t\in[0,r],$ $\gamma(t)\in \gamma_{\mathcal{N}},$ where
$$\gamma(t):=\exp\sr(y,-\mathcal{N}(x))(t)\qquad
t\in[0,r].$$\end{corollario}

\subsection{2-step case: explicit integration and other
features}\label{2-stepgeod}

In this section we shall explicitly analyze the case of 2-step
Carnot groups. Remind that in the 2-step setting there exist no
abnormal minimizers and our working hypothesis is satisfied. In
this case, it is well-known that CC-geodesics are smooth; see
\cite{GoKA}.

 In order to describe the
 sub-Riemannian exponential map,  we note that the system for
normal CC-geodesics can explicitly  be integrated. Using
\eqref{ngstep2}, i.e.
\begin{eqnarray*}
\left\{\begin{array}{ll}\,\,\,\,\dot{x}=P\cc\\\dot{P}\cc=-
C\cc(P\cd) P\cc\\\dot{P}\cd=0,\end{array}\right.\end{eqnarray*}we
get that $P\cd\in\R^{\DH_2}\cong\HH_2$ is a constant vector. By
standard results about O.D.E.'s, we therefore get that
$$P\cc(t)=e^{-C\cc(P\cd)t}P\cc(0)
.$$ To obtain
 the solution in exponential coordinates first note
that the equation $\dot{x}=P\cc$ is
equivalent\footnote{Explicitly, we have \begin{eqnarray*}
\dot{x}_\alpha=\langle P\cc,\ee_\alpha\rangle=\sum_{i\in
I\cc}P_i\langle X_i,\ee_\alpha\rangle= \sum_{i\in
I\cc}P_i\Big(-\frac{1}{2}\langle C^\alpha\cc x,\ee_i\rangle\Big)=
-\frac{1}{2}\langle C^\alpha\cc x,P\cc\rangle\qquad(\alpha\in
I\cd).\end{eqnarray*}} to
$$
\begin{cases}\,\dot{x}_i=P_i&\qquad(i\in I\cc)\\
\dot{x}_\alpha=-\frac{1}{2}\langle C^\alpha\cc x\cc,P\cc\rangle
&\qquad(\alpha\in I\cd).\end{cases}
$$

Therefore, setting $x\cc(t):=(x_1(t),...,x_\DH(t))\in\R^\DH$ to
denote the projection of the solution $x(t)$ of \eqref{ngstep2}
onto the firsts $\DH$ variables\footnote{Using exponential
coordinates for $\GG$, every point $x\in\GG$ is $n$-tuple
$x=(x_1,...,x_\DH, x_{\DH+1},...,x_n)$. So it seems natural to
``divide'' the variables as
follows:$$x\cc:=(x_1,...,x_\DH)\in\R^\DH,\qquad
x\cd:=(x_{\DH+1},...,x_n)\in\R^{\DH_2} \qquad(n=\DH+ \DH_2)$$.}, one gets
\begin{eqnarray}\label{pdio}x\cc(t)=x\cc(0) +
\int_0^te^{-C\cc(P\cd)s}P\cc(0)\,d s\end{eqnarray}and
\begin{eqnarray}\label{caz1}x_\alpha(t)=
x_\alpha(0)-\frac{1}{2} \int_{0}^t \langle C^\alpha\cc
x\cc,\dot{x}\cc\rangle\,d s.\end{eqnarray}

These equations describe the sub-Riemannian exponential map in the
2-step case. More precisely, fixing a base point $x_0$, we have
that
$$\exp\sr(x_0,\cdot)(\cdot):\UH\times\HH_2\times\R\longrightarrow\GG$$
is given by
\begin{eqnarray}\label{explicit}\exp\sr(x_0,P_0)(t):=x_0 + \int_0^te^{-C\cc(P\cd)s}P\cc(0)\,d
s-\frac{1}{2} \sum_{\alpha\in I\cd}\bigg\{\int_{0}^t \langle
C^\alpha\cc
x\cc,\dot{x}\cc\rangle\,ds\bigg\}\ee_\alpha.\end{eqnarray}

Here above $\UH$ denotes the bundle of all unit horizontal
vectors, i.e. $\UH\subset\HH\cong\mathbb{S}^{\DH-1}$.

\begin{oss}Setting $x(t):=\exp\sr(x_0,P_0)(t)$ and \[\mathcal{E}(t):=\int_0^te^{-C\cc(P\cd)s}\,d
s\in \mathcal{M}_{\DH},\] we get that\\

$x(t)=x_0+\Big(\langle\mathcal{E}(t)\P\cc(0), \ee_1
\rangle,...,\langle\mathcal{E}(t)\P\cc(0),
\ee_\DH\rangle,\ldots,\underbrace{\frac{1}{2}\Big(\int_{0}^t
\langle C^\alpha\cc\dot{x}\cc ,x\cc\rangle\,ds\Big)}_{\alpha-th \,
place},\ldots\Big)$.
\end{oss}

\begin{no}In the sequel, we shall denote by $d x\cc$ the following vector valued
1-form:
\[d x\cc:=(dx_1,...,dx_{\DH})^{\rm Tr}\in
\underbrace{\HH^\ast\times...\times\HH^\ast}_{\DH-\mbox{\tiny
times}}.\] \end{no}

\begin{no}
Let $x:[0, T]\longrightarrow\GG$ be a CC-normal geodesic such that
$x(0)=x_0$ and $x(T)=x_1$. Later on we shall set $[x_0,
x_1]:=x([0, T])=\{x(t)\in\GG : t\in[0, T]\}$.
\end{no}

\begin{oss}\label{uiiu}

For any $\alpha\in I\cd$, the integral
$$\mathcal{I}^\alpha_{x\cc}(t):=\int_{0}^t \langle C^\alpha\cc
x\cc,\dot{x}\cc\rangle\,ds=\int_{[x\cc(0),x\cc(t)]} \langle
C^\alpha\cc x\cc,d{x}\cc\rangle$$ can explicitly be evaluated, by
means of standard linear algebra arguments, by noting that, as
every skew-symmetric linear operator, $C^\alpha\cc$ can be
written, after an orthogonal change of basis, in a ``canonical''
form. More precisely, there exists $O^\alpha\in
\mathbf{O}_{\DH}(\R)$\footnote{i.e. the orthogonal group on
$\R^{\DH}(\cong\HH)$.} such that
$$(O^\alpha)^{-1}C^\alpha\cc O^\alpha=\left(%
\begin{array}{cccccccccc}
  0 & \lambda^\alpha_1 & 0 & 0 &  & \ldots &  &  \\
 -\lambda^\alpha_1 & 0 & 0 & 0 &  & &  &  \\
 0 & 0 & 0& \lambda^\alpha_2 & & &  & \\
  0 &0 & -\lambda^\alpha_2 & 0 & & &  &  \\
& &  &  &  &\ddots & & \\
 & &  &  &  & & 0&\lambda^\alpha_R &  \\
& &  &  &  &  &-\lambda^\alpha_R&  0&\\
 & &  &  &  &  & && \mathbf{0}_{N} \\
\end{array}
\right)$$where
$\pm\,\mathbf{i}\,\lambda^{\alpha}_j\,\,(j=1,...,R)$ are the
eigenvalues - purely imaginary - of $C^\alpha\cc$ computed with
their multiplicity, and $\mathbf{0}_N$ denotes the zero $N\times
N$-matrix where $N$ is the nullity of $C^\alpha\cc$. Here
$R=\frac{1}{2}\mbox{\rm {rank}}\,C^\alpha\cc$ and  $\DH=2R +N$.
Setting
\[y\cc(t)=(O^\alpha)^{-1}x\cc(t),\quad y\cc(0)=(O^\alpha)^{-1}x\cc(0),\]
 we get that\begin{eqnarray*}\mathcal{I}^\alpha_{x\cc}(t)=\int_{[x\cc(0),x\cc(t)]} \langle
C^\alpha\cc x\cc,d{x}\cc\rangle&=&\sum_{j=1}^R
\lambda^\alpha_j\int_{[y\cc(0),y\cc(t)]}{y_{j+1}dy_j-y_jdy_{j+1}}.\end{eqnarray*}
We stress that similar integrals was previously introduced and
studied by Pansu  in his deep study of differentiability in
CC-spaces; see \cite{P2}, Definition 4.6, p.15.
\end{oss}

Using a standard result about periodic solutions of linear
O.D.E.'s \footnote{The following result holds true (see
\cite{Bose}):\begin{teo}Let
\begin{equation}\label{azlan}\dot{y}:=A(t)y \end{equation} be a $n$-th order
system of linear homogeneous differential equations, where
$A(t)=[a_{ij}(t)]_{i, j=1,...,n}$ is a $n\times n$ matrix and each
 element $a_{ij}(t)$ is a real-valued function continuous
 on the real line $\R$. Let $\mathbf{W}$ denote a fundamental matrix of \eqref{azlan}.
  Let $k$ be a non-negative integer, $0\leq k\leq n$. Then there
exists a $k$-dimensional sub-space $\mathcal{S}_k$ of the solution
space $\mathcal{S}_n$ of \eqref{azlan} such that each member of
$S_k$ is periodic of period $T$ and no member of
$\mathcal{S}_n\setminus \mathcal{S}_k$ is periodic of period $T$
if and only if the rank of the matrix
$\mathbf{W}(T)-\mathbf{W}(0)$ is $n-k$\end{teo}}, we may study the
existence of $T$-periodic solutions, $T>0$, of the $\DH$-th order
linear O.D.E.'s system
\begin{equation}\label{lan} \dot{P}\cc=-
C\cc(P\cd)P\cc,\end{equation}where $P\cd\in\HH_2$ is constant.
More precisely, we get that there exists a $k$-dimensional
sub-space $\mathcal{S}_k$ of the solution space $\mathcal{S}_\DH$
of the equation \eqref{lan}  such that each member of
$\mathcal{S}_k$ is $T$-periodic, and no member of
$\mathcal{S}_\DH\setminus \mathcal{S}_k$ is $T$-periodic, if and
only if the following holds:
\[{\rm rank}(e^{-C\cc(P\cd)T}-\mathrm{Id}_\DH)=\DH-k.\]Hereafter, we shall analyze this
condition. First, note that $P\cd=0$ implies that $P\cc$ is
constant and so we may assume $P\cd\neq 0$. As in Remark
\ref{uiiu}, we make use of a standard Linear Algebra argument.
More precisely, as every skew-symmetric linear operator,
$C\cc(P\cd)$ can be written, after an orthogonal change of basis,
in canonical form. In particular, there exists $O\in
\mathbf{O}_{\DH}(\R)$ such that
$${\widetilde{C\cc}(P\cd)}:=O^{-1}C\cc(P\cd) O=\left(%
\begin{array}{cccccccccc}
  0 & \lambda_1 & 0 & 0 &  & &  &  \\
 -\lambda_1 & 0 & 0 & 0 &  & &  &  \\
 0 & 0 & 0& \lambda_2 & & &  & \\
  0 &0 & -\lambda_2 & 0 & & &  &  \\
& &  &  &  &\ddots & & \\
 & &  &  &  & & 0&\lambda_R &  \\
& &  &  &  &  &-\lambda_R&  0&\\
 & &  &  &  &  & && \mathbf{0}_{N} \\
\end{array}
\right)$$where $\pm\,\mathbf{i}\,\lambda_j\,\,(j=1,...,R)$ are the
- purely imaginary - eigenvalues of $C\cc(P\cd)$, computed with
their multiplicity, and $\mathbf{0}_N$ denotes the zero $N\times
N$-matrix, where $N$ is the nullity of $C\cc(P\cd)$. Here
$R=\frac{1}{2}\mbox{\rm {rank}}\,C\cc(P\cd)$ and  $\DH=2R +N$.
Notice that the eigenvalues are functions of $P\cd$, i.e.
$\lambda_j=\lambda_j(P\cd)\,\,(j=1,...,R)$. Setting
\[\widetilde{P}\cc(t)=O^{-1}P\cc(t),\qquad
\widetilde{P}\cc(0)=O^{-1}P\cc(0),\]
 we obtain the following equivalent equation:\begin{equation}\label{1lan}
\dot{\widetilde{P}\cc}=-
\widetilde{C\cc}(P\cd)\widetilde{P}\cc.\end{equation}By applying
the previous argument, we infer that there exists a
$k$-dimensional sub-space $\mathcal{S}_k$ of the solution space
$\mathcal{S}_\DH$ of the equation \eqref{1lan}  such that each
member of $\mathcal{S}_k$ is $T$-periodic, and no member of
$\mathcal{S}_\DH\setminus \mathcal{S}_k$ is $T$-periodic, if and
only if:
\begin{equation}\label{2lan}{\rm rank}(e^{-\widetilde{C\cc}(P\cd)T}-\mathrm{Id}_\DH)=\DH-k.\end{equation}
Therefore, it remains us to analyze the matrix
$e^{-\widetilde{C\cc}(P\cd)T}$  given by

$$\left(%
\begin{array}{cccccccccc}
 \cos(\lambda_1(P\cd) T) & \sin(\lambda_1(P\cd) T)& 0 \ldots  \\
-\sin(\lambda_1(P\cd) T)&  \cos(\lambda_1(P\cd) T)& 0 \ldots \\
0 & 0 & \ddots & & &\\
 \vdots& \vdots&  &  &     \cos(\lambda_R(P\cd) T)&  \sin(\lambda_R(P\cd) T) &  \\
& &  &  &  -\sin(\lambda_R(P\cd) T)&   \cos(\lambda_R(P\cd) T)&\\
 & &  &  &  &  & && \mathrm{Id}_{N} \\
\end{array}
\right).$$

We claim that for every given $P\cd\neq 0$, there exist
$T$-periodic solutions of \eqref{lan}, for some positive
$T=T(P\cd)$.  Indeed, if $N\neq 0$, i.e. the nullity $N$ of
$C\cc(P\cd)$ is non-zero, the claim follows by the above
discussion since ${\rm
rank}(e^{-\widetilde{C\cc}(P\cd)T}-\mathrm{Id}_\DH)\geq \DH-N$ and
$k\geq N>0$. Furthermore, if $N=0$, the claim follows by choosing
$T={2\pi }/{\lambda_j(P\cd) }$ for some $j=1,...,R$. Indeed in
such a case one sees that $k\geq 2$. With an analogous argument,
one can easily show that the dimension of the space
$\mathcal{S}_k$ of $T$-periodic solutions of \eqref{lan}
satisfies:$$N+ 2\leq k\leq \DH.$$Note that these solutions can
also be constant functions. In similar way, we see that for every
$P\cd\neq 0$ there exists $T=T(P\cd)>0$ and there exist -at least-
two non-constant, linearly independent, $T$-periodic solutions of
\eqref{lan}.

We summarize the above discussion in the next:
\begin{Prop}\label{df}For every $P\cd\in \HH_2$, $P\cd\neq 0$, there
exists $T=T(P\cd)>0$ and there exists a $k$-dimensional sub-space
$\mathcal{S}_k$ of the solution space $\mathcal{S}_\DH$ of
\eqref{1lan} such that each member of $\mathcal{S}_k$ is
$T$-periodic and no member of $\mathcal{S}_\DH\setminus
\mathcal{S}_k$ is $T$-periodic. Furthermore it turns out that $N+
2\leq k\leq \DH,$ where $N$ is the nullity of $C\cc(P\cd)$.
Finally, there exists a positive integer $r\geq 2$ and there
exists a $r$-dimensional sub-space $\mathcal{S}_{r}$ of
$\mathcal{S}_k$ such that each member of $\mathcal{S}_{r}$ is
$T$-periodic and non-constant.
\end{Prop}

\begin{oss}In the previous Proposition \ref{df} the numbers $k$ and $\,r$ can
be characterized in a slightly different way.  Indeed it turns out
that $k$ is the multiplicity  of the eigenvalue $\lambda=1$ of the
scalar matrix $e^{-C\cc(P\cd)T}$. Moreover $r=k-N$. To prove this
claim one can use another characterization of $T$-periodic
solutions of linear homogeneous systems of O.D.E.'s which can be
found in \cite{Bose}.
\end{oss}

We may use the above Proposition \ref{df} to study the
$T$-periodicity of $x\cc$; see \eqref{pdio}.

\begin{oss}If $x\cc$ is $T$-periodic then
$\int_0^TP\cc(s)\,ds=0$. This follows by hypothesis, using the
first equation of \eqref{ngstep2}. Note also that if $P\cc$ is
$T$-periodic, then $x\cc$ is $T$-periodic if and only if
$\int_0^TP\cc(s)\,ds=0$\footnote{Remind that if
$f:\R\longrightarrow\R$ is a continuous $T$-periodic function,
then $\int_0^tf(s)\,ds$ is $T$-periodic  if and only if
$\int_0^Tf(s)\,ds=0.$}.
\end{oss}

\begin{es} Let $P\cd\neq 0$ and let us assume that $C\cc(P\cd)$ is  invertible. For instance, this is the case
of the Heisenberg group $\mathbb{H}^n$. In such case, by a direct
computation  based on the very definition of $e^{-C\cc(P\cd)T}$,
it turns out that\[x\cc(t)=x\cc(0)+(C\cc(P\cd))^{-1}\big({\rm
Id}_n-e^{-C\cc(P\cd)t}\big)P\cc(0).\]\end{es}

\begin{oss}\label{k}Let us consider the system of normal CC-geodesics \eqref{ngstep2}. We would like to
remark that the $T$-periodicity of $x\cc$  is related with some
other things. To this aim, let us assume that $x\cc$ be
$T$-periodic. More precisely, we assume that there exists a
minimal $T=T(P\cd)>0$ such that $x\cc(t)$ is $T$-periodic, for any
given $P\cd\neq 0$. Furthermore, let $x_1:=\exp\sr(x_0, P_0)(T)$,
where $P_0=(P\cc(0), P\cd)\in\UH\times\HH_2$. Then it can be shown
that:\begin{itemize}\item[(i)] The point $x_1$ is conjugate to
$x_0$ along the normal CC-geodesic $x(t)=\exp\sr(x_0, P_0)(t)$,
$t\in[0, T]$;\item[(ii)]The ``segment'' $[x_0, x_1]:=x([0,
T])=\{x(t)\in\GG : t\in[0, T]\}$ is a minimizing normal
CC-geodesic and $x_1$ is the cut-point of $x_0$ along the normal
CC-geodesic $x(t)$.\end{itemize}In this case, we will say that
$[x_0, x_1]$ is a {\bf minimizing CC-geodesic segment}. Note that
the CC-length of segment $[x_0, x_1]$ is simply $T$. The proofs of
these claims can be done by following a classical pattern, for
which we refer the reader to \cite{DC}. However, this is beyond
the scope of this paper.
\end{oss}

\begin{oss}\label{k22}For 2-step Carnot groups, similar arguments can be used to show  that
 the CC-sphere $\mathbb{S}^n\sr(x,r)=\{y\in\GG :
\dc(x,y)=r\}$ is $\cin$-smooth out of the set $$\HH_2(x)\cap \mathbb{S}^n\sr(x,r),$$where
\[\HH_2(x):=\{z=\esp(z\cc, z\cd)\in\GG : z\cc=x\cc\}.\]More
precisely, each point  $y\in\mathbb{S}^n\sr(x,r)\setminus
\HH_2(x)$ can be joined to the center $x$ of
$\mathbb{S}^n\sr(x,r)$ by a unique minimizing normal CC-geodesic.
In particular, one can show that for every
$y\in\mathbb{S}^n\sr(x,r)\setminus \HH_2(x)$ there exists a unique
$P_0\in\UH\times\HH_2$ such that $y=\exp\sr(x, P_0)(r)$ and
$d\exp\sr(x,P_0)$ has maximal rank. On the other hand, if
$y\in\mathbb{S}^n\sr(x,r)\cap \HH_2(x)$, then it turns out that
$y\in\mathrm{Conj}(x)$.
\end{oss}

\subsection{Appendix: iterative integration for  normal CC-geodesics}\label{appe} In
this section we will how how, at least in principle, the system of
normal CC-geodesics  canbe integrated, step by step. In what
follows, we shall use the notation $P\ci=\P\ci(P)$. Remind that
\eqref{ngs1} is given
by\begin{eqnarray*}\left\{\begin{array}{ll}\,
\dot{x}=P\cc\,\,\,\qquad\qquad\big(\Longleftrightarrow\dot{x}\cc=P\cc,\,\,
\dot{x}\vv=0\big)\\\dot{P}=- C(P\vv)
P\cc;\end{array}\right.\end{eqnarray*}By orthogonal projection
onto $\HH_k$, we get that $\dot{P}\ck=0$. Hence $P\ck$ is
constant. Then the $k-1$-th vector equation becomes
$\dot{P}{_{^{_{\HH_{k-1}}}}}=-C(P\ck)\dot{x}$ and so
\begin{eqnarray*}{P}{_{^{_{\HH_{k-1}}}}}(t)={P}{_{^{_{\HH_{k-1}}}}}(0)-C(P\ck)(x(t)-x_0).
\end{eqnarray*}By iterating the same procedure
one gets
 $\dot{P}{_{^{_{\HH_{k-2}}}}}=-C(P{_{^{_{\HH_{k-1}}}}})\dot{x}\cc$. Hence\begin{eqnarray*}{P}{_{^{_{\HH_{k-2}}}}}(t)&=&{P}{_{^{_{\HH_{k-2}}}}}(0)-\int_0^t
 C(P{_{^{_{\HH_{k-1}}}}}(s))\dot{x}\cc\,ds\\&\ldots&\\{P}\ci(t)&=&{P}\ci(0)-\int_0^t
 C(P{_{^{_{\HH_{i-1}}}}}(s))\dot{x}\cc\,ds\\&\ldots&\\{P}\cd(t)&=&{P}\cd(0)-\int_0^t
 C(P\ctr(s))\dot{x}\cc\,ds\\{P}\cc(t)&=&{P}\cc(0)-\int_0^t
 C\cc(P\cd(s))\dot{x}\cc\,ds.\end{eqnarray*}Finally we get that
\begin{eqnarray*}{x}\cc(t)&=& x\cc(0) + \int_0^t P\cc(s)\,ds\\&=&x\cc(0) + \int_0^t
\Big({P}\cc(0)-\int_0^{s_1}
 C\cc(P\cd(s_2))\dot{x}\cc\,ds_2\Big)\,ds_1\\&=&x\cc(0) + {P}\cc(0)t-\int_0^t
\Big(\int_0^{s_1}
 C\cc(P\cd(s_2))\dot{x}\cc\,ds_2\Big)\,ds_1,\end{eqnarray*}
or equivalently, that
\begin{eqnarray*}\ddot{x}\cc(t)=
 C\cc(P\cd(t))\dot{x}\cc.\end{eqnarray*}

 In the sequel we shall apply this procedure to
the case of general 3-step Carnot groups. We stress that the first
vector equation in \eqref{ngstep3}, i.e. $\dot{x}=P\cc$, in
exponential coordinates, is equivalent\footnote{By the formula for
left-invariant vector field stated in Example \ref{3stex}, we get
that
\begin{eqnarray*} \dot{x}_\alpha&=&\langle
P\cc,\ee_\alpha\rangle=\sum_{i\in I\cc}P_i\langle
X_i,\ee_\alpha\rangle= \sum_{i\in I\cc}P_i\Big(-\frac{1}{2}\langle
C^\alpha\cc x\cc,\ee_i\rangle\Big)= -\frac{1}{2}\langle
C^\alpha\cc x\cc,P\cc\rangle\qquad\qquad\,\,\, (\alpha\in
I\cd)\\\dot{x}_\beta&=&\langle P\cc,\ee_\beta\rangle=\sum_{i\in
I\cc}P_i\langle X_i,\ee_\beta\rangle= -\Big(\frac{1}{2}\langle
C^\beta x,P\cc\rangle  -\frac{1}{12}\sum_{\alpha\in I\cd} \langle
C^\beta x,\ee_\alpha\rangle\langle C^\alpha\cc
x\cc,P\cc\rangle\Big)\qquad(\beta\in I\ctr).\end{eqnarray*}} to
the following system:
\begin{eqnarray*}
\left\{\begin{array}{ll}\dot{x}_i = P_i
\qquad\qquad\qquad\qquad\qquad\qquad\qquad\qquad\qquad\qquad\qquad(i\in
I\cc)\\\\
\dot{x}_\alpha = -\frac{1}{2}\langle C^\alpha\cc
x\cc,P\cc\rangle\qquad\qquad\qquad\,\qquad\qquad\qquad\qquad\qquad\,(\alpha\in
I\cd)\\\\
\dot{x}_\beta=-\big(\frac{1}{2}\langle C^\beta x,P\cc\rangle  -
\frac{1}{12}\sum_{\alpha\in I\cd}\langle C^\beta
x,\ee_\alpha\rangle\langle C^\alpha\cc x\cc,P\cc\rangle
\big)\,\,\,\quad(\beta\in
I\ctr).\end{array}\right.\end{eqnarray*}Therefore, for any
$\alpha\in I\cd$, we have
\begin{eqnarray}\label{caz2}x_\alpha(t)=
x_\alpha(0)-\frac{1}{2} \int_{[x\cc(0),x\cc(t)]} \langle
C^\alpha\cc x\cc,d{x}\cc\rangle\end{eqnarray}and, for any
$\beta\in I\ctr$, we have
\begin{eqnarray}\nonumber x_\beta(t)&=&
x_\beta(0)-\int_0^t\Big(\frac{1}{2}\langle C^\beta
x,\dot{x}\cc\rangle - \frac{1}{12}\sum_{\alpha\in I\cd}\langle
C^\alpha\cc x\cc,\dot{x}\cc\rangle \langle C^\beta
x,\ee_\alpha\rangle\Big)ds\\\label{caz3}&=&
x_\beta(0)-\frac{1}{2}\int_{[x\cc(0),x\cc(t)]}\Big(\langle C^\beta
x,dx\cc\rangle - \frac{1}{6}\sum_{\alpha\in I\cd}\langle C^\beta
x,\ee_\alpha\rangle\langle C^\alpha\cc x\cc,dx\cc\rangle
\Big)\end{eqnarray}Now we proceed with the integration of
\eqref{ngstep3} as we have explained in general, at the beginning
of this section. As before, we have that $\dot{P}\ctr=0$ and so
$P\ctr$ is constant. Since $\dot{P}\cd=-C(P\ctr)\dot{x}\cc$ we get
\begin{eqnarray*}{P}\cd(t)={P}\cd(0)-C(P\ctr)(x\cc(t)-x\cc(0)).
\end{eqnarray*}Moreover
 $\dot{P}\cc=-C(P\cd)\dot{x}\cc$,
 and so\begin{eqnarray*}{P}\cc(t)&=&{P}\cc(0)-\int_0^t
 C\cc(P\cd(s))\dot{x}\cc(s)\,ds\\&=&{P}\cc(0)-\int_0^t
 C\cc\big(\underbrace{{P}\cd(0)-C(P\ctr)[x\cc(s)-x\cc(0)]}_{=P\cd(s)}\big)\dot{x}\cc(s)\,ds\end{eqnarray*}
 Finally, we have
\begin{eqnarray*}x\cc(t)&=& x\cc(0) + \int_0^t P\cc(s)\,ds\\&=&x\cc(0) + \int_0^t
\Big({P}\cc(0)-\int_0^{s_1}
 C\cc(P\cd(s_2))\dot{x}\cc(s_2)\,ds_2\Big)\,ds_1\\&=&x\cc(0) + \int_0^t
\Big({P}\cc(0)-\int_0^{s_1}
 C\cc\big(\underbrace{{P}\cd(0)-C(P\ctr)[x\cc(s_2)-x\cc(0)]}_{=P\cd(s_2)}
\big)\dot{x}\cc(s_2)\,ds_2\Big)\,ds_1\\&=&x\cc(0) +
{P}\cc(0)t-\int_0^t \Big[\int_0^{s_1}
 C\cc\big(\underbrace{{P}\cd(0)-C(P\ctr)[x\cc(s_2)-x\cc(0)]}_{=P\cd(s_2)}
\big)\dot{x}\cc\,ds_2\Big]\,ds_1\end{eqnarray*}or, equivalently,
\begin{eqnarray}\label{for3src}\ddot{x}\cc=-C\cc
\big(\underbrace{P\cd(0)-C(P\ctr)[x\cc(t)-x\cc(0)]}_{=P\cd(t)}\big)\dot{x}\cc.\end{eqnarray}

\section{\large CC-distance from hypersurfaces in 2-step Carnot
groups}\label{us}
 {\it Throughout this section we shall assume that $\GG$ be a $2$-step Carnot group.} Let $S\subset\GG$ be
 a smooth hypersurface (i.e. closed $(n-1)$-dimensional submanifold of $\GG$) and let
\[\delta\cc:\GG\longrightarrow\R^+\cup\{0\}\] denote the {\it CC-distance
function for $S$}, i.e. $\delta\cc(x)=\inf_{y\in S}\dc(x,y)$.

 In the sequel we
shall use the notation $$\mathcal{N}:=\frac{\nu}{|\PH(\nu)|}=(\nn,
\varpi),$$ where $\nu$ is the Riemannian unit normal along $S$. We
stress that the normal (non-unit) vector field $\mathcal{N}$ is
defined at each non-characteristic point $x\in S\setminus C_S$.

 \begin{oss}We would like to remind
  some classical results about the regularity of the distance function to smooth hypersurfaces
  (or, more generally, submanifolds) which hold in the Euclidean -and Riemannian-
  setting. If $N\subset\Rn$ is a $\cont^k$-smooth
  manifold  with $k\geq 2$, then one easily sees that, near $N$, the distance function
  $\delta_N$ is $\cont^{k-1}$-smooth. A first, somewhat surprising result, was proved by Gilbarg
  Trudinger in the Appendix of their celebrated book
  \cite{Gilbarg}. Indeed they prove the existence of
  a neighborhood $U$ of $N$ such that the  distance function $\delta_N$ is of class
  $\cont^k$ on $U\setminus N$. A similar problem was also considered by Federer
  in its theory of Sets of Positive Reach; see \cite{FE1}.
  The complete solution of this problem, in the $\cont^1$ case, can be found in a paper by
   Krantz and Parks; see \cite{Krantz}. We stress that, in such a case,
   a further hypothesis is needed. More precisely, if we assume
   that $N$ is just $\cont^1$-smooth and that there exists a  a neighborhood $U$ of
   $N$ having the so-called ``unique nearest point property'',
   then $\delta_N$ is  of class
  $\cont^1$ on $U\setminus N$. Later on, a simplified proof of this result was proved by
  Foote; see \cite{Foote}.  We would like to remind that, some previous results of this type, in a sub-Riemannian
  setting,  was proved by
  Arcozzi and Ferrari for the first Heisenberg group $\mathbb{H}^1$; see\cite{AF}.
 \end{oss}

So let us begin by assuming that $S$ is  $\mathbf{C}^k$-smooth
with $k\geq 2$. As a first thing, let us define the mapping
$\Phi:S\times]-\epsilon,
\epsilon[\longrightarrow\GG\,(\epsilon>0)$
\begin{equation}\Phi(y,t):=\exp\sr(y,\,\mathcal{N}(y))(t).\end{equation}By construction, it turns
out that $\Phi\in\cont^{k-1} (S\setminus C_S\times ]-\epsilon,
\epsilon[)$ for any - small enough - $\epsilon>0$. The ``size'' of
$\epsilon$ depends on the local geometry \footnote{More precisely,
$\epsilon$ represents the CC-length of a {\it minimizing
CC-geodesic segment} having ``lagrangian multiplier''
$P\cd=\varpi(y)$; see Remark \ref{k}.} of $S$.

Next we shall apply the Inverse Mapping Theorem to this map. For
this reason, we have to compute the differential of $\Phi$ at
$t=0$. To this aim, for any fixed point $y_0\in S$, let us
introduce a system of {\it Riemannian normal coordinates}
\footnote{Let $(\tau_1,...,\tau_{n-1})$ be an orthonormal basis of
$\TT_{y_0}S$. For $1\leq j\leq n-1$ define a real-valued function
$u_j$ on a neighborhood of $y_0$ by
\[u_j\Big((\exp\rr)_{y_0}\Big(\sum_{j=1}^{n-1}t_i\tau_i\Big)\Big)=t_j,\]where
$\exp\rr$ denotes the Riemannian exponential map. Then, by
definition, $(u_1,...,u_{n-1})$ is a system of {\it normal
coordinates} corresponding to the orthonormal basis
$(\tau_1,...,\tau_{n-1})$.} in a neighborhood of $y_0$; see, for
instance, \cite{Ch1} or \cite{Hicks}. Hence $y\equiv
y(u_1,...,u_{n-1})$ and
\[\frac{\partial}{\partial u_1},...,\frac{\partial}{\partial
u_{n-1}}\] are the coordinate vector fields associated with the
normal coordinate system $(u_1,...,u_{n-1})$. Using these
coordinates,it turns out that $\mathcal{V}(y):=\frac{\partial
y}{\partial u_1}\wedge...\wedge\frac{\partial y}{\partial
u_{n-1}}$ is a normal (non-unit) vector along $S$, in a
neighborhood of $y_0\in S$. By definition, the Riemannian unit
normal $\nu$ along $S$ (in this neighborhood of $y_0$) is given,
up the the sign, by $\nu=\frac{\mathcal{V}}{|\mathcal{V}|}$, while
the horizontal unit normal $\nn$ can be written out as
$\nn=\frac{\PH\mathcal{V}}{|\PH\mathcal{V}|};$ see Section
\ref{dere}.

\begin{lemma}\label{CINV}Let
$S\subset\GG$ be  $\mathbf{C}^k$-smooth with $k\geq 2$. Then
\begin{equation}\label{condInv}\big|\det\big[\mathcal{J}_{(y,0)}\Phi\big]\big|=|\PH\mathcal{V}|,\end{equation}where
$\mathcal{J}_{(y,0)}\Phi$ denotes the Jacobian matrix operator at
$(y,0)\in S\times ]-\epsilon, \epsilon[$. Therefore
$\big|\det\big[\mathcal{J}_{(y,0)}\Phi\big]\big|\neq 0$ at each
non-characteristic point of $y\in S\setminus C_S$.
\end{lemma}\begin{proof}Throughout this proof we will need some results of Section \ref{2-stepgeod}.
By \eqref{explicit} we get$$\Phi(y,t)=y +
\int_0^te^{-C\cc(\varpi(y))s}\nn(y)\,d s-\frac{1}{2}
\sum_{\alpha\in I\cd}\Big(\int_{0}^t \langle C^\alpha\cc
x\cc(s),\dot{x}\cc(s)\rangle\,d s\Big)\,\ee_\alpha,$$where
\begin{eqnarray*}x\cc(t)&:=& y\cc+\int_0^te^{-C\cc(\varpi(y))s}\nn(y)\,d
s\\&=&\bigg(..., \underbrace{y_i
+\Big\langle\Big(\int_0^te^{-C\cc(\varpi(y))s}\nn(y)\,d
s\Big),\ee_i\Big\rangle}_{i-th\,place},...\bigg)\in\R^{\DH}.
\end{eqnarray*}To sake of simplicity, we also set
\[\Phi(y, t):= y + \mathcal{A}(y, t).\]By using the explicit
expression of $\Phi$ and the Fundamental Theorem of Calculus we
get\footnote{Since $x(t):=\Phi(y, t)\in\GG$, in exponential
coordinates we have $\Phi=\esp((\Phi\cc, \Phi\cd))\equiv(\Phi\cc,
\Phi\cd)$, where $x\cc=\Phi\cc\in\R^\DH$ and
$x\cd=\Phi\cd\in\R^\vd$; see Notation \ref{Janeiro}.}
\begin{eqnarray*}\frac{\partial\Phi\cc}{\partial
t}&=&e^{-C\cc(\varpi)t}\nn,\\\frac{\partial\Phi\cd}{\partial
t}&=&\frac{1}{2}\sum_{\alpha\in I\cd}\langle C^\alpha\cc
\dot{x}\cc, x\cc\rangle\,\ee_\alpha.\end{eqnarray*}In the last
formula we have also used the skew-symmetry of the structure
constants operators $C\cc^\alpha\, (\alpha\in I\cd)$. Note that
from the previous computations we obtain:\footnote{Remind that we
are working in exponential coordinates.  So we have
\[\nn(y)=\sum_{i\in I\cc}\nn^i(y) X_i(y)=\sum_{i\in
I\cc}\nn^i(y)\underbrace{\bigg(\ee_i-\frac{1}{2}\sum_{\alpha\in
I\cd}\langle C^\alpha\cc y,\ee_i\rangle\bigg)}_{=X_i(y)}.\]
}\begin{eqnarray*}\Big(\frac{\partial\Phi}{\partial
t}\Big|_{t=0}\Big)^{Tr}&=&\sum_{i\in
I\cc}\nn^i(y)\ee_i+\frac{1}{2}\sum_{\alpha\in I\cd}\langle
C^\alpha\cc\nn(y), y\cc\rangle \ee_\alpha\\&=&\sum_{i\in
I\cc}\nn^i(y)X_i(y)\\&=&\nn(y).\end{eqnarray*}

Furthermore, using the system of normal coordinates
$u=(u_1,...,u_{n-1})$, we get that$$\frac{\partial\Phi}{\partial
u_i}=\frac{\partial y}{\partial u_{i}} +
\frac{\partial\mathcal{A}}{\partial u_i}\qquad(i=1,...,n-1).$$As
usual we will set
$\mathcal{A}=\esp((\mathcal{A}\cc,\mathcal{A}\cd))\equiv(\mathcal{A}\cc,\mathcal{A}\cd)$.
We compute
\begin{eqnarray*}\frac{\partial\mathcal{A}\cc}{\partial
u_i}&=&\int_0^t\bigg(\Big[
\mathcal{J}_y \Big(e^{-C\cc(\varpi(y))s}\nn(y)\Big)\Big]\frac{\partial y}{\partial u_i}\bigg)\,d s\in\R^{\DH},\\
\frac{\partial\mathcal{A}\cd}{\partial u_i}&=&-\frac{1}{2}
\sum_{\alpha\in I\cd}\Bigg(\int_{0}^t\Big\langle\grad_y
\big(\langle C^\alpha\cc
x\cc,\dot{x}\cc\rangle\big),\frac{\partial y}{\partial u_i}
\Big\rangle\,d s\Bigg)\ee_\alpha\in\R^\vd\qquad (i=1,...,n-1).
\end{eqnarray*}
Therefore, choosing $t=0$, one gets
\begin{eqnarray*}\mathcal{J}_{(y,0)}\Phi&=&{\rm col}\Big[\frac{\partial y}{\partial
u_1},...,\frac{\partial y}{\partial u_{n-1}},
\frac{\partial\Phi}{\partial t}\Big|_{t=0}\Big]\\&=&{\rm
col}\Big[\frac{\partial y}{\partial u_1},...,\frac{\partial
y}{\partial u_{n-1}}, \nn(y)\Big].\end{eqnarray*}Finally, we may
compute the Jacobian determinant of $\mathcal{J}_{(y,0)}\Phi$. By
using standard Linear Algebra arguments, one gets
\begin{eqnarray*}\big|\det\big[\mathcal{J}_{(y,0)}\Phi\big]\big|&=&\Big|
\det\Big({\rm col}\Big[\frac{\partial y}{\partial
u_1},...,\frac{\partial y}{\partial u_{n-1}},
\nn(y)\Big]\Big|\\&=&\Big|\Big\langle \Big(\frac{\partial
y}{\partial u_1}\wedge...\wedge\frac{\partial y}{\partial
u_{n-1}}\Big),\,\nn(y)\Big\rangle\Big|\\&=&\Big|\frac{\partial
y}{\partial u_1}\wedge...\wedge\frac{\partial y}{\partial
u_{n-1}}\Big||\langle\nu(y),\nn(y)\rangle|\\&=&\Big|\frac{\partial
y}{\partial u_1}\wedge...\wedge\frac{\partial y}{\partial
u_{n-1}}\Big||\PH\nu(y)|\\&=&|\PH\mathcal{V}(y)|,\end{eqnarray*}
which achieves the proof.
\end{proof}

\begin{oss}[Invertibility at the non-characteristic set]\label{inversion}Let us set
$S_0:=S\setminus C_S$. It turns out that $S_0$ is an open subset
of $S$, in the relative topology. Moreover, since we are assuming
that $S$ is $\cont^k$-smooth with $k\geq 2$, one gets that
$\dim\,C_S\leq (n-2)$; see \cite{Mag, Mag2}. Now let $U_0\Subset
S_0$ be an open set compactly contained in $S_0$. By Lemma
\ref{CINV}  we know that the Jacobian of the mapping $\Phi:
U_0\times ]-\epsilon, \epsilon[\longrightarrow\GG$ is non-zero
along $U_0\times\{0\}$. The Inverse Mapping Theorem implies that
there exists $\epsilon_0\in]0,\epsilon]$ such that
 \[\Phi: U_0\times ]-\epsilon_0,
\epsilon_0[\longrightarrow\Phi(U_0\times ]-\epsilon_0,
\epsilon_0[)\] is a $\cont^{k-1}$-diffeomorphism.
\end{oss}

\begin{no}[Projection mapping]The previous Remark \ref{inversion} enables us to define the following
mapping:
\[\Psi:=\Phi^{-1}: \Phi(U_0\times ]-\epsilon_0,
\epsilon_0[)\longrightarrow U_0\times ]-\epsilon_0,
\epsilon_0[.\]By construction, $\Psi$ is $\cont^{k-1}$-smooth. In
the sequel, we shall denote by $\Psi_S$ the projection of the map
$\Psi$ onto its 1st factor, i.e. $\Psi(x)=(\Psi_S(x), t(x))$.
\end{no}
 Let us set $U:=\Phi(U_0\times
]-\epsilon_0, \epsilon_0[)\subset\GG$ and let $x\in U$. The
previous discussion can be summarized by saying that every open
set $U_0$, which is compactly contained in $S_0$, has a
neighborhood $U\subset\GG$ satisfying the {\it unique nearest
point property} \footnote{With respect to the CC-distance $\dc$.},
i.e. for every $x\in U$ there exists a unique point $y\in
U_0\subset S_0$ such that $\delta\cc(x)=\dc(x, y)$. By using the
previous notation, one has $\Psi(x)=(y, t)$, where $y=\Psi_S(x)$
and $t(x)=\dc(x, y)=\delta\cc(x)$.

\begin{teo}\label{ccdreg}Let $\GG$ be a $2$-step Carnot group. Let $S\subset\GG$ be
 a $\cont^k$-smooth hypersurface with $k\geq 2$ and let
$\delta\cc$ denote the CC-distance function for $S$, i.e.
$\delta\cc(x)=\inf_{y\in S}\dc(x,y)$. Set $S_0:=S\setminus C_S$,
where $C_S$ denote the characteristic set of $S$. Then, for every
open set $U_0$ compactly contained in $S_0$, there exists a
neighborhood $U\subset\GG$ of $U_0$ having the {\bf unique nearest
point property} with respect to the CC-distance $\dc$. Finally,
the CC-distance function from $U_0\cap S$ is
$\delta\cc|_{U\setminus U_0}$ is a $\cont^k$-smooth function.

\end{teo} \begin{proof}We just have to prove the last claim. To
this aim, let $X\in\XX(\GG)$ and set $y:=\Psi_S(x)$, where
$\Psi_S(x)$ denotes the projection along $U_0\subset S_0$ of the
point $x\in U:=\Phi(U_0\times]-\epsilon_0, \epsilon_0[)$. Moreover
set $t:=\dc(x, y)$. We have
\begin{eqnarray}\nonumber &&\langle\grad\,\delta\cc(x),
X\rangle\\ \nonumber &=&\langle\grad\, \dc(x, \Psi_S(x)), X\rangle
\\\label{popi}&=&\big\langle\big(\grad_x \dc(x, y)\big)\big|_{y=\Psi_S(x)}, X\big\rangle + \Big\langle
\big[\mathcal{J}_x\Psi_S(x)\big]X,\big(\grad_y \dc(x,
y)\big)\big|_{y=\Psi_S(x)}\Big\rangle.\end{eqnarray}

Now let us introduce the following further notation.

\begin{oss}Let $x\in\GG$, $y\in\GG\setminus\HH_2(x)=\{z=\esp(z\cc, z\cd)\in\GG:
z\cc=x\cc\}$
 and set $t:=\dc(x, y)$.
 Moreover, let \[\widetilde{\nu}(y):=\nu_{\,{\mathbb{S}\sr(x,
t)}}(y)\] denote the Riemannian unit normal along the CC-sphere
${\mathbb{S}\sr(x, t)}$ at any regular point $y\in
{\mathbb{S}\sr(x, t)}$.  We stress that, under our hypotheses, the
point $y$ turns out to be  a regular point of ${\mathbb{S}\sr(x,
t)}$; see, for more details, Remark \ref{k22}. Below we shall set
$$\mathcal{N}_{x, t}(y):=\frac{\widetilde{\nu}(y)}{|\PH(\widetilde{\nu}(y))|}.$$
\end{oss}

 \noindent According to the
results of Section \ref{GLem}, we immediately get that:
\begin{eqnarray*}
\grad_x \dc(x, y)&=&\mathcal{N}_{y, t}(x),\\\grad_y \dc(x,
y)&=&\mathcal{N}_{x, t}(y).\end{eqnarray*}

 \noindent By applying the results of  Section \ref{GLem}, together with the explicit
 form of CC-geodesic for $2$-step Carnot
  groups (see Section \ref{2-stepgeod}),  we get that

  $$\mathcal{N}_{y, t}(x)=\big(e^{-C\cc(\varpi(y)) t}\nn(y),
  \varpi(y)\big).$$Notice that
$\mathcal{N}_{y, t}$ is $\cont^{k-1}$-smooth as well as
$\mathcal{N}=(\nn, \varpi)$. Furthermore, it is easy to see that
$\mathcal{N}_{x, t}(y)=\pm \mathcal{N}(y)$, where  the sign only
depends on the given orientation of $S$.

By the previous discussion and \eqref{popi} we get that

\begin{eqnarray*}\langle\grad\,\delta\cc(x),
X\rangle &=&\big\langle\mathcal{N}_{\Psi_S(x), t}(x), X\big\rangle
+ \big\langle \big[\mathcal{J}_x\Psi_S(x)\big]X,\mathcal{N}_{x,
t}(\Psi_S(x))\big\rangle.\end{eqnarray*} Since
$\big[\mathcal{J}_x\Psi_S(x)\big]X\in\TT_{\Psi_S(x)} S$, by using
the fact that $\mathcal{N}_{x, t}(\Psi_S(x))$ is normal to $S$ at
$\Psi_S(x)$  one gets $$\big\langle
\big[\mathcal{J}_x\Psi_S(x)\big]X,\mathcal{N}_{x,
t}(\Psi_S(x))\big\rangle=0.$$Therefore
$\langle\grad\,\delta\cc(x), X\rangle
=\big\langle\mathcal{N}_{\Psi_S(x), t}(x), X\big\rangle.$ By the
arbitrariness of $X\in\XG$, it follows that $\grad\, \delta\cc$ is
of class $\cont^{k-1}$ on $U\setminus U_0$. Hence$\delta\, \cc$ is
of class $\cont^{k}$ on $U\setminus U_0$. This achieves the proof.
\end{proof}

\vspace{2cm}

{\footnotesize \noindent Nicola Arcozzi:
\\Dipartimento di Matematica, Universit\`a degli Studi di Bologna,\\
  Piazza di P.ta S.Donato, 5, 40126  Bologna, Italia\,
 \\ {\it E-mail address}:  {\textsf arcozzi@dm.unibo.it}}\\

{\footnotesize \noindent Fausto Ferrari:
\\Dipartimento di Matematica, Universit\`a degli Studi di Bologna,\\
  Piazza di P.ta S.Donato, 5, 40126  Bologna, Italia\,
 \\ {\it E-mail address}:  {\textsf ferrari@dm.unibo.it}}}\\

{\footnotesize \noindent Francescopaolo Montefalcone:\\
Dipartimento di Matematica, Universit\`{a} degli Studi di Trento,\\
Povo (Trento)- Via Sommarive, 14 Italia\\
Dipartimento di Matematica, Universit\`a degli Studi di Bologna,\\
  Piazza di P.ta S.Donato, 5, 40126  Bologna, Italia\,
 \\ {\it E-mail address}:  {\textsf montefal@dm.unibo.it}}

\end{document}